\documentclass{amsart}

\usepackage{amsfonts, amssymb, amsmath, eucal, verbatim, amsthm, amscd, enumerate}

\usepackage{graphicx}
\usepackage{framed}
\usepackage{tikz}
\usepackage{enumitem}
\usepackage{mathtools}

\usepackage{ascmac}

\DeclareMathOperator*{\esssup}{ess\,sup}

\newtheorem{theorem}{Theorem}[section]
\newtheorem{lemma}[theorem]{Lemma}
\newtheorem{proposition}[theorem]{Proposition}
\newtheorem{conjecture}[theorem]{Conjecture}

\newtheorem{corollary}[theorem]{Corollary}
\theoremstyle{definition}

\newtheorem{problem}[theorem]{Problem}
\newtheorem{claim}[theorem]{Claim}

\theoremstyle{remark}
\newtheorem*{remark}{Remark}
\newtheorem*{remarks}{Remarks}

\newcommand{\vertiii}[1]{{\left\vert\kern-0.25ex\left\vert\kern-0.25ex\left\vert #1 
\right\vert\kern-0.25ex\right\vert\kern-0.25ex\right\vert}}
\newcommand{\R}{{\mathbb R}}

\numberwithin{equation}{section}

\usepackage{setspace}

\def\1{\textbf{\rm 1}}

\def\XXint#1#2#3{{\setbox0=\hbox{$#1{#2#3}{\int}$}
\vcenter{\hbox{$#2#3$}}\kern-.5\wd0}}

\makeatletter
\@namedef{subjclassname@2020}{\textup{2020} Mathematics Subject Classification}
\makeatother

\begin{document}

\date{\today}


%\author[Bez]{Neal Bez}
%\address[Neal Bez]{Department of Mathematics, Graduate School of Science and Engineering,
%Saitama University, Saitama 338-8570, Japan}
%\email{nealbez@mail.saitama-u.ac.jp}
\author[Nakamura]{Shohei Nakamura}
\address[Shohei Nakamura]{Department of Mathematics, Graduate School of Science, Osaka University, Toyonaka, Osaka 560-0043, Japan}
\email{srmkn@math.sci.osaka-u.ac.jp}
\author[Tsuji]{Hiroshi Tsuji}
\address[Hiroshi Tsuji]{Department of Mathematics, Graduate School of Science, Osaka University, Toyonaka, Osaka 560-0043, Japan}
\email{u302167i@ecs.osaka-u.ac.jp}

%\thanks{This work was supported by JSPS Kakenhi grant numbers 18KK0073, 19H00644 and 19H01796 (Bez), Grant-in-Aid for JSPS Research Fellow no. 17J01766 and JSPS Kakenhi grant numbers 19K03546, 19H01796 and 21K13806 (Nakamura), and JST, ACT-X Grant Number JPMJAX200J, Japan (Tsuji).}

\title[Hypercontractivity and its applications to convex geometry]{Hypercontractivity beyond Nelson's time and its applications to Blaschke--Santal\'{o} inequality and inverse Santal\'{o} inequality}

\begin{abstract}
We explore an interplay between an analysis of diffusion flows such as  Ornstein--Uhlenbeck flow and Fokker--Planck flow and inequalities from convex geometry regarding the volume product. %, the Blaschke--Santal\'{o} inequality and the inverse Santal\'{o} inequality. 
%This link enables us to introduce ideas from the analysis of the diffusion flow into an investigation of the volume product of convex bodies. 
More precisely, we introduce new types of hypercontractivity for the Ornstein--Uhlenbeck flow and clarify how these imply the Blaschke--Santal\'{o} inequality and the inverse Santal\'{o} inequality, also known as Mahler's conjecture. 
Motivated the link, we establish two types of new hypercontractivity in this paper. The first one is an improvement of Borell's reverse hypercontractivity inequality in terms of Nelson's time relation under the restriction that the inputs have an appropriate symmetry. We then prove that it implies the Blaschke--Santal\'{o} inequality. At the same time, it also provides an example of the inverse Brascamp--Lieb inequality due to Barthe--Wolff \cite{BW} beyond their non-degenerate condition. 
The second one is Nelson's forward hypercontractivity inequality with exponents below 1 for the inputs which are log-convex and semi-log-concave. 
This yields new lower bounds of the volume product for convex bodies whose boundaries are well curved. This consequence provides a quantitative result of works by Stancu \cite{Stancu} and Reisner--Sch\"{u}tt--Werner \cite{RSW} where they observed that a convex body with well curved boundary is not a local minimum of the volume product. 
\end{abstract}

\maketitle

\section{Introduction}\label{S1}
Mahler's conjecture expects the minimizer of the volume product of convex bodies with the origin in its interior is attained for the non-degenerate centered simplex and it is essentially unique. If one considers over centrally symmetric convex bodies, then it conjectures that the minimizer is attained for the cube. Despite of the recent resolution of the symmetric case in $\mathbb{R}^3$ by Iriyeh--Shibata \cite{IriShi}  the problem is still open in general.   
Our purpose in this paper is to reveal a new interplay between Mahler's conjecture for general convex bodies, as well as the Blaschke--Santal\'{o} inequality, and an analysis of diffusion flows such as Ornstein--Uhlenbeck flow and Fokker--Planck flow, and then exhibit the strength and wealth of this link.  %by providing new lower bounds of the volume product for a convex body whose boundary is well curved. 
This link consists of two simple observations, the first one is a reinterpretation of the inequality reagarding the volume product as the  inequality of Hamilton--Jacobi flow, the second one is the vanishing viscosity argument employed by Bobkov--Gentil--Ledoux \cite{BGLJMPA} in order to bridge Ornstein--Uhlenbeck flow and Hamilton--Jacobi flow. This link enables us to introduce ideas from a study of these flows into a study of the volume product and vice versa. 
A motivation of this work comes from works by Stancu \cite{Stancu} and Reisner--Sch\"{u}tt--Werner \cite{RSW} where they explored a role of the curvature of the convex body in the context of Mahler's conjecture with the aid of a problem of identifying convex bodies for which the minimum of the volume product is attained.  Briefly speaking, they observed that if the boundary of the convex body $K$ has a point at which the generalized Gauss curvature is strictly positive, then the volume product of $K$ cannot be a local minimum. In other words, the boundary of $K$ needs to be flat if it attains the minimum of the volume product. 
While there are several works focusing on the role of the symmetry of the convex body in Mahler's conjecture \cite{GMR,Mey,Rei,Saint}, their results focus only on the (local) curvature of the boundary and hence are free from the symmetry. 
We will provide quantitative estimates of their results by showing new symmetry free lower bounds of the volume product for convex bodies whose boundaries are uniformly well curved in an appropriate sense.
%The use of these notions is motivated from the work by Schmuckenschl\"{a}ger \cite{Sch} and Klartag--Milman \cite{KlaMil08} where they proved \textit{Hyperplane conjecture} for convex bodies whose uniform convexity and smoothness are appropriately bounded.  In a recent work by Klartag \cite{KlaAdv18}, a link between stronger version of the Hyperplane conjecture and Mahler's conjecture was revealed. 
This will be established by reducing the problem to the analysis of Ornstein--Uhlenbeck flow by virtue of our new link. 
Our main results are two types of new hypercontractivity inequalities. 
The first one is Nelson's time improvement of Borell's reverse hypercontractivity inequality \cite{Borell} under the restriction that the barycenter of the input function is at the origin. We then prove that the  improved reverse hypercontractivity inequality implies the Blaschke--Santal\'{o} inequality. As another viewpoint, the improved reverse hypercontractivity inequality provides an example of the inverse Brascamp--Lieb inequality due to Barthe--Wolff \cite{BW} beyond their non-degenerate condition. 
The second one is Nelson's forward hypercontractivity inequality with exponent less than 1 under the restriction that the inputs are log-convex and semi-log-concave. This in particular yields the aforementioned result on the lower bounds of the volume product for well curved convex bodies.

\subsection{Blaschke--Santal\'{o} inequality and inverse Santal\'{o} inequality}
Let $K \subset \mathbb{R}^n$ be a convex body, namely compact convex set with $\mathrm{int}\, K \neq \emptyset$. 
A polar $K^\circ \subset \mathbb{R}^n$ of $K$ is defined by 
$$
K^\circ \coloneqq \{ y \in \mathbb{R}^n : \langle x, y \rangle \le 1, \, \forall x \in K \},
$$
where $\langle x,y\rangle$ is the standard inner product in $\mathbb{R}^n$. 
In most of the case in this paper, we will concern $K$ with the origin in its interior, $0 \in {\rm int}\, K$ for short, in which case $K^\circ$ becomes a convex body. 
Then the volume product of such convex body $K$ is defined by $v(K) \coloneqq |K| |K^\circ|$ where $|K|$ denotes the volume of $K$ with respect to the Lebesgue measure.  
This volume product is linear invariant in the sense that $v(TK) = v(K)$ for any linear isomorphism $T:\mathbb{R}^n\to \mathbb{R}^n$. 
We say that $K$ is (centrally) symmetric if $K = -K := \{ - x : x \in K\}$.
Then the set of symmetric convex bodies is compact with respect to Banach--Mazur distance and $K \mapsto v(K)$ is continuous. Hence there exist the maximum and minimum of the volume product over all symmetric convex bodies. 
The celebrated Blaschke--Santal\'{o} inequality provides the maximum of $v(K)$ and states that 
\begin{equation}\label{e:BS}
v(K) \le v(\mathrm{B}_2^n)
\end{equation}
for any symmetric convex body $K\subset \R^n$,  where $\mathrm{B}_2^n$ denotes the closed unit ball in the standard $n$-dimensional Euclidean space.  More generally we use a notation $\mathrm{B}_p^n \coloneqq \{x \in \mathbb{R}^n : \sum_{i=1}^n |x_i|^p \le 1\}$ which is an unit $\ell^p$-ball for $1\le p \le \infty$.  Remark that $(\mathrm{B}_p^n)^\circ = \mathrm{B}_{p'}^n$ and hence $v( \mathrm{B}_p^n ) = v(\mathrm{B}_{p'}^n)$. 
The inequality \eqref{e:BS} is first proved by Blaschke \cite{Blaschke} for $n=2,3$ and later Santal\'{o} \cite{Santalo} proved for all dimensions.  
The case of equality for \eqref{e:BS} was addressed by Saint-Raymond \cite{Saint} and it is achieved if and only if $K$ is a symmetric ellipsoid.  
There are alternative and simpler proofs of \eqref{e:BS} and its equality, see \cite{MeyPa,MeyRei,Petty}.   
%For the stability result for \eqref{e:BS} we refer \cite{BBF,Boro09}. 
%It is also worth to mention that \eqref{e:BS} is known to be equivalent to affine isoperimetric inequality, see \cite{Boro09} and references there in. 
Later the assumption for \eqref{e:BS} was weakened to a condition on the barycenter of the convex body. For a convex body $K$, we denote its barycenter  by $b_K \coloneqq \int_K x\, dx \in \mathbb{R}^n$.
Then \eqref{e:BS} holds true for all $K$ with $b_K =0$, see \cite{AKM} for instance. 

Compared to the upper bound of the volume product, the lower bound, which is called the inverse Santal\'{o} inequality, is far from understanding.  In fact, there is a longstanding open problem from 1939 regarding the lower bound which is known as Mahler's conjecture \cite{Mahler}. 
This problem can be formulated for general convex bodies\footnote{Our  formulation of Mahler's conjecture for general convex bodies might not be standard. However, it is indeed equivalent to the  formulation due to Fradelizi--Meyer \cite{FM}, see Appendix \ref{S6.0}. 
}.  

\begin{conjecture}
Let $n \in \mathbb{N}$. 
For any convex body $K \subset \mathbb{R}^n$ with $b_K = 0$, it holds 
\begin{equation}\label{e:MahGene}
v(K) \ge v(\Delta^n_0) = \frac{(n+1)^{n+1}}{(n!)^2}
\end{equation}
where $\Delta^n_0$ is an arbitrary non-degenerate centered simplex, namely $b_{\Delta^n_0}=0$. 

For any centrally symmetric convex body $K \subset \mathbb{R}^n$, it holds 
\begin{equation}\label{e:Mah}
v(K) \ge v( \mathrm{B}_\infty^n) = \frac{4^n}{n!}.
\end{equation}
\end{conjecture}
Motivated from the study of the geometry of number, Mahler \cite{Mahler} realized inequalities \eqref{e:MahGene} and \eqref{e:Mah} and proved it for $n=2$ by himself.  
It is very recent that the problem for $n=3$, symmetric case, was settled down with affirmative answer by Iriyeh--Shibata \cite{IriShi}, see \cite{FHMRZ} for the short proof.  
The problem in general case is open for $n\ge3$ and the problem in symmetric case is open for $n\ge4$ despite of several contributions. For instance \eqref{e:Mah} was proved for unconditional convex bodies \cite{Saint,Mey}, zonoids \cite{Rei,GMR},  see also \cite{BF,FMZ,Karasev} for other special cases. 
For general convex bodies, the inverse Santal\'{o} inequality was proved up to some multiplicative constant by Bourgain--Milman \cite{BouMil} and later the multiplicative constant was improved by Kuperberg \cite{Kuper}, see also \cite{Bernd,MaRu,Naza}. 

Regarding the case of equality, in general case, it is expected that $\Delta^n_0$ is the unique extremizer to \eqref{e:MahGene} up to linear isomorphism, see \cite{FM,Mey91}. 
The case of equality in symmetric case is a different story. We remark that $v( \mathrm{B}_\infty^n) = v( \mathrm{B}_1^n)$ and hence linear isomorphism images of the standard cube $\mathrm{B}^n_\infty$ and its polar $\mathrm{B}^n_1$ achieve equality in \eqref{e:Mah}.  In fact, it was proved that these are the only possibility of equality when $n=2,3$, see \cite{FHMRZ,IriShi,Mey91}.  However, unlikely the general case, there are other convex bodies for equality if $n\ge4$ such as Hanner polytopes.  
We mention the work by Nazarov--Petrov--Ryabogin--Zvavitch \cite{NPRZ} where they observed that $\mathrm{B}^n_\infty$,  $\mathrm{B}^n_1$ and Hanner polytopes are local minimizers of the volume product, see also \cite{Kim}. 
So as far as we are aware, the problem of identifying the case of equality for the inverse Santal\'{o} inequality is also open for $n\ge3$ in general and for $n\ge4$ in the symmetric case. In this direction we refer to works of Stancu \cite{Stancu} and Reisner--Sch\"{u}tt--Werner \cite{RSW} where they observed that the boundary of the convex body needs to be flat in order to attain equality in \eqref{e:MahGene} or \eqref{e:Mah}. 
\begin{theorem}[\cite{Stancu,RSW}]\label{t:RSW}
Let $n\in\mathbb{N}$ and $K \subset \mathbb{R}^n$ be a convex body with $0 \in {\rm int}\, K$. 
If there is a point in either $\partial K$ or $\partial K^\circ$ at which the generalized Gauss curvature exists and is not 0, then $v(K)$ is not a local minimum. 
\end{theorem}
More precisely Stancu \cite[Corollary 3.4]{Stancu} first observed that if a symmetric convex body $K$ is of class $C^2$ and $\partial K$ has strict positive Gauss curvature everywhere, then $v(K)$ is not a minimal among  symmetric convex bodies. 
This result was further improved by Reisner--Sch\"{u}tt--Werner \cite[Theorem 2.1]{RSW} as in Theorem \ref{t:RSW}.  
In fact they observed that if $\partial K$ has a curved point then one can slice $K$ around the point to make it flat and decrease the volume product. Same is true if $\partial K^\circ$ has a curved point since $v(K) = v(K^\circ)$.  
This observation raises a question of how curvatures of $\partial K$ and $\partial K^\circ$ improve the inverse Santal\'{o} inequality\footnote{We will assume both of $K$ and $K^\circ$ are well curved to give a quantitative result of Theorem \ref{t:RSW}. This also matches to the spirit of the work by \'{A}lvarez Paiva--Balacheff--Tzanev \cite{ABT}; \textit{Never consider a convex body without considering its dual at the same time} [I. M. Gelfand]. }. This is a starting point of our research. 
Namely our aim is to clarify the role of curvatures of $\partial K$ and $\partial K^\circ$ in the context of the inverse Santal\'{o} inequality in a quantitative way.  
%While such plausible indication of the link by Stancu and Reisner--Sch\"{u}tt--Werner there are two critical distinction. Firstly the volume product, and hence the inverse Santal\'{o} inequality, is linearly invariant as $v(TK) = v(K)$ for any linear isomorphism $T$ while the standard Gaussian curvature is not.  Secondly the inverse Santal\'{o} inequality is of course global inequality while the curvature is local information. 
%To overcome such distinction,  we appeal to notions of uniform convexity and its dual property, uniform smoothness. 

For a convex body $K \subset \mathbb{R}^n$ with $0 \in {\rm int}\, K$ the gauge function (or asymmetric  norm) of $K$ is defined by 
$$
\| x \|_K\coloneqq  \inf \{r>0 :\, x \in rK\},\;\;\; x\in \mathbb{R}^n. 
$$
From the definition $K$ is a closed unit ball with respect to $\| \cdot\|_K$ and $\|\cdot\|_K$ becomes a norm if $K$ is symmetric.  Conversely, provided a norm without symmetry first,  one can define a convex body with the origin in its interior from it.  Hence we do not distinguish convex bodies and gauge functions in below.  
Throughout this paper we denote the standard Euclidean norm by $|\cdot| = \| \cdot\|_{{\rm B}_2^n}$.  
%We are interested in convex bodies whose boundaries are well curved and particularly smooth or at least $C^2$. 
We measure the curvature of the boundary of the convex body by quantifying the convexity of its gauge function. 
To this end a special class of gauge functions so-called Minkowski norms is useful. A gauge function $\|\cdot\|$ is said to be the Minkowski norm if $x\mapsto \| x\|^2$ is smooth on $\mathbb{R}^n\setminus\{0\}$ and its Hessian is positive definite. We refer to the book \cite{BCS} and the work by \'{A}lvarez Paiva--Balacheff--Tzanev \cite{ABT} for the relevancy of Minkowski norms in Finsler geometry and a link between Mahler's conjecture and Finsler geometry.   
As is mentioned in \cite{ABT}, the unit sphere of the Minkowski norm is a quadratically convex hypersurface in the sense that the osculating quadratic at each points is an ellipsoid and hence this notion is useful for our purpose. 
%With this in mind we measure the curvature of $\partial K$ by quantifying 
Our curvature condition quantifies the definition of the Minkowski norm and  takes the following form:  
\begin{equation}\label{e:CondCurvature}
\nabla^2 \big( \frac12 \| \cdot\|_K^2 \big)(x),\; \nabla^2 \big( \frac12 \| \cdot\|_{K^\circ}^2 \big)(x) \ge \kappa {\rm id}, \;\;\; \forall x \in \mathbb{S}^{n-1}\coloneqq \partial {\rm B}^n_2
\end{equation}
for some $\kappa \in (0,1]$. 
Or more generally, in view of the linear invariance of the volume product, we may consider convex bodies $K$ whose gauge function satisfies\footnote{Note that the condition \eqref{e:CondCurvNov} yields a global lower bound on the principal curvatures, and hence Gaussian curvature too, of $\partial K$ and $\partial K^\circ$, see Appendix \ref{S6.6}.} 
\begin{equation}\label{e:CondCurvNov}
\nabla^2\big( \frac12 \|\cdot\|_K^2)(x) \ge \kappa \Lambda^{-1},\;\;\; 
\nabla^2\big( \frac12 \|\cdot\|_{K^\circ}^2)(x) \ge \kappa \Lambda,
\;\;\;
\forall x \in \mathbb{S}^{n-1}
\end{equation}
for some $\kappa \in (0,1]$ and some positive definite\footnote{We say a symmetric matrix $A$ is positive definite (semi-definite) if the corresponding quadratic form is positive definite (semi-definite). For two positive definite matrix $A,B$ we denote $A\le B$ if $B-A$ is positive semi-definite.}
symmetric matrix $\Lambda$. 
%While the lower bound of $\nabla^2 \big( \frac12 \| \cdot\|_K^2 \big)$ in \eqref{e:CondCurvature} represents the well curvedness of $\partial K$ in a quantitative way,  the upper bound ensures the well curvedness of $\partial K^\circ$. 
%We will present concrete relation between our condition \eqref{e:CondCurvature} and the principle curvature  of $\partial K$ in forthcoming Proposition \ref{Prop:Curv} although it might be folklore. 
We will obtain a family of lower bounds of weighted and unweighted volume products under the curvature assumption \eqref{e:CondCurvNov} in Theorem \ref{t:RegInvSanSet} and Corollary \ref{t:WeightInv}. Here we exhibit the case of the unweighted volume product.

\begin{theorem}\label{t:ImpInv}
Let $n \ge 2$ and $\kappa\in (0,1]$. 
Then for any convex body $K\subset \mathbb{R}^n$ with $0 \in \mathrm{int}\, K$ satisfying \eqref{e:CondCurvNov} for some positive definite $\Lambda$, 
we have 
\begin{equation}\label{e:OurMahler}
v(K)
\ge 
\big( \kappa^2 e^{1-\kappa^2} \big)^\frac{n}{2} 
v( {\rm B}_2^n).
\end{equation}
\end{theorem}
Examples of convex bodies satisfying \eqref{e:CondCurvNov} can be obtained from a regularization of the convex body by 2-Firey intersection with ${\rm B}_2^n$. 
%To introduce the notion we take an arbitrary convex body $K_0$ with $0 \in {\rm int}\, K_0$ whose gauge function is smooth.  Then let $K_\lambda$, $\lambda>0$,  be 2-Firey intersection of $K_0$ and ${\rm B}_2^n$, namely a convex body whose gauge function is given by 
%$$
%\| x \|_{K_\lambda}^2
%\coloneqq 
%\| x \|_{K_0}^2 
%+ 
%\lambda 
%|x|^2.
%$$
%A family of convex bodies $(K_\lambda)_{\lambda}$ interpolates arbitrary convex body $K_0$ and Euclidean ball ${\rm B}_2^n$. 
%Such regularization of convex body can be found in \cite[Section 1.3]{BCS}. 
%As is clear from the definition this regularization makes the boundary of $K$ and $K^\circ$ well curved in the sense of \eqref{e:CondCurvature} as $\lambda$ increases even if either $\partial K_0$ or $\partial K_0^\circ$ has a flat point.  
%More quantitatively speaking, if we denote $S_{K_0}' \coloneqq \max_{x \in \partial K_0} \mu_{K_0}(x)$, where $\mu_{K_0}(x)$ is the maximum of eigenvalue of $\nabla^2 \big(\frac12 \|\cdot\|_{K_0}^2(x)$, then $K_\lambda$
Let $K_0$ be any convex body with $0 \in {\rm int}\, K_0$ such that $\|\cdot\|_{K_0}^2$ is $C^2$ on $\mathbb{R}^n\setminus\{0\}$. 
For $\lambda>0$, define the regularization of $K_0$, denoted by $K_\lambda$,  to be the 2-Firey intersection of $K_0$ and $\lambda^{-\frac12}{\rm B}_2^n$, namely it is a convex body whose gauge function is given by 
$$
\| x \|_{K_\lambda}^2
\coloneqq 
\| x \|_{K_0}^2 
+ 
\lambda 
|x|^2.
$$
Then $K_\lambda$ satisfies \eqref{e:CondCurvNov} with $\Lambda = (\lambda(\lambda + M_{K_0}))^{-\frac12} {\rm id}$ and $\kappa^2 = \frac{\lambda}{\lambda+ M_{K_0}}$, 
%\begin{equation}\label{e:Example20Oct}
%\lambda{\rm id} \le \nabla^2 \big( \frac12 \|\cdot\|_{K_\lambda}^2 \big) \le (\lambda + S_{K_0}') {\rm id}
%\end{equation}
where $M_{K_0} \coloneqq \max_{x \in \mathbb{S}^{n-1}} \mu_{K_0}(x)$ and $\mu_{K_0}(x)$ is the maximum  eigenvalue of $\nabla^2 \big(\frac12 \|\cdot\|_{K_0}^2\big)(x)$.
Hence we may apply Theorem \ref{t:ImpInv} to $K_\lambda$ to see that  
$$
v(K_\lambda )
\ge 
\big( 
\frac{\lambda}{ \lambda + M_{K_0}}
{\rm exp}\big[ \frac{ M_{K_0}}{ \lambda+ M_{K_0} }  \big] 
\big)^\frac{n}{2} 
v( {\rm B}_2^n)
$$
for any $\lambda>0$.  
Such regularization of convex bodies can be found in \cite[Section 1.3]{BCS}, see also \cite{KlaMil08} for an investigation of convex bodies given by the 2-Firey intersection. 
We will derive \eqref{e:OurMahler} as a consequence from new hypercontractivity for Ornstein--Uhlenbeck flow. 
To this end let us give a brief overview of hypercontractivity next.

\subsection{Hypercontractivity: Link to the Blaschke--Santal\'{o} inequality and the inverse Santal\'{o} inequality}
The Ornstein--Uhlenbeck flow $P_s$, $s>0$,  is defined by 
$$
P_s f(x) \coloneqq \int_{\mathbb{R}^n} f(e^{-s}x + \sqrt{1-e^{-2s}} y)\, d\gamma(y),\;\;\; x \in \mathbb{R}^n
$$
for nonnegative $f\in L^1(\gamma)$ where $\gamma$ denotes the standard Gaussian measure $\gamma(x) = (2\pi)^{-\frac{n}{2}} e^{- \frac12|x|^2}$.  
If we let $u(s,x) = P_sf(x)$,  this solves the heat equation 
$$
\partial_s u = \Delta u - x\cdot \nabla u,\;\;\; u(0) = f. 
$$
The Nelson's fundamental hypercontractivity inequality \cite{Nelson},  which quantifies the regularizing property of Ornstein--Uhlenbeck flow,   states that %\footnote{This is equivalent to a statement that the inequality 
%$$
%\| P_s f \|_{L^q(\gamma)} \le  \| f \|_{L^p(\gamma)}
%$$
%holds for all nonnegative $f\in L^p(\gamma)$. 
%} 
\begin{equation}\label{e:ForwardHC}
\big\| P_s \big[ f^\frac1p \big] \big\|_{L^q(\gamma)} \le \big( \int_{\mathbb{R}^n} f\, d\gamma \big)^\frac1p
\end{equation}
for all nonnegative $f\in L^1(\gamma)$ if $1 < p,q <\infty$ and $s>0$ satisfy 
\begin{equation}\label{e:NelsonTime}
\frac{q-1}{p-1} \le e^{2s}.  
\end{equation}
The reverse form of \eqref{e:ForwardHC} is also known as Borell's reverse hypercontractivity inequality \cite{Borell} and states that 
\begin{equation}\label{e:ReverseHC}
\big\| P_s \big[ f^\frac1p \big] \big\|_{L^q(\gamma)} \ge \big( \int_{\mathbb{R}^n} f\, d\gamma \big)^\frac1p
\end{equation}
for all nonnegative $f\in L^1(\gamma)$ if $ p,q \in (-\infty,1)\setminus \{0\}$ and $s>0$ satisfy \eqref{e:NelsonTime}. 
The latter one quantifies the positivity improving property of $P_s$ in Borell's sense.
%\footnote{If one could show \eqref{e:ReverseHC} for some $q<0$ whose absolute value is large, then it manifests that $P_sf$ would have absolute positive lower bound on a large set. 
%{\color{red}One can say something like 
%$$
%\gamma( \{ x \in \mathbb{R}^n : P_s f > \frac1{100} \} ) \gg 0??
%$$
%Link to Talagrands's conjecture?}
%}.

%Both forward and reverse hypercontractivity are known to be sharp in terms of Nelson's time relation.  
The condition \eqref{e:NelsonTime} is known as \textit{Nelson's time relation} and provides the sharp threshold for the validity of \eqref{e:ForwardHC} and \eqref{e:ReverseHC}.  
Namely if one has \eqref{e:ForwardHC} for some $p,q \in (1,\infty), s>0$ or \eqref{e:ReverseHC} for some $p,q\in (-\infty,1)\setminus\{0\}, s>0$, then these exponents must satisfy \eqref{e:NelsonTime}. This can be checked with the translated Gaussian $f = \gamma(\cdot + a)/\gamma$, $|a|\to\infty$.  
Moreover this example shows that \eqref{e:ForwardHC} and \eqref{e:ReverseHC} dramatically fail to hold with \textit{any finite constant} if $\frac{q-1}{p-1} > e^{2s}$.  
%Also if one has \eqref{e:ForwardHC} then $p,q\ge1$ and if one has \eqref{e:ReverseHC} then $p,q \le 1$.  {\color{red} IS THIS TRUE IN THIS GENERALITY? $\to$ NOT ALWAYS TRUE! }
Also constants in \eqref{e:ForwardHC} and \eqref{e:ReverseHC} are sharp as equalities are attained for constant functions. We refer \cite{BraLi_Adv,LedouxJFA92} for the problem of identifying the case of equalities. 
%We are interested in the sharpness of the first kind. 
In the context of differential geometry and probability theory, hypercontractivity plays an important role by virtue of its equivalence to the logarithmic Sobolev inequality. This equivalence was discovered by Gross \cite{Gross}. 
We refer readers to the survey paper \cite{DGS} and the book \cite{Faris} for further detailed historical background and broad picture of the vast connections of hypercontractivity to several branches of the mathematical sciences.  
We also mention seminal works due to Mossel--O'Donnell--Regev--Steif--Sudakov \cite{MORSS} and Mossel--Oleszkiewicz--Sen \cite{MOS} for the link between the analysis of mixing of short random walks on the discrete cube, as well as the Non-Interactive Correlation Distillation problem, and reverse hypercontractivity. 

While there is no study of exploring an explicit link between hypecontractivity, the Blaschke--Santal\'{o} inequality and the inverse Santal\'{o} inequality as far as we are aware, there are few works indicating such relations. 
In this direction, we refer to works due to Caglar--Fradelizi--Gu\'{e}don--Lehec--Sch\"{u}tt--Werner \cite{CFGLSW}, Fathi \cite{Fathi},  and Gozlan \cite{Goz}.  In \cite{CFGLSW}, it was proved that the functional form of \eqref{e:BS}, see also \eqref{e:Lehec2},  implies the inverse form of the logarithmic Sobolev inequality due to Artstein-Avidan--Klartag--Sch\"{u}tt--Werner \cite{AKSW}.  
It was then recently observed by Fathi \cite{Fathi} that the reverse implication is also true and hence the functional form of \eqref{e:BS} and the inverse logarithmic Sobolev inequality are equivalent. 
In \cite{Goz}, Gozlan figured out some equivalence between the functional forms of \eqref{e:MahGene} and \eqref{e:Mah}, see also \eqref{e:FuncInvSantaloGene}, \eqref{e:FuncInvSantalo}, and some improved  logarithmic Sobolev inequality. 
%The classical hypercontractivity and the classical logarithmic Sobolev inequality are known to be equivalent by Gross's argument and so it is natural to expect some relation to hypercontractivity too. 
%On the other hand,  
Despite of the classical equivalence between the logarithmic Sobolev inequality and hypercontractivity, it is not immediately clear how one can apply the Gross's classical argument to the inverse logarithmic Sobolev inequality or Gozlan's improved logarithmic Sobolev inequality in order to derive some reasonable hypercontractivity. 
Instead of such route we observe more direct link as follows.  We say a function $f$ is centrally symmetric if $f(-\cdot) =f$. 

\begin{proposition}\label{Prop:HCSet}
Let $n \in \mathbb{N}$. 
	\begin{enumerate}
		\item 
		Suppose that for all sufficiently small $s>0$, there exist $q_s<0$ with order $q_s = -2s + O(s^2)$ and constant ${\rm BS}_{s}>0$ such that 
		\begin{equation}\label{e:RHC_BSIntro}
		\big\| 
		P_s \big[ f^{\frac1{p_s}} \big]
		\big\|_{L^{q_s}(\gamma)}
		\ge 
		{\rm BS}_{s}^\frac{1}{p_s} 
		\big(
		\int_{\mathbb{R}^n}
		f\, d\gamma
		\big)^\frac1{p_s},
		\;\;\;
		p_s \coloneqq 1-e^{-2s} 
		\end{equation}
		holds for all centrally symmetric and nonnegative $f\in L^1(\gamma)$. 
		Then for any centrally symmetric convex body $K$, 
		\begin{equation}\label{e:SetBSIntro}
			v(K) 
			\le 
			(\limsup_{s\downarrow0} {\rm BS}_{s})^{-1}
			v({\rm B}_2^n).
		\end{equation}
		\item 
		Suppose that for all sufficiently small $s>0$, there exist $p_s \in (0,1)$ with order $p_s = 2s + O(s^2)$ and constant ${\rm IS}_{s}>0$ such that 
		\begin{equation}\label{e:FHC_MaIntro}
		\big\| 
		P_s \big[ f^{\frac1{p_s}} \big]
		\big\|_{L^{q_s}(\gamma)}
		\le 
		{\rm IS}_{s}^\frac{1}{p_s} 
		\big(
		\int_{\mathbb{R}^n}
		f\, d\gamma
		\big)^\frac1{p_s},
		\;\;\;
		q_s\coloneqq 1-e^{2s}
		\end{equation}
		holds for all nonnegative $f \in L^1(\gamma)$. Then for any convex body $K$ with $0 \in {\rm int}\, K$,  
		\begin{equation}\label{e:SetMaIntro}
			v(K)
			\ge 
			(\liminf_{s\downarrow0} {\rm IS}_{s})^{-1}
			v({\rm B}_2^n). 
		\end{equation}
	\end{enumerate}
\end{proposition}

\if0
\begin{proposition}\label{Prop:HCSet}
Let $n \in \mathbb{N}$ and $K \subset \mathbb{R}^n$ be sufficiently regular convex body with $0 \in {\rm int}\, K$. 
Define a positive function $f  \in L^1(\gamma)$ by $\log\, f (x) = \frac12|x|^2 - \frac12 \| x \|_K^2$. 
	\begin{enumerate}
		\item 
		Suppose that 
		\begin{equation}\label{e:RHC_BSIntro}
		\big\| 
		P_s \big[ f_s^{\frac1{p_s}} \big]
		\big\|_{L^{q_s}(\gamma)}
		\ge 
		{\rm BS}_{s}^\frac{1}{p_s} 
		\big(
		\int_{\mathbb{R}^n}
		f_s\, d\gamma
		\big)^\frac1{p_s},
		\;\;\;
		p_s \coloneqq 1-e^{-2s},\; q_s \coloneqq -p_s
		\end{equation}
		holds for some constant ${\rm BS}_{s}>0$ and all small $s>0$ where $f_s\coloneqq f^{\frac{p_s}{2s}}$. 
		Then 
		\begin{equation}\label{e:SetBSIntro}
			v(K) 
			\le 
			\limsup_{s\downarrow0} {\rm BS}_{s}^{-1}
			v({\rm B}_2^n).
		\end{equation}
		\item 
		Suppose that 
		\begin{equation}\label{e:FHC_MaIntro}
		\big\| 
		P_s \big[ f^{\frac1{p_s}} \big]
		\big\|_{L^{q_s}(\gamma)}
		\le 
		{\rm IS}_{s}^\frac{1}{p_s} 
		\big(
		\int_{\mathbb{R}^n}
		f\, d\gamma
		\big)^\frac1{p_s},
		\;\;\;
		p_s\coloneqq 1-e^{-2s},\; q_s\coloneqq 1-e^{2s}
		\end{equation}
		for some constant ${\rm IS}_{s}>0$ and all small $s>0$. Then 
		\begin{equation}\label{e:SetMaIntro}
			v(K)
			\ge 
			\liminf_{s\downarrow0} {\rm IS}_{s}^{-1}
			v({\rm B}_2^n). 
		\end{equation}
	\end{enumerate}
\end{proposition}
\fi
Few remarks on this proposition are in order. 
\begin{remarks}
\begin{enumerate}
	\item Typical examples of exponents are $q_s = -2s$ for \eqref{e:RHC_BSIntro} and $p_s = 2s$ for \eqref{e:FHC_MaIntro} but we stated in more flexible way in the above. In fact, we will prove \eqref{e:RHC_BSIntro} with ${\rm BS}_s =1$ for $q_s = -1 + e^{-2s}$ in forthcoming Theorem \ref{t:RevHCIntro}. On the other hand, the most challenging cases would be $q_s = 1-e^{2s}$ for \eqref{e:RHC_BSIntro} and $p_s = 1-e^{-2s}$ for \eqref{e:FHC_MaIntro}, see forthcoming discussion \eqref{e:Admissible} and \eqref{e:Admissible2} for more details. 
	\item  
	We stated the above proposition in a way that \eqref{e:RHC_BSIntro} for \textit{all} centrally symmetric functions implies \eqref{e:SetBSIntro}  for \textit{all} centrally symmetric convex bodies and similarly  \eqref{e:FHC_MaIntro} for \textit{all} functions implies  \eqref{e:SetMaIntro} for \textit{all} convex bodies. 
We will indeed prove more tight relation. That is, for a given convex body $K$ with $0\in {\rm int}\, K$, if one could prove  \eqref{e:RHC_BSIntro} / \eqref{e:FHC_MaIntro} for $f = f_s$ defined by $ \log\, f_s(x) = \frac{p_s}{2s}(  \frac12|x|^2  - \frac12\|x\|_K^2 )$ for all small $s>0$, then it implies \eqref{e:SetBSIntro} / \eqref{e:SetMaIntro} for the convex body $K$, see forthcoming Propositions \ref{Prop:HCBS_2} and \ref{Prop:HCMahler}.  

\end{enumerate}
\end{remarks}
Here we emphasize that both of \eqref{e:RHC_BSIntro} and \eqref{e:FHC_MaIntro} are beyond the scope of the classical hypercontractivity inequalities \eqref{e:ForwardHC} and \eqref{e:ReverseHC} as the relevant exponents\footnote{For \eqref{e:RHC_BSIntro}, if $p_s = 1-e^{-2s}$, $q$ and $s>0$ satisfy $\frac{q-1}{p_s-1}\le e^{2s}$ then the $q$ must be $q\ge0$ while $q_s = -2s + O(s^2)<0$. 
For \eqref{e:FHC_MaIntro}, if $p$, $q_s = 1-e^{2s}$ and $s>0$ satisfy $\frac{q_s-1}{p-1}\le e^{2s}$ then $p$ must be $p\le0$ while $p_s = 2s + O(s^2) >0$.
} do not satisfy \eqref{e:NelsonTime}. 
In other words, \textit{one needs to establish new hypercontractivity beyond Nelson's time relation \eqref{e:NelsonTime} in order to build a bridge between hypercontractivity and convex geometry}.  
A study of hypercontractivity beyond Nelson's time can be found in the work due to Janson \cite{Janson} for instance where it was observed that the forward hypercontractivity inequality \eqref{e:ForwardHC} holds for $p,q>1$ and $s>0$ satisfying $\frac{q}{p}\le e^{2s}$ and $n \in 2\mathbb{N}$ as long as $f$ is holomorphic, see also \cite{GKL}. 

\subsection{Brief discussion on new hypercontractivity from a view point of  the Brascamp--Lieb theory}\label{S2.1}
Before exhibiting our main results on new hypercontractivity, let us give a brief discussion on \eqref{e:RHC_BSIntro} and \eqref{e:FHC_MaIntro} from a view point of the Brascamp--Lieb theory, which in particular reveals a connection to the work by Barthe--Wolff \cite{BW} on the inverse Brascamp--Lieb inequality.  
%If one regards these two inequalities as functional inequalities,  there is no reason to restrict $f$ to the form of $\log\, f(x) = \frac12|x|^2 - \frac12 \|x\|_K^2$. Thus let us discuss \eqref{e:RHC_BSIntro} and \eqref{e:FHC_MaIntro} for general nonnegative $f \in L^1(\gamma)$.  
%Firstly if one regards inequalities as functional inequalities, then there is no apriori reason to restrict centrally symmetric $f$ and thus we discuss \eqref{e:RHC_BSIntro} and \eqref{e:FHC_MaIntro} for general nonnegative $f \in L^1(\gamma)$.  
Our investigation here is based on a ``centered Gaussian exhaustion principle" in the context of the Brascamp--Lieb inequality by Brascamp--Lieb  \cite{BraLi_Adv} and Lieb \cite{Lieb} as follows.   
\begin{theorem}[\cite{BraLi_Adv}, \cite{Lieb}]\label{t:Lieb}
Let $m, d, d_1,\ldots,d_m\in \mathbb{N}$. We take $c_j \in [0,1]$,  a linear surjective map $L_j:\mathbb{R}^d \to \mathbb{R}^{d_j}$ for each $j =1,\ldots, m$, and a positive semi-definite matrix on $\mathbb{R}^d$ denoted by $\mathcal{Q}$. 
Define a Brascamp--Lieb functional by 
\begin{equation}\label{e:BLFunc}
{\rm BL}(f_1,\ldots,f_m) 
\coloneqq 
\frac{\int_{\mathbb{R}^d} e^{-\pi \langle x, \mathcal{Q}x\rangle} \prod_{j=1}^m f_j \circ L_j^{c_j}\, dx}{\prod_{j=1}^m \big(\int_{\mathbb{R}^{d_j}} f_j\, dx_j \big)^{c_j}}
\end{equation}
for nonnegative $f_j \in L^1(dx_j)$ such that $\int_{\mathbb{R}^{d_j}} f_j \, dx_j >0$. 
Then the supremum ${\rm BL}(f_1,\ldots,f_m)$ over all such functions is equal to its supremum over centered Gaussian functions. 
\end{theorem}
The relevant feature of this theorem to us is not the form or generality of the inequality but the simple principle that one has only to investigate centered Gaussian functions rather than general inputs to identify the best constant. 
This simple principle has been providing an important source of developments for broad scientific fields including analytic number theory, harmonic analysis, probability and information theory, and statistical physics, see  \cite{Barthe1,BWAnn,BW,BBJFA,BBFL,BBBCF,BCCT,BCT,BN,BDG,CC,CLL,CDP} and references therein.  
The classical forward hypercontractivity inequality \eqref{e:ForwardHC} is also known to be a member of the family of the Brascamp--Lieb inequality. 
In fact the duality argument shows that \eqref{e:ForwardHC} with $p,q>1$ is equivalent to the inequality 
\begin{equation}\label{e:HCBL}
\int_{\mathbb{R}^{2n}} e^{-\pi \langle x, \mathcal{Q}x\rangle} \prod_{j=1,2} f_j(x_j)^{c_j} \, dx
\le 
%{\rm H}(c_1,c_2) 
\big( \frac{\sqrt{1-e^{-2s}}}{(2\pi)^{\frac12(c_1+c_2) -1}} \big)^n
\prod_{j=1,2} \big( \int_{\mathbb{R}^n} f_j\, dx_j \big)^{c_j} 
\end{equation}
for all nonnegative $f_1,f_2 \in L^1(dx)$ with 
\begin{align}\label{e:HCBLLink}
&\mathcal{Q} \coloneqq \frac1{2\pi (1-e^{-2s})} 
\begin{pmatrix}
\big( 1-(1-e^{-2s})\frac1p \big) {\rm id}_{\mathbb{R}^n} & -e^{-s} {\rm id}_{\mathbb{R}^n} \\
-e^{-s} {\rm id}_{\mathbb{R}^n}& \big( 1-(1-e^{-2s})\frac1{q'} \big) {\rm id}_{\mathbb{R}^n}
\end{pmatrix},\\
&c_1 \coloneqq \frac1p,
\;\;\; 
c_2\coloneqq \frac1{q'}, \nonumber 
\end{align}
see \cite{BNT,Lieb} and Appendix \ref{S.6HC}.
It is worth to remark that the assumption $p,q>1$ ensures $c_1,c_2 \in (0,1)$ and Nelson's time relation \eqref{e:NelsonTime} ensures that $\mathcal{Q}$ is positive semi-definite. 
Similarly one would be able to derive the new forward hypercontractivity inequality \eqref{e:FHC_MaIntro} once one could prove \eqref{e:HCBL}. However, in this case,  $c_1,c_2 \notin (0,1)$ and $\mathcal{Q}$ is no longer positive semi-definite.  As far as we know there is no (forward) Brascamp--Lieb theory with $c_j \in \mathbb{R}$ and $\mathcal{Q}$ which is not necessarily positive semi-definite. 

It is natural to ask an inverse analogue of Theorem \ref{t:Lieb} which particularly contains Borell's reverse hypercontractivity inequality \eqref{e:ReverseHC}.  This was recently established by Barthe--Wolff \cite{BW}. 
For the inverse case, one needs something more on the set up.  In fact,  one considers the Brascamp--Lieb functional \eqref{e:BLFunc} for any $c_1,\ldots,c_m \in \mathbb{R}$ and any self-adjoint matrix $\mathcal{Q}$.  We then order $(c_j)_j$ so that $c_1,\ldots, c_{m_+} > 0 > c_{m_++1},\ldots,c_m$ for some $0\le m_+ \le m$. Correspondingly, we let ${\bf L}_+:\mathbb{R}^d\to \prod_{j=1}^{m_+} \mathbb{R}^{d_j}$ by 
$$
{\bf L}_+x \coloneqq (L_1x,\ldots, L_{m_+}x). 
$$
Then Barthe--Wolff's non-degenerate condition states that 
\begin{equation}\label{e:NondegBW}
\mathcal{Q}\big|_{{\rm Ker}\, {\bf L}_+} >0 
\;\;\;
{\rm and}
\;\;\;
d \ge s^+(\mathcal{Q}) + \sum_{j=1}^{m_+} d_j, 
\end{equation}
where $\mathcal{Q}\big|_{{\rm Ker}\, {\bf L}_+}$ is the restriction of $\mathcal{Q}$ to the subspace ${\rm Ker}\, {\bf L}_+$ and $s^+(\mathcal{Q})$ is the number of positive eigenvalues of $\mathcal{Q}$. 
\begin{theorem}[\cite{BW}]\label{t:BW}
Suppose \eqref{e:NondegBW}. Then the infimum of ${\rm BL}(f_1,\ldots,f_m)$ over all nonnegative $f_j \in L^1(dx_j)$ such that $\int_{\mathbb{R}^{d_j}} f_j\, dx_j >0$ is equal to its infimum over centered Gaussian functions. 
\end{theorem}

Following the idea of Theorems \ref{t:Lieb} and \ref{t:BW},  let us investigate our new hypercontractivity for centered Gaussians first. 
%Let us stay on the case $n=1$ for a simplicity and investigate a functional of the hypercontractivity with  $f = \frac{\gamma_\beta}{\gamma}$ . 
To this end, let us give the precise definition of the centered Gaussian with variance $\beta > 0$ by 
$$
\gamma_\beta(x) \coloneqq \frac1{ (2\pi \beta)^\frac{n}{2} } e^{ -\frac1{2\beta} |x|^2 },\;\;\; x \in \mathbb{R}^n. 
$$
Note that $\gamma_\beta$ is $L^1(dx)$ normalized or equivalently $f = \frac{\gamma_\beta}{\gamma}$ is $L^1(\gamma)$ normalized $\int_{\mathbb{R}^n} \frac{\gamma_\beta}{\gamma} \, d\gamma =1$. 
A Gaussian integration calculus, see also Appendix \ref{S6.1}, reveals that 
\begin{equation}\label{e:PsGauss1}
\big\| P_s\big[ \big( \frac{\gamma_\beta}{\gamma} \big)^\frac1p \big] \big\|_{L^q(\gamma)}
=
\beta^{ - \frac{n}{2p} }\big(
\frac{1}{1+ \frac{ 1-\beta }{p\beta} (1-e^{-2s})}
\big)^{\frac{n}2}
\big( \frac{1 + \frac{ 1-\beta }{p\beta} ( 1+ (q-1)e^{-2s} )}{1+ \frac{ 1-\beta }{p\beta} (1-e^{-2s})} \big)^{-\frac{n}{2q}} 
\end{equation}
for any $p,q \in \mathbb{R}\setminus \{0\}$, $s,\beta>0$ as long as 
\begin{equation}\label{e:CondGauss}
1+ \frac{ 1-\beta }{p\beta} (1-e^{-2s})>0
\;\;\;
{\rm and}
\;\;\; 
1 + \frac{ 1-\beta }{p\beta} ( 1+ (q-1)e^{-2s} ) >0.
\end{equation}
In the case of $\frac{q-1}{p-1} = e^{2s}$, this quantity becomes simpler 
\begin{equation}\label{e:PsGauss2}
\big\| P_s\big[ \big( \frac{\gamma_\beta}{\gamma} \big)^\frac1p \big] \big\|_{L^q(\gamma)}
=
\beta^{\frac{n}{2p'}} \beta_{s,p}^{-\frac{n}{2q'}},
\;\;\; 
\beta_{s,p}\coloneqq 1 + (\beta-1) \frac{q}{p} e^{-2s}
\end{equation}
as long as $\beta_{s,p}>0$. 
From this expression one can see 
$
\big\| P_s\big[ \big( \frac{\gamma_\beta}{\gamma} \big)^\frac1p \big] \big\|_{L^q(\gamma)}
\le 
1
$
for all $\beta>0$ if $1< p,q <\infty$ satisfy $\frac{q-1}{p-1}\le e^{2s}$  and this corresponds to Nelson's forward hypercontractivity.  Similarly one can see Borell's reverse hypercontractivity for centered Gaussian inputs.  
%$
%\big\| P_s\big[ \big( \frac{\gamma_\beta}{\gamma} \big)^\frac1p \big] \big\|_{L^q(\gamma)}
%\ge 
%1
%$
%for all $\beta>0$ under $p,q<1$ and  \eqref{e:NelsonTime}. 
On the other hand, our new hypercontractivity inequalities \eqref{e:RHC_BSIntro} and  \eqref{e:FHC_MaIntro} require $\frac{q-1}{p-1} > e^{2s}$ and $q<0<p<1$. In such a case, one can see from \eqref{e:PsGauss1} that 
%$$
%\inf_{\beta >0}
%\big\| P_s\big[ \big( \frac{\gamma_\beta}{\gamma} \big)^\frac1p \big] \big\|_{L^q(\gamma)}
%> 0
%as long as 
\begin{equation}\label{e:Admissible}
\inf_{\beta >0}
\big\| P_s\big[ \big( \frac{\gamma_\beta}{\gamma} \big)^\frac1p \big] \big\|_{L^q(\gamma)}
> 0
\;\;\;
\Leftrightarrow
\;\;\;
1 - e^{2s} \le q < 0 < p \le 1-e^{-2s}, 
\end{equation}
and 
\begin{equation}\label{e:Admissible2}
\sup_{\beta >0}
\big\| P_s\big[ \big( \frac{\gamma_\beta}{\gamma} \big)^\frac1p \big] \big\|_{L^q(\gamma)}
<\infty 
\;\;\;
\Leftrightarrow
\;\;\;
q\le 1 - e^{2s},\; 1-e^{-2s}\le p. 
\end{equation}
We give a brief proof of this fact in Appendix \ref{App}. 
In particular one can ensure the expected inequalities \eqref{e:RHC_BSIntro} with $q_s \ge 1-e^{2s}$ and \eqref{e:FHC_MaIntro} with $p_s \le 1-e^{-2s}$ for all centered Gaussian inputs. 
Moreover, the critical threshold appears at $(p_s,q_s) = (1-e^{-2s}, 1-e^{2s})$ and it is indeed a scaling critical exponent in the sense that 
\begin{equation}\label{e:ScaleGauss}
\big\| P_s\big[ \big( \frac{\gamma_\beta}{\gamma} \big)^\frac1{1-e^{-2s}} \big] \big\|_{L^{1-e^{2s}}(\gamma)}
=1,\;\;\;\forall \beta>0. 
\end{equation}
Despite of such feasible features, the inequality \eqref{e:RHC_BSIntro}  with some ${\rm BS}_s >0$ for \textit{general} $f$ turns out to be false if $\frac{q-1}{p-1}>e^{2s}$. 
In such a case, by testing with \textit{translated} Gaussians, one can see that 
\begin{equation}\label{e:FailureOct}
\inf_{\beta>0,\; a \in \mathbb{R}} 
\big\| P_s\big[ \big( \frac{\gamma_\beta(\cdot + a)}{\gamma} \big)^\frac1p \big] \big\|_{L^q(\gamma)}
=0.
\end{equation}
This failure is indeed consistent with the Blaschke--Santal\'{o} inequality as one needs to impose some symmetric assumption on $K$ to ensure \eqref{e:BS}. From this observations,  it is reasonable to expect   reverse hypercontractivity beyond Nelson's time relation by assuming appropriate symmetric assumption. Our forthcoming result Theorem \ref{t:RevHCIntro} realizes this observation.   

It is worth to point out the link between \eqref{e:FailureOct} and Barthe--Wolff's non-degenerate condition \eqref{e:NondegBW} for the inverse Brascamp--Lieb inequality. 
As in the forward case, by using the duality argument, one can realize the  reverse hypercontractivity inequality \eqref{e:ReverseHC} as an example of the inverse Brascamp--Lieb inequality. In fact, \eqref{e:ReverseHC}  can be read as 
\begin{equation}\label{e:DualRevHC}
	\int_{\mathbb{R}^{2n}}
	e^{-\pi \langle x, \mathcal{Q}x \rangle}
	\prod_{j=1,2} f_j(x_j)^{c_j} \, dx 
	\ge 
	\big(
	\frac{ \sqrt{1-e^{-2s}} }{(2\pi)^{ \frac12(c_1+c_2)-1}}
	\big)^n
	\prod_{j=1,2} \big( \int_{\mathbb{R}^n} f_j\, dx_j \big)^{c_j}
\end{equation}
for $f_1 \coloneqq f\cdot \gamma$ and $f_2 \coloneqq \big\| P_s \big[ f^{1/p} \big] \big\|_{L^q(\gamma)}^{-q} P_s \big[ f^{1/p} \big]^q\cdot \gamma$  
with the relation \eqref{e:HCBLLink} regardless of the validity of \eqref{e:NelsonTime}, see \cite{BW} or Appendix \ref{S.6HC} for the details.
A crucial point in this link is that Nelson's time relation \eqref{e:NelsonTime} ensures Barthe--Wolff's non-degenerate condition \eqref{e:NondegBW} and moreover \eqref{e:NondegBW} fails if $\frac{q-1}{p-1}>e^{2s}$. Hence our reverse hypercontractivity \eqref{e:RHC_BSIntro}, or its dual form \eqref{e:DualRevHC}, does not satisfy  \eqref{e:NondegBW} and corresponds to Barthe--Wolff's case 1.0.1 in \cite{BW}.
Remark that Barthe--Wolff observed that their non-degenerate condition is in general necessary to ensure the consequence of Theorem \ref{t:BW} by referring to the reverse hypercontractivity inequality with $\frac{q-1}{p-1}>e^{2s}$. As in \eqref{e:FailureOct}, their failure came from translated Gaussians, see \cite[Example 2.13]{BW}. 
We postpone the Gaussian investigation on \eqref{e:FHC_MaIntro} until  Subsection \ref{S3.1} and next exhibit main results on hypercontractivity.

\subsection{Main results on hypercontractivity}
Our main result on reverse hypercontractivity realize the above observation with the centrally symmetric assumption. 

\begin{theorem}\label{t:RevHCIntro}
Let $n\in\mathbb{N}$, $s>0$, and $p_s = 1-e^{-2s}$. 
Then 
\begin{equation}\label{e:RevHCIntro}
		\big\| 
		P_s \big[ f^{\frac1{p_s}} \big]
		\big\|_{L^{-p_s}(\gamma)}
		\ge 
		\big(
		\int_{\mathbb{R}^n}
		f\, d\gamma
		\big)^\frac1{p_s}
\end{equation}
		holds for all nonnegative and centrally symmetric $f\in L^1(\gamma)$. 
		Equality in \eqref{e:RevHCIntro} is established if and only if $f$ is constant a.e. on $\mathbb{R}^n$ when $0< \int_{\mathbb{R}^n} f\, d\gamma < \infty$. 
\end{theorem}

\begin{remark}
%\begin{enumerate}
%\item 
%The case $p=1-e^{-2s}$ should be regarded as a critical case in the sense that  \eqref{e:ReverseHC} with $q = -p$ for non critical case $ p\in (0, 1-e^{-2s})$ follows from that case and moreover \eqref{e:ReverseHC} with $q = -p$ fails even under the symmetric constraint if $p > 1-e^{-2s}$. 
%In this regard, Theorem \ref{t:NelsonGene} is sharp. 
%\item 
%When $q = -p $ and $p=1-e^{-2s}$, one has $\frac{q-1}{p-1} > e^{2s}$.  
One can admit the inequality \eqref{e:RevHCIntro} for wider range of exponents. In fact, applying H\"{o}lder's inequality to \eqref{e:RevHCIntro},  one can see that 
\begin{equation}\label{e:RHC10Nov}
\big\| P_s \big[ f^\frac1p \big] \big\|_{L^{-p}(\gamma)}
\ge 
\big( \int_{\mathbb{R}^n}  f\, d\gamma \big)^\frac1p
\end{equation}
holds for all centrally symmetric $f$ and all $0<p\le1-e^{-2s}$.  
Note that this range of $p$ is sharp for \eqref{e:RHC10Nov} by virtue of \eqref{e:Admissible}. 
%\item 
%Therefore one is lead to the broader question about the Lieb's fundamental theorem for the inverse Brascamp-Lieb inequality, especially in Case 1.0.1 in \cite{BW} under the symmetric assumption. 
%\end{enumerate}
\end{remark}

In forthcoming Theorem \ref{t:NelsonGene}, we will provide stronger statement of this result where the centrally symmetric assumption is weakened to the assumption on the barycenter of $f$. 
%If $K$ is symmetric, then $f$ and $f_s$ in Proposition \ref{Prop:HCSet} are also symmetric for all $s>0$. Hence Theorem \ref{t:RevHCIntro} formally rederives Blashke--Santal\'{o} inequality for all sufficiently regular symmetric convex bodies via Proposition \ref{Prop:HCSet}.  We refer forthcoming Subsection \ref{S4.4} for the  rigorous argument. 
By virtue of Proposition \ref{Prop:HCSet}, we rederive the Blaschke--Santal\'{o} inequality for symmetric convex bodies from Theorem \ref{t:RevHCIntro} since $-p_s = -1+e^{-2s} = -2s + O(s^2)$. 
It is also appealing feature that Theorem \ref{t:RevHCIntro} provides an example of the inverse Brascamp--Lieb inequality without Barthe--Wolff's  non-degenerate condition \eqref{e:NondegBW}.  
In fact, by the identity \eqref{e:Id10Nov2} in Appendix \ref{S.6HC} and the reverse H\"{o}lder inequality, Theorem \ref{t:RevHCIntro} shows that  \eqref{e:DualRevHC} holds with $c_1 = \frac1{p_s}, c_2 = \frac1{(-p_s)'}$ and $\mathcal{Q}$ given by \eqref{e:HCBLLink} for all nonnegative and centrally symmetric $f_1,f_2 \in L^1(dx)$. 
%As we mentioned, this data corresponds to Barthe--Wolff's case 1.0.1 and does not satisfy \eqref{e:NondegBW}. 
We can state this in the following form. 
%It is then worth to note that 
%$$
%(2\pi)^{1 - \frac12(c_1+c_2)} (1-e^{-2s})^\frac12
%=
%\inf_{a_1,a_2>0} 
%\int_{\mathbb{R}^{2}}
%	\prod_{j=1,2} 
%	\gamma_{a_j}(x_j)^{c_j} 
%	e^{-\pi \langle x, \mathcal{Q}x \rangle}\, dx 
%
%and hence we obtain the following.  
\begin{corollary}
Let $n\in\mathbb{N}$, $d =2n$, $m=2$, $d_1=d_2=n$, and $p_s = 1-e^{-2s}$, $q_s = -p_s$ for $s>0$. 
For $c_1,c_2$ and $\mathcal{Q}$ given by \eqref{e:HCBLLink} and $L_j(x_1,x_2)\coloneqq  x_j$, $j=1,2$, 
$$
\inf_{f_1,f_2:\; {\rm symmetric}}{\rm BL}(f_1,f_2)
= 
\inf_{a_1,a_2>0}
{\rm BL}(\gamma_{a_1},\gamma_{a_2})
= 
\big(
	\frac{ \sqrt{1-e^{-2s}} }{(2\pi)^{ \frac12(c_1+c_2)-1}}
	\big)^n
$$
where the infimum of the left hand side is taken over all nonnegative and centrally symmetric  $f_j \in L^1(dx_j)$ such that $ \int_{\mathbb{R}^n} f_j\, dx_j >0$, $j=1,2$. 
\end{corollary}
This corollary suggests the study of the inverse Brascamp--Lieb inequality without the non-degenerate condition \eqref{e:NondegBW} by assuming instead the symmetry on the inputs $f_j$.  We do not address this problem in this paper.

We next state our result on forward hypercontractivity \eqref{e:FHC_MaIntro}.  
For this,  it is clear that one cannot use the classical forward hypercontractivity inequality\footnote{Remark that it is easy to show the inequality \eqref{e:FHC_MaIntro} with ${\rm IS}_s = 1$ if $q_s\le 1 \le p_s$ by using H\"{o}lder's inequality and the mass preservation $\| P_s f \|_{L^1(\gamma)} = \|f\|_{L^1(\gamma)}$. } since the relevant exponents $p_s=2s + O(s^2)$ and $q_s = 1-e^{2s}$ are below 1.
With its link to Mahler's conjecture in mind, it would be a hard problem to obtain some positive result on forward hypercontractivity with exponent below 1.  Nonetheless,  we establish sharp estimates for inputs satisfying the log-convexity and semi-log-concavity condition. 
%{\color{red}Something more should be written like mentioning BNT or why we impose semi-log-convexity / concavity.}

\begin{theorem}\label{t:RegFHC}

Let $n\in\mathbb{N}$, $0<p<1$, $q \in (- \infty, 1)\setminus \{0\}$, $s>0$ satisfy $\frac{q-1}{p-1} = e^{2s}$ and $\beta\ge1$. 
Then 
$$
\| P_s[f^\frac1p] \|_{L^q(\gamma)} \le \beta^{\frac{n}{2p'}} \beta_{s,p}^{-\frac{n}{2q'}} (\int_{\mathbb{R}^n} f \, d\gamma)^\frac1p
$$
holds for any positive function $f \in C^2(\mathbb{R}^n)$ satisfying   
$$
0\le \nabla^2\log\, f \le (1-\frac1\beta){\rm id},
$$
where $\beta_{s,p} \coloneqq 1+(\beta-1)\frac{q}{p}e^{-2s}$. 
Moreover  equality is achieved for $f = \frac{\gamma_\beta}{\gamma}$. 
\end{theorem}

One can at least formally obtain \eqref{e:FHC_MaIntro} with ${\rm IS}_s = \big(
\frac{\beta}{e^{\beta-1}}
\big)^{- \frac n2 e^{-2s}}$ and $p_s = 1-e^{-2s}$ for $f$ satisfying the assumption in Theorem \ref{t:RegFHC}. In fact, by taking a formal limit $p \to 1-e^{-2s}$ in Theorem \ref{t:RegFHC} in which case $q\to 0$ and $\lim_{p\to 1-e^{-2s}}  \beta^{\frac{n}{2p'}} \beta_{s,p}^{-\frac{n}{2q'}} = \big(
\frac{\beta}{e^{\beta-1}}
\big)^{-\frac n2 \frac{e^{-2s}}{p_s}}$ and then applying H\"{o}lder's inequality, it follows that 
$$
\| P_s[f^\frac1p] \|_{L^{q_s}(\gamma)} 
\le 
\| P_s[f^\frac1p] \|_{L^{0}(\gamma)} 
\le 
\big(
\frac{\beta}{e^{\beta-1}}
\big)^{- \frac n2\frac{e^{-2s}}{p_s}}
 (\int_{\mathbb{R}^n} f \, d\gamma)^\frac1{p_s}. 
$$
While this gives some nontrivial inverse Santal\'{o} inequality,  we will need more technical argument in order to derive Theorem \ref{t:ImpInv}, see forthcoming Subsection \ref{S4.6} for the precise argument. 
Finally, regarding the regularity of $f$, we will prove this theorem with lower regularity condition in Theorem \ref{t:FHC}.

The structure of the rest of this paper is as follows. 
In Section \ref{S2}, we will prove Theorem \ref{t:RevHCIntro} as well as it stronger statement Theorem \ref{t:NelsonGene}. 
We will prove Theorem \ref{t:RegFHC} and its general statement Theorem \ref{t:FHC} in Section \ref{S3}. 
We will then give applications of our hypercontractivity to problems in  convex geometry in Section \ref{S4}. 
After introducing functional forms of the Blaschke--Santal\'{o} inequality and the inverse Santal\'{o} inequality in Subsection \ref{S4.1}, we will give an idea of the proof of Proposition \ref{Prop:HCSet} in Subsection \ref{S4.2}. 
In Subsections \ref{S4.3} and \ref{S4.4} we will show how Theorem \ref{t:RevHCIntro} implies the Blaschke--Santal\'{o} inequality.  After providing the precise link between \eqref{e:FHC_MaIntro} and the inverse Santal\'{o} inequality in Subsection \ref{S4.5}, we will give the proof of Theorem \ref{t:ImpInv} in Subsection \ref{S4.6} in which we will also give a lower bound for the weighted volume product with respect to the Cauchy-type distribution.
We will give an alternative way of understanding Theorem \ref{t:ImpInv} in terms of uniform convexity and uniform smoothness in Subsection \ref{S4.7}.

\section{Reverse hypercontractivity: Proof of Theorem \ref{t:RevHCIntro}}\label{S2}

\subsection{Preliminaries to the proof of Theorem \ref{t:RevHCIntro}}\label{S2.2}
Let us give a proof of Theorem \ref{t:RevHCIntro} here. 
We indeed prove some stronger statement as follows. 

\begin{theorem}\label{t:NelsonGene}
Let $n\ge1$, $s>0$ and $  p = 1- e^{-2s}$.
\begin{enumerate}
\item 
For any nonnegative $f \in L^1(\gamma)$,  
\begin{equation}\label{e:HCTrans}
\big\| P_s \big[ f^\frac1p \big] \big\|_{L^{-p}(\gamma)}
\ge 
\inf_{c \in \mathbb{R}^n} 
e^{\frac1{2p}|c|^2} \big( \int_{\mathbb{R}^n}  f(\cdot + e^{-s}c)\, d\gamma \big)^\frac1p. 
\end{equation}
Moreover, if $f$ satisfies $0< \int_{\mathbb{R}^n} f\, d\gamma < \infty$, and achieves equality in \eqref{e:HCTrans}, then $f$ is a constant function a.e. on $\mathbb{R}^n$. 

\item
For any nonnegative $f \in L^1(\gamma)$, whose barycenter is zero in the sense that $\int_{\mathbb{R}^n} x f\, d\gamma = 0$, then 
\begin{equation}\label{e:HCBarycen2}
\big\| P_s \big[ f^\frac1p \big] \big\|_{L^{-p}(\gamma)}
\ge 
\big( \int_{\mathbb{R}^n}  f\, d\gamma \big)^\frac1p. 
\end{equation}
In particular,  \eqref{e:RevHCIntro} holds. 
Moreover, if $f$ satisfies $0< \int_{\mathbb{R}^n} f\, d\gamma < \infty$ and  achieves equality in \eqref{e:HCBarycen2}, then $f$ is a constant function a.e.  on $\mathbb{R}^n$. 
\end{enumerate}
\end{theorem}

%Few remarks are in order. 

Our proof of Theorem \ref{t:NelsonGene} is motivated from Bobkov--Ledoux's argument \cite{BoLe} showing that the Pr\'{e}kopa--Leindler inequality implies the logarithmic Sobolev inequality. As we mentioned, it is known that the latter one is equivalent to classical hypercontractivity via Gross's argument. We in fact make use of the Pr\'{e}kopa--Leindler inequality as a key tool. 
In addition, we also introduce an idea from the alternative proof of the Blaschke--Santal\'{o} inequality by Lehec \cite{LehecYaoYao} where he could managed to make use of the Yao--Yao partition theorem. 
%Because of this nature, there might be a possibility that \eqref{e:HCTrans} or \eqref{e:HCBarycen2} is indeed equivalent to the Blaschke--Santal\'{o} inequality. While one direction follows from Proposition \ref{Prop:HCSet}, the reverse direction is not clear to us. 
%Let us correct tools we will use in  below. 
\begin{theorem}[Yao--Yao partition, Theorem 7 in \cite{LehecYaoYao}]
For any Borel measure $\mu$ on $\mathbb{R}^n$, there exists a point $c\in \mathbb{R}^n$, called Yao--Yao center, and $A_1,\ldots, A_{2^n} \in SL(n)$ such that 
\begin{equation}\label{e:YaoYao}
\mathbb{R}^n = \bigcup_{i=1}^{2^n}  A_i(\mathbb{R}^n_+),\; {\rm int}\,  \big( A_k(\mathbb{R}^n_+) \big) \cap {\rm int}\,  \big( A_l(\mathbb{R}^n_+) \big) = \emptyset,\; 
\mu( c + A_i(\mathbb{R}^n_+)) = \frac{1}{2^n} \mu( \mathbb{R}^n )
\end{equation}
for any $k, l = 1, \dots, 2^n$ with $k \neq l$. 
Moreover, if $\mu(-\cdot) = \mu$, then  we can take the Yao--Yao center $c$ as the origin.. 
\end{theorem}

\begin{remark}
As it will be cleared from the proof, we will prove slightly stronger estimate than \eqref{e:HCBarycen2} by invoking the Yao--Yao center. 
Namely, we will prove 
$$
\big\| P_s \big[ f^\frac1p \big] \big\|_{L^{-p}(\gamma)}
\ge 
e^{\frac1{2}|c_0|^2} \big( \int_{\mathbb{R}^n}  f\, d\gamma \big)^\frac1p
$$
for all nonnegative $f \in L^1(\gamma)$ whose barycenter is zero, where $c_0 \in \mathbb{R}^n$ is the Yao--Yao center of $P_s \big[f^\frac1p \big]^{-p} \gamma$. 
\end{remark}

Another ingredient of our proof is the following Wang's Harnack inequality.
\begin{theorem}[Wang's Harnack inequality,  Theorem 5.6.1 in \cite{BGL}]\label{t:Harnack}
Let $\alpha>1$ and $s>0$. Then for any $f:\mathbb{R}^n\to \mathbb{R}_+$, 
\begin{equation}\label{e:Harnack}
P_s f(x)^\alpha \le P_s\big[f^\alpha\big](y) \exp\big( { \frac{ \alpha|x-y|^2 }{2(\alpha-1) (e^{2s} - 1)} }\big),\;\;\; x,y \in \mathbb{R}^n.
\end{equation}
\end{theorem}

In order to identify the case of equality in \eqref{e:HCTrans} and \eqref{e:HCBarycen2}, we will appeal to the result due to Dubuc \cite{Dubic} where the case of equality for the Pr\'{e}kopa--Leindler inequality was identified, see also \cite{BoDe21}. 
\begin{theorem}[\cite{Dubic}]\label{t:EqPL}
Let $f,g,h:\mathbb{R}^n \to \mathbb{R}_+$ be integrable with positive integral,  satisfy $h( \frac12(x+y) ) \ge f(x)^\frac12 g(y)^\frac12$ for all $x,y\in\mathbb{R}^n$, and 
$$
\big( \int_{\mathbb{R}^n} f\, dx \big)^\frac12
\big( \int_{\mathbb{R}^n} g\, dy \big)^\frac12
=
\int_{\mathbb{R}^n} 
h\, dz.
$$
Then $f,g,h$ are log-concave up to a set of measure zero and 
\begin{equation}\label{e:EqPL}
\exists a>0,\; \exists z_0 \in \mathbb{R}^n:\; 
f(x) = a^\frac12 h(x - \frac12 z_0),\; 
g(y) = a^{-\frac12} h(y+\frac12z_0).
\end{equation}
\end{theorem}

For our purpose, it is useful to restate Theorem \ref{t:EqPL} in the following form. % which immediately follows from Theorem \ref{t:EqPL} with $h(z) = \esssup_{\substack{x,y \in \mathbb{R}^n: \\ z = \frac12(x+y)}}f(x)^\frac12 g(y)^\frac12$: 
\begin{corollary}\label{t:EqPL2}
Suppose $f,g:\mathbb{R}^n \to \mathbb{R}_+$ be integrable with positive integral and  satisfy 
$$
\big( \int_{\mathbb{R}^n} f\, dx \big)^\frac12
\big( \int_{\mathbb{R}^n} g\, dy \big)^\frac12
=
\int_{\mathbb{R}^n} 
\esssup_{\substack{x,y \in \mathbb{R}^n: \\ z = \frac12(x+y)}}
f(x)^\frac12 g(y)^\frac12
\, dz.
$$
Then there exists a log-concave $h:\mathbb{R}^n\to\mathbb{R}_+$ up to measure zero such that  
\begin{equation}\label{e:EqPL2}
\exists a>0,\; \exists z_0 \in \mathbb{R}^n:\; 
f(x) = a^\frac12 h(x - \frac12 z_0),\; 
g(y) = a^{-\frac12} h(y+\frac12z_0).
\end{equation}
\end{corollary}

\subsection{Proof of Theorem \ref{t:NelsonGene}}\label{S2.3}
%The first step is to establish the inequality on the cone. 
\begin{proposition}\label{Prop:HalfHC}
Let $s>0$, $p=1-e^{-2s}$, $A \in SL(n)$ and $c \in \R^n$.
Then for any $f \colon \R^n \to \R_+$,
\begin{equation}\label{e:HalfHC}
( \int_{A(\R_+^n)} P_s[f^\frac1p]^{-p}( x + c)\gamma(x + c) \, dx)^\frac12 (\int_{A^*(\R_+^n)} P_sf(y) \gamma(y- c)\, dy)^\frac12 
\le 
\frac{1}{2^n} e^{-\frac14 |c|^2},
\end{equation}
where $A^* \coloneqq (A^{-1})^{\rm t}$. 
Moreover, if $f$ satisfies $ 0< \int_{\mathbb{R}^n} f\, d\gamma <\infty $ and establishes equality in \eqref{e:HalfHC}, then $f$ is a constant function a.e. on $\mathbb{R}^n$.
\end{proposition}
 
\begin{proof}
In view of ${\rm det}\, A = 1$, the changing of variables shows that 
\begin{align*}
&( \int_{A(\R_+^n)} P_s[f^\frac1p]^{-p}( x + c)\gamma(x + c) \, dx)^\frac12 (\int_{A^*(\R_+^n)} P_sf(y) \gamma(y- c)\, dy)^\frac12 \\
&=
( \int_{\R_+^n} P_s[f^\frac1p]^{-p}( Ax + c)\gamma(Ax + c) \, dx)^\frac12 (\int_{\R_+^n} P_sf(A^*y) \gamma(A^*y- c)\, dy)^\frac12. 
\end{align*}
We will make use of the Pr\'{e}kopa--Leindler inequality.  
To this end, we introduce $F_A,G_A : \mathbb{R}^n\to \mathbb{R}_+$ by 
\begin{align*}
F_A(X)&\coloneqq e^{X_1+\cdots + X_n} \big( P_s[f^\frac1p]^{-p}\cdot \gamma\big)( A(e^{X_1},\ldots,e^{X_n}) + c),\;\;\; X\in \mathbb{R}^n,\\
G_A(Y)&\coloneqq e^{Y_1+\cdots + Y_n}P_sf(A^*(e^{Y_1},\ldots,e^{Y_n}) ) \gamma(A^*(e^{Y_1},\ldots,e^{Y_n})- c),\;\;\; Y\in\mathbb{R}^n.
\end{align*}
Then the changing of variables $x_i = e^{X_i}$ and the Pr\'{e}kopa--Leindler inequality yield that\footnote{We remark that this consequence is just the logarithmic Pr\'{e}kopa--Leindler inequality as was used in \cite{LehecYaoYao}. Nevertheless, we prefer to argue directly in order to investigate the case of equality. } 
\begin{align}
&= 
\big(\int_{\mathbb{R}^n} F_A(X)\, dX \big)^\frac12
\big(\int_{\mathbb{R}^n} G_A(Y)\, dY \big)^\frac12\nonumber\\
&\le 
\int_{\mathbb{R}^n} 
\big[
\esssup_{\substack{ X,Y \in \mathbb{R}^n : \\ Z = \frac12 (X+Y) }}
F_A(X)
G_A(Y)
\big]^\frac12 \,
dZ\label{e:ApplyPL}\\
&= 
\int_{\mathbb{R}^n_+} 
\bigg[
\esssup_{ \substack{ x,y \in \mathbb{R}^n_+: \\ z_i = \sqrt{x_i y_i }}} 
P_s[f^\frac1p]^{-p}( Ax + c)\gamma(Ax + c)
P_sf(A^*y) \gamma(A^*y- c)
\bigg]^\frac12\, dz. \nonumber 
%\\
%&\le
%\int_{\mathbb{R}^n_+} 
%\bigg[
%\esssup_{ \substack{ x,y \in \mathbb{R}^n_+: \\ |z| = \sqrt{\langle x,y\rangle}} } 
%P_s[f^\frac1p]^{-p}( Ax + c)\gamma(Ax + c)
%P_sf(A^*y) \gamma(A^*y- c)
%\bigg]^\frac12\, dz,\nonumber 
\end{align} 
Notice that we have $\frac{p}{1-p} = e^{2s} - 1$ from the assumption  $p = 1-e^{-2s}$. 
Therefore we may employ the Harnack inequality \eqref{e:Harnack} with $\alpha = \frac1p$ to see that 
\begin{equation}\label{e:ApplyHar}
P_s[f^\frac1p]^{-p}( Ax + c)
P_sf(A^*y)
\le 
e^{\frac12 |Ax+c-A^*y|^2}
\end{equation}
for each $x,y \in \mathbb{R}^n_+$. 
This reveals that 
\begin{align}
&P_s[f^\frac1p]^{-p}( Ax + c)\gamma(Ax + c)
P_sf(A^*y) \gamma(A^*y- c)\nonumber\\
&\le
\frac{1}{(2\pi)^n}
e^{- \frac12 ( |Ax+c|^2 + |A^*y - c|^2 )} 
e^{\frac12 |Ax+c-A^*y|^2}\nonumber\\
&= 
%\frac{1}{(2\pi)^n}
%e^{- \frac12 ( |A^*y|^2 - 2\langle A^*y, c\rangle +|c|^2 )} 
%e^{\frac12 ( -2\langle Ax+c, A^*y \rangle + |A^*y|^2)}\nonumber\\
%&= 
\frac{1}{(2\pi)^n}
e^{- \frac12 |c|^2 } 
e^{-\langle x, y\rangle }\label{e:CompGauss}
\end{align}
since $\langle Ax, A^* y \rangle = \langle x,y\rangle$. 
Hence we conclude the inequality 
\begin{align*}
&( \int_{A(\R_+^n)} P_s[f^\frac1p]^{-p}( x + c)\gamma(x + c) \, dx)^\frac12 (\int_{A^*(\R_+^n)} P_sf(y) \gamma(y- c)\, dy)^\frac12 \\
&\le 
\int_{\mathbb{R}^n_+} 
\bigg[
\esssup_{ \substack{ x,y \in \mathbb{R}^n_+: \\ z_i = \sqrt{x_i y_i }}} 
\frac{1}{(2\pi)^n}
e^{- \frac12 |c|^2 } 
e^{-\langle x, y\rangle }
\bigg]^\frac12\, dz \\
&= 
e^{-\frac14|c|^2} 
\int_{\mathbb{R}^n_+} 
e^{- \frac12|z|^2}\, \frac{dz}{(2\pi)^{\frac{n}{2}}}
=
\frac1{2^n} e^{-\frac14|c|^2}.
\end{align*}

We next investigate the case of equality by assuming $f$ achieves  equality in \eqref{e:HalfHC}. 
For such $f$, inequalities we used above should be equality. 
Especially  equality in \eqref{e:ApplyPL} means $F_A$ and $G_A$ satisfy  equality for the Pr\'{e}kopa--Leindler inequality and hence Corollary \ref{t:EqPL2} ensures the existence of some log-concave function $H_A:\mathbb{R}^n \to \mathbb{R}_+$, $a>0$, and $Z_0 \in \mathbb{R}^n$ such that 
$$
H_A(Z)
= 
a^{-\frac12} 
F_A(Z+ \frac12Z_0)
= 
a^{\frac12} 
G_A(Z- \frac12Z_0),
\;\;\;
Z\in \mathbb{R}^n. 
$$
Since $H_A$ is log-concave, 
\begin{align*}
\esssup_{\substack{ X,Y \in \mathbb{R}^n : \\ Z = \frac12 (X+Y) }}
F_A(X)^\frac12
G_A(Y)^\frac12
&=
\esssup_{\substack{ X,Y \in \mathbb{R}^n : \\ Z = \frac12 (X+Y) }}
H_A (X- \frac12 Z_0)^\frac12
H_A(Y + \frac12 Z_0)^\frac12\\
&=
\esssup_{\substack{ X',Y' \in \mathbb{R}^n : \\ Z = \frac12 (X'  + Y' ) }}
H_A (X')^\frac12
H_A(Y')^\frac12\\
&=
H_A(Z)^\frac12 H_A(Z)^\frac12.
\end{align*}
Hence we see that 
\begin{align*}
&( \int_{A(\R_+^n)} P_s[f^\frac1p]^{-p}( x + c)\gamma(x + c) \, dx)^\frac12 (\int_{A^*(\R_+^n)} P_sf(y) \gamma(y- c)\, dy)^\frac12 \\
&= 
\big(\int_{\mathbb{R}^n} F_A(X)\, dX \big)^\frac12
\big(\int_{\mathbb{R}^n} G_A(Y)\, dY \big)^\frac12\\
&= 
\int_{\mathbb{R}^n} H_A(Z)^\frac12H_A(Z)^\frac12 \, dZ \\
&= 
\int_{\mathbb{R}^n} F_A(Z + \frac12 Z_0)^\frac12 G_A(Z - \frac12 Z_0)^\frac12 \, dZ.
\end{align*}
Recalling definitions of $F_A$ and $G_A$, we obtain that 
\begin{align*}
&( \int_{A(\R_+^n)} P_s[f^\frac1p]^{-p}( x + c)\gamma(x + c) \, dx)^\frac12 (\int_{A^*(\R_+^n)} P_sf(y) \gamma(y- c)\, dy)^\frac12 \\
&= 
\int_{\mathbb{R}^n} 
e^{Z_1+\cdots + Z_n}
\big(
P_s \big[f^\frac1p \big]^{-p} \cdot \gamma 
\big)^\frac12(A( e^{Z_1+ \frac12 (Z_0)_1},\ldots, e^{Z_n + \frac12 (Z_0)_n} ) + c)\\
&\quad \times 
\big(P_sf \cdot \gamma(\cdot - c)\big)^\frac12( A^*( e^{Z_1- \frac12 (Z_0)_1},\ldots, e^{Z_n - \frac12 (Z_0)_n} )  )
\, dZ\\
&=
\int_{\mathbb{R}^n_+} 
\big(
P_s \big[f^\frac1p \big]^{-p} \cdot \gamma 
\big)^\frac12(A( \lambda_1^{\frac12} z_1,\ldots, \lambda_n^{\frac12} z_n ) + c)\\
&\quad \times 
\big(P_sf \cdot \gamma(\cdot - c)\big)^\frac12( A^*( \lambda_1^{-\frac12}z_1,\ldots,\lambda_n^{-\frac12} z_n )  )
\, dz\\
&=
\prod_{i=1}^n \lambda_i^\frac12 
\int_{\mathbb{R}^n_+} 
\big(
P_s \big[f^\frac1p \big]^{-p} \cdot \gamma 
\big)^\frac12(A(\lambda_1z_1,\ldots,\lambda_nz_n) + c)
\big(P_sf \cdot \gamma(\cdot - c)\big)^\frac12( A^*z  )
\, dz,
\end{align*}
where $\lambda_i \coloneqq e^{(Z_0)_i} >0$. 
We write $(\lambda_1z_1,\ldots,\lambda_nz_n) = \Lambda z$ with $\Lambda := {\rm diag}\, (\lambda_1,\ldots,\lambda_n) $ and then apply the Harnack inequality 
\begin{equation}\label{e:ApplyHar2}
P_s \big[f^\frac1p \big]^{-p} (A(\Lambda z) + c)
P_sf ( A^*z  )
\le 
e^{ \frac12 | A(\Lambda z) + c - A^* z |^2 },
\;\;\;
z \in \mathbb{R}^n_+
\end{equation}
to see that 
\begin{align*}
&( \int_{A(\R_+^n)} P_s[f^\frac1p]^{-p}( x + c)\gamma(x + c) \, dx)^\frac12 (\int_{A^*(\R_+^n)} P_sf(y) \gamma(y- c)\, dy)^\frac12 \\
&\le 
\prod_{i=1}^n \lambda_i^\frac12 
\int_{\mathbb{R}^n_+}
\bigg[
\gamma(A(\Lambda z) + c)
\gamma(A^* z - c)
e^{ \frac12 | A(\Lambda z) + c - A^* z |^2 }
\bigg]^\frac12\, dz.
\end{align*}
Similar calculation as in \eqref{e:CompGauss} reveals that 
\begin{align*}
&\prod_{i=1}^n \lambda_i^\frac12
\int_{\mathbb{R}^n_+}
\bigg[
\gamma(A(\Lambda z) + c)
\gamma(A^* z - c)
e^{ \frac12 | A(\Lambda z) + c - A^* z |^2 }
\bigg]^\frac12\, dz \\
&= 
e^{-\frac14 |c|^2}
\prod_{i=1}^n \lambda_i^\frac12
\int_{\mathbb{R}^n_+}
\gamma( \lambda_1^\frac12 z_1,\ldots, \lambda_n^\frac12 z_n )\, dz \\
&= 
e^{-\frac14 |c|^2}
\int_{\mathbb{R}^n_+}
\, d\gamma
=
\frac1{2^n}e^{-\frac14|c|^2}
\end{align*}
which rederives \eqref{e:HalfHC}. Since we assumed  equality in \eqref{e:HalfHC} and $P_s f, P_s \big[ f^\frac1p \big]^{-p}$ are continuous, this means that the Harnack inequality \eqref{e:ApplyHar2} must be equality for all $z\in \mathbb{R}^n_+$. 
As we will see in the following Lemma \ref{l:EqHarGene}, this yields the conclusion. 
\end{proof}

\begin{lemma}\label{l:EqHarGene}
Let $\alpha>1$, $s>0$, $\lambda_1,\ldots, \lambda_n >0$, $A \in SL(n)$, and $c \in \mathbb{R}^n$. 
Suppose $f :\mathbb{R}^n \to \mathbb{R}_+$ satisfies $0< \int_{\mathbb{R}^n} f \, d\gamma < \infty$ and 
\begin{equation}\label{e:EqHarGene}
P_s f(A^* z)^\alpha 
= 
P_s\big[f^\alpha\big](A(\Lambda z) + c) 
\exp \big( \frac{\alpha}{2(\alpha-1) (e^{2s}-1)} | A(\Lambda z) + c - A^*z |^2  \big)
\end{equation}
for all $z \in \mathbb{R}^n_+$ where $\Lambda \coloneqq {\rm diag}\, (\lambda_1,\ldots,\lambda_n)$. Then 
$$
c=0, \;\;\; A\Lambda = A^*, \;\;\; f(x) = C_A
$$
holds for a.e. $x\in \mathbb{R}^n$ and for some constant $C_A >0$.
\end{lemma}

\begin{proof}
From the definition of $P_s$, we have that 
\begin{align*}
&P_s f(A^*z)^\alpha \\
&=
\bigg(
\int_{\mathbb{R}^n}
f( e^{-s} \big[ A(\Lambda z) + c \big] + e^{-s} \big[ A^*z - A(\Lambda z) - c  \big] + \sqrt{1-e^{-2s}} w ) \gamma(w)\, dw
\bigg)^\alpha \\
&= 
\bigg(
\int_{\mathbb{R}^n}
f( e^{-s} \big[ A(\Lambda z) + c \big] + \sqrt{1-e^{-2s}} u ) 
\gamma(u - \frac{e^{-s} \big[ A^*z - A(\Lambda z) - c  \big]}{\sqrt{1-e^{-2s}}}  )
\, du
\bigg)^\alpha \\
&= 
\bigg(
\int_{\mathbb{R}^n}
f( e^{-s} \big[ A(\Lambda z) + c \big] + \sqrt{1-e^{-2s}} u ) 
\gamma(u)\\
&\quad \times
e^{  \frac{1}{\sqrt{e^{2s}-1}} \langle u,  A^*z - A(\Lambda z) - c  \rangle } 
e^{  -\frac12 \frac1{e^{2s}-1} | A^*z - A(\Lambda z) - c  |^2 } 
\, du
\bigg)^\alpha \\
&= 
\bigg(
\int_{\mathbb{R}^n}
\big[ f( e^{-s} \big[ A(\Lambda z) + c \big] + \sqrt{1-e^{-2s}} u ) 
\gamma(u)^\frac1\alpha \big]\\
&\quad \times
\big[ e^{  \frac{1}{\sqrt{e^{2s}-1}} \langle u,  A^*z - A(\Lambda z) - c  \rangle }
\gamma(u)^{1-\frac1\alpha} \big] 
\, du
\bigg)^\alpha
e^{  -\frac12 \frac{\alpha}{e^{2s}-1} | A^*z - A(\Lambda z) - c  |^2 }. 
\end{align*}
We then apply H\"{o}lder's inequality 
\begin{align}\label{e:ApplyHolGene}
&\int_{\mathbb{R}^n}
\big[ f( e^{-s} \big[ A(\Lambda z) + c \big] + \sqrt{1-e^{-2s}} u ) 
\gamma(u)^\frac1\alpha \big]
\big[ e^{  \frac{1}{\sqrt{e^{2s}-1}} \langle u,  A^*z - A(\Lambda z) - c  \rangle }
\gamma(u)^{1-\frac1\alpha} \big] 
\, du \nonumber \\
&\le 
P_s\big[f^\alpha\big]( A(\Lambda z) + c )^\frac1\alpha
\big(
\int_{\mathbb{R}^n} 
e^{  \frac{\alpha'}{\sqrt{e^{2s}-1}} \langle u,  A^*z - A(\Lambda z) - c  \rangle }
\gamma(u)\, du
\big)^{\frac1{\alpha'}}.
\end{align}
We can directly compute the Gaussian integral 
$$
\int_{\mathbb{R}^n} 
e^{  \frac{\alpha'}{\sqrt{e^{2s}-1}} \langle u,  A^*z - A(\Lambda z) - c  \rangle }
\gamma(u)\, du
= 
e^{\frac12 \frac{ |\alpha'|^2 }{e^{2s}-1} | A^*z - A(\Lambda z) - c |^2}
$$
and hence 
\begin{align*}
P_sf(A^*z)^\alpha 
&\le 
P_s\big[f^\alpha\big](A(\Lambda z) +c) 
e^{\frac12 \frac{ \alpha \alpha' }{e^{2s}-1} | A^*z - A(\Lambda z) - c |^2}
e^{  -\frac12 \frac{\alpha}{e^{2s}-1} | A^*z - A(\Lambda z) - c  |^2 }\\
&= 
P_s\big[f^\alpha\big](A(\Lambda z) +c) 
e^{\frac12 \frac{ \alpha }{(e^{2s}-1)(\alpha-1)} | A^*z - A(\Lambda z) - c |^2}.
\end{align*}
However, our assumption tells us that the inequality we obtained is indeed equality for all $z \in \mathbb{R}^n_+$. 
In particular, H\"{o}lder's inequality \eqref{e:ApplyHolGene} must be equality. 
This means that 
\begin{equation}\label{e:EqHar2}
\big[ f( e^{-s} \big[ A(\Lambda z) + c \big] + \sqrt{1-e^{-2s}} u ) 
\gamma(u)^\frac1\alpha \big]^\alpha
=
C
\big[ e^{  \frac{1}{\sqrt{e^{2s}-1}} \langle u,  A^*z - A(\Lambda z) - c  \rangle }
\gamma(u)^{1-\frac1\alpha} \big]^{\alpha'} 
\end{equation}
for all $z\in \mathbb{R}^n_+$, $u \in \mathfrak{A}(z)$, and some $C>0$, where complement of $\mathfrak{A}(z)$ in $\mathbb{R}^n$ is some Lebesgue null set depending on $z$. 
We can reformulate \eqref{e:EqHar2} as 
\begin{align}\label{e:EqqCond}
f(w) 
=
C^\frac1\alpha
{\rm exp}
\big[& \frac{\alpha'}{\alpha} \frac{1}{ \sqrt{e^{2s}-1} \sqrt{1-e^{-2s}}} \langle w, A^*z - A(\Lambda z)- c \rangle \\
&- \frac{\alpha'}{\alpha} \frac{e^{-s}}{ \sqrt{e^{2s}-1} \sqrt{1-e^{-2s}}} \langle A(\Lambda z) + c, A^*z - A(\Lambda z) - c \rangle
\big]\nonumber 
\end{align}
for any $z \in \mathbb{R}_+$ and $w \in E(z)$, where $E(z)$ is given by 
$$
E(z) \coloneqq e^{-s} A(\Lambda z) + e^{-s}c + \sqrt{1-e^{-2s}} \mathfrak{A}(z).
$$
We remark that $E(z)$ coincides to $\mathbb{R}^n$ up to Lebesgue zero measure for each $z \in \mathbb{R}_+^n$.
In particular, considering $z=0$, this shows that 
\begin{equation}\label{e:EqqCond1}
f(w) = C^\frac1\alpha {\rm exp} \big[{ -\frac{\alpha'}{\alpha} \frac{1}{ \sqrt{e^{2s}-1} \sqrt{1-e^{-2s}}} \langle w, c \rangle 
+ \frac{\alpha'}{\alpha} \frac{e^{-s}}{ \sqrt{e^{2s}-1} \sqrt{1-e^{-2s}}} |c|^2}\big],  
\end{equation}
for all $w \in E(0)$.
For arbitrary fixed $z \in \mathbb{R}_+^n$, since the right hands of \eqref{e:EqqCond} and \eqref{e:EqqCond1} must coincide for any $ w \in E(z) \cap E(0)$, we obtain $A^*z = A(\Lambda z)$. 
Since $z \in \mathbb{R}^n_+$ is arbitrary and $A^*, A\Lambda$ are nondegenerate, we obtain $A^* = A\Lambda$. 
This fact together with \eqref{e:EqqCond} yields that 
\begin{align}\label{e:EqqCond2}
f(w) 
= 
C^\frac1\alpha
{\rm exp} \big[& - \frac{\alpha'}{\alpha} \frac{1}{ \sqrt{e^{2s}-1} \sqrt{1-e^{-2s}}} \langle w, c \rangle \\
&- \frac{\alpha'}{\alpha} \frac{e^{-s}}{ \sqrt{e^{2s}-1} \sqrt{1-e^{-2s}}} \langle A(\Lambda z) + c, - c \rangle \big]\nonumber 
\end{align}
for any $z \in \mathbb{R}^n_+$ and $w \in E(z)$. 
When we take $z$ as $e_1, \dots, e_n$ which are the standard basis in $\mathbb{R}^n$, since \eqref{e:EqqCond1} and \eqref{e:EqqCond2} must coincide on $E(0) \cap E(z)$ for each $z$, we see that 
$\langle A(\Lambda e_i), c \rangle = 0$ for any $i=1, \dots, n$. Since $A\Lambda$ is nondegenerate by the assumptions, we also conclude $c=0$. 
Finally, it follows from $c=0$ and \eqref{e:EqqCond1} that $f$ is constant a.e. on $\mathbb{R}^n$.
\end{proof}

%\begin{remark}\label{Rem:Eq}
%In Lemma \ref{l:EqHarGene}, one can drop the assumption of the continuity of $f$ in the case $c=0$. 
%In fact, in that case, we only have \eqref{e:EqHar2} for a.e. $u \in \mathbb{R}^n$ and all $z\in \mathbb{R}^n_+$. So we cannot simply choose $u=0$ as before. However, we can choose $z = 0$ in \eqref{e:EqHar2} to see that $f(x) \equiv C $ for a.e. $x \in \mathbb{R}^n$ since $c=0$.  
%Similarly, one can drop the assumption of the continuity of $f$ in the case $c=0$ in Proposition \ref{Prop:HalfHC} too and the consequence becomes $f(x) \equiv C$ for a.e. $x \in \mathbb{R}^n$. 
%\end{remark}

\begin{proof}[Proof of Theorem \ref{t:NelsonGene}]
Let us first prove \textit{(1)}. To see \eqref{e:HCTrans}, it suffices to show 
\begin{equation}\label{e:Goal23Aug}
\int_{\R^n} P_s[f^\frac1p]^{-p}\, d\gamma  \le e^{-\frac12 |c_0|^2} \big( \int_{\R^n} P_sf( \cdot + c_0)\, d\gamma \big)^{-1},
\end{equation}
for some $c_0$ since we know that $P_sf(x+c_0) = P_s\big[ f(\cdot+e^{-s}c_0)\big](x)$ and $\int_{\mathbb{R}^n} P_sf\, d\gamma = \int_{\mathbb{R}^n} f\, d\gamma$. 
We choose $c_0\in \mathbb{R}^n$ as a center of Yao--Yao equipartition of $\mu = P_s\big[f^\frac1p\big]^{-p}\gamma$. By a translation, this means that the center of Yao--Yao equipartition of $P_s\big[f^\frac1p\big]^{-p}(\cdot +c_0)\gamma(\cdot+c_0)$ is $0$. Namely, there exist $A_1,\ldots, A_{2^n} \in SL(n)$ so that $( A_i(\mathbb{R}^n_+) )_{i=1}^{2^n}$ is a partition of $\mathbb{R}^n$ and 
$$
\int_{A_i(\mathbb{R}^n_+)} P_s\big[f^\frac1p\big]^{-p}(x +c_0)\gamma(x+c_0)\, dx 
=
\frac{1}{2^n} 
\int_{\mathbb{R}^n} P_s\big[f^\frac1p\big]^{-p}(x +c_0)\gamma(x+c_0)\, dx,
$$
for all $i=1,\ldots, 2^n$. 
As we observed in Proposition \ref{Prop:HalfHC}, we have 
\begin{equation}\label{e:ApplyProp}
\int_{A_i(\mathbb{R}^n_+)} P_s\big[f^\frac1p\big]^{-p}(x +c_0)\gamma(x+c_0)\, dx
\le 
\frac{1}{2^{2n}} e^{-\frac12|c_0|^2} 
\big(
\int_{A_i^*(\mathbb{R}^n_+)} P_sf( y ) \gamma (y-c_0)\, dy
\big)^{-1},
\end{equation}
and hence 
$$
\int_{\mathbb{R}^n} P_s\big[f^\frac1p\big]^{-p}(x +c_0)\gamma(x+c_0)\, dx
\le 
2^n 
\frac{1}{2^{2n}} e^{-\frac12|c_0|^2} 
\big(
\int_{A_i^*(\mathbb{R}^n_+)} P_sf( y ) \gamma (y-c_0)\, dy
\big)^{-1},
$$
for all $i=1,\ldots, 2^n$.  Equivalently 
\begin{equation}\label{e:EachPiece}
\big(
\int_{\mathbb{R}^n} P_s\big[f^\frac1p\big]^{-p}(x +c_0)\gamma(x+c_0)\, dx
\big)
\big(
\int_{A_i^*(\mathbb{R}^n_+)} P_sf( y ) \gamma (y-c_0)\, dy
\big)
\le 
\frac{1}{2^{n}} e^{-\frac12|c_0|^2}.  
\end{equation}
We next claim that 
\begin{equation}\label{e:23Aug_1}
\sum_{i=1}^{2^n} 
\int_{A_i^*(\mathbb{R}^n_+)} P_sf( y ) \gamma (y-c_0)\, dy
=
\int_{{\R}^n} P_sf( y ) \gamma (y-c_0)\, dy. 
\end{equation}
Once we prove this, then we conclude \eqref{e:Goal23Aug} by summing up \eqref{e:EachPiece} with this identity. 
%that there exists at least one $i$ such that 
%\begin{equation}\label{e:23Aug_1}
%\int_{A_i^*(\mathbb{R}^n_+)} P_sf( y ) \gamma (y-c_0)\, dy
%\ge 
%\frac{1}{2^n} 
%\int_{\mathbb{R}^n} P_sf( y ) \gamma (y-c_0)\, dy
%\end{equation}
%from which we conclude \eqref{e:Goal23Aug}. 
To see \eqref{e:23Aug_1}, we notice that $A_i^*(\mathbb{R}^n_+)$ is a dual cone of $A_i(\mathbb{R}^n_+)$ in the sense that 
$$
A_i^*(\mathbb{R}^n_+)
=
\{ x \in \mathbb{R}^n: \langle x, y\rangle \ge0,\; \forall y \in A_i(\mathbb{R}^n_+) \}.
$$
Hence we may apply Corollary 5 in \cite{LehecYaoYao} to ensure that $(A_i^*(\mathbb{R}^n_+))_{i=1}^{2^n}$ is also a partition of $\mathbb{R}^n$ and so \eqref{e:23Aug_1} is confirmed.

We next investigate the case of equality. 
Suppose $f$ achieves  equality in \eqref{e:HCTrans}. Then for $c_0$ which is the center of Yao--Yao equipartition of $P_s \big[f^\frac1p\big]^{-p}\gamma$, we have a chain of inequalities 
\begin{align*}
\int_{\R^n} P_s[f^\frac1p]^{-p}\, d\gamma  
&\le 
 e^{-\frac12 |c_0|^2} 
\big( \int_{\R^n} P_sf( \cdot + c_0)\, d\gamma \big)^{-1}\\
&\le 
\sup_{c\in\mathbb{R}^n}
 e^{-\frac12 |c|^2} 
\big( \int_{\R^n} P_sf( \cdot + c)\, d\gamma \big)^{-1}\\
&=
\int_{\R^n} P_s[f^\frac1p]^{-p}\, d\gamma  
\end{align*}
where the first inequality follows from \eqref{e:Goal23Aug} and the last equality comes from the assumption. 
In particular, this shows that the inequality \eqref{e:Goal23Aug} must be equality for such $f$. Hence inequalities we used to prove \eqref{e:Goal23Aug} must be equalities too. 
In particular, \eqref{e:ApplyProp} is equality for all $i=1,\ldots,2^n$. 
Hence we may appeal to the case of equality in Proposition \ref{Prop:HalfHC} to conclude that $f$ is constant a.e. on $\mathbb{R}^n$. %see that 
%$$
%f(x) = C_{A_i},\;\;\; x\in A_i(\mathbb{R}^n_+) + e^{-s}c_0,\;\;\; i = 1,\ldots, 2^n. 
%$$ 
%{\color{red}Since $f$ is continuous and $( A_i(\mathbb{R}^n_+) + e^{-s}c_0 )_{i=1}^{2^n}$ is partition of $\mathbb{R}^n$, $C_{A_1} = \cdots = C_{A_{2^n}}$ from which 
%we conclude that $f$ is a constant function on $\mathbb{R}^n$. 
%NEED TO CHECK!
%}

Next let us show the part \textit{(2)}.  
To show \eqref{e:HCBarycen2} for $f$ whose barycenter is zero,  we take $c_0\in \mathbb{R}^n$ so that \eqref{e:Goal23Aug} holds and let us show 
\begin{equation}\label{e:31Aug}
e^{-\frac12 e^{-2s} |c_0|^2} \big( \int_{\R^n} P_sf( y + c_0)\, d\gamma(y) \big)^{-1}
\le 
\big( \int_{\R^n} f \, d\gamma \big)^{-1}
\end{equation}
from which we conclude \eqref{e:HCBarycen2} since $p=1-e^{-2s}$. 
A simple application of an inequality $e^{\langle y, e^{-s}c_0\rangle} \ge 1+\langle y, e^{-s} c_0\rangle$ for any $y \in \mathbb{R}^n$ shows that 
\begin{align*}
&\int_{\R^n} P_sf( y + c_0)\, d\gamma(y)\\
%&=
%\int_{\R^n} P_s \big[ f(\cdot + e^{-s} c_0) \big](y)\, d\gamma(y)\\
&=
\int_{\R^n} f( x ) \gamma(x - e^{-s} c_0)\, dx \\
&= 
e^{- \frac12 e^{-2s} |c_0|^2} \int_{\R^n} f( x ) \gamma(x) e^{\langle x, e^{-s}c_0\rangle }\, dx \\ 
&\ge
e^{- \frac12 e^{-2s} |c_0|^2} \int_{\R^n} f( x ) \, d \gamma(x) 
+
e^{- \frac12 e^{-2s} |c_0|^2} \int_{\R^n} \langle x, e^{-s} c_0\rangle f( x ) \, d \gamma(x). 
\end{align*}
Applying $\int x f\, d\gamma =0$, we conclude \eqref{e:31Aug}.

Suppose $f$ achieves equality in \eqref{e:HCBarycen2}. 
Then equality of \eqref{e:Goal23Aug} must hold for the Yao--Yao center $c_0$ from the above argument. This fact with Lemma \ref{l:EqHarGene} yields that $f$ is constant a.e. on $\mathbb{R}^n$.
\if0
 Then above inequalities are also equalities and in particular we must have 
$$
e^{\langle y, e^{-s}c_0\rangle} = 1+\langle y, e^{-s} c_0\rangle,
\;\;\; y \in \mathbb{R}^n
$$
from which $c_0 = 0$ follows. 
Hence $f$ achieves  equality in \eqref{e:Goal23Aug} with $c_0=0$. This concludes $f$ is constant $a.e.$ on $\mathbb{R}^n$. 
%Hence we may appeal to the Remark for the Lemma \ref{l:EqHarGene} to conclude that $f$ is constant a.e.
\fi
\end{proof}

\section{Forward hypercontractivity: Proof of Theorem \ref{t:RegFHC}}\label{S3}
\subsection{Preliminaries to the proof of Theorem \ref{t:RegFHC}}\label{S3.1}
We next address forward hypercontractivity \eqref{e:FHC_MaIntro}. 
As in the introduction, let us begin with an investigation of centered Gaussians.  
%For $p,q \in \mathbb{R}$ and $s>0$,  let $\mathcal{F}_{p,q,s} \in [0,\infty]$ be the best constant of an inequality 
%\begin{equation}\label{e:FHC}
%\big\|
%P_s \big[ 
%f^\frac1p
%\big]
%\big\|_{L^q(\gamma)} 
%\le 
%\mathcal{F}_{p,q,s}
%\big(
%\int_{\mathbb{R}^n}
%f\, d\gamma
%\big)^\frac1p
%\end{equation}
%for all nonnegative $f \in L^1(\gamma)$.  
%In this terminology Nelson's hypercontractivity says that $\mathcal{F}_{p,q,s} = 1$ for $s>0$ if $p,q >1$ satisfy \eqref{e:NelsonTime} and $\mathcal{F}_{p,q,s} = \infty$ if $p,q >1$ satisfy $\frac{q-1}{p-1} > e^{2s}$.  
%It is worth to note a simple property
%\begin{equation}\label{e:ConstMono}
%p_1 \le p_2,\; q_1\le q_2
%\;\;\; \Rightarrow \;\;\; 
%\mathcal{F}_{p_2,q_1,s} \le \mathcal{F}_{p_1,q_2,s} 
%\end{equation}
%thanks to H\"{o}lder's inequality {\color{red}TRUE???}. 
%To approach the case of $p,q<1$,  as in the beginning of subsection \ref{SubS:NelImp}, let us start with an Gaussian investigation.  
For $p,q \in (-\infty, 1) \setminus\{0\}$,  it follows from \eqref{e:PsGauss1} that 
\begin{equation}\label{e:ForwCentre}
\sup_{\beta>0}
\big\| P_s \big[ \big(\frac{\gamma_\beta}{\gamma}\big)^\frac1p \big] \big\|_{L^q(\gamma)}
< \infty 
\;\;\;
\Leftrightarrow
\;\;\;
q \le 1-e^{2s},\; 1-e^{-2s} \le p. 
\end{equation}
From this observation one can realize that the end point of the expected inequality \eqref{e:FHC_MaIntro} is at $(p_s,q_s) = (1-e^{-2s}, 1-e^{2s})$.  
Moreover in this case, there is the scaling structure \eqref{e:ScaleGauss}. 
However this fact tells us that the best constant of \eqref{e:FHC_MaIntro}  at the endpoint cannot be exhausted by centered Gaussians in a light of the link to Mahler's conjecture. Recall that equality is established by simplex rather than ${\rm B}_2^n$.
In fact we may apply the example of (functional form of) Mahler's conjecture $f_*(x) \coloneqq {\bf 1}_{[-1,\infty)^n} e^{-(x_1 + \dots +x_n)}/ \gamma(x)$ from \cite{FM}. Here ${\bf 1}_{[-1,\infty)^n}(x)$ is a characteristic function of $[-1,\infty)^n$. 
Then one can see from the direct calculation that 
\begin{equation}\label{e:Example18Oct}
\big\|
P_s \big[
f_*^{\frac1{p_s}}
\big]
\big\|_{L^{q_s}(\gamma)}/ 
\big(
\int_{\mathbb{R}^n}
f_*\, d\gamma 
\big)^{\frac1{p_s}}
= 
e^{-\frac{n}{p_s}} 
(2\pi)^{-\frac{n}{q_s}} 
p_s^{\frac{n}2}
e^{ns(2 - \frac1{q_s})}
\Gamma(1-q_s)^{\frac{n}{q_s}}
\end{equation}
for $p_s = 1-e^{-2s}$ and $q_s = 1-e^{2s}$. 
Here $\Gamma$ means the Gamma function.
Therefore it is reasonable to conjecture \eqref{e:FHC_MaIntro} at the endpoint with the constant 
$$
{\rm IS}_s^{\frac1{p_s}}
=
e^{-\frac{n}{p_s}} 
(2\pi)^{-\frac{n}{q_s}} 
p_s^{\frac{n}2}
e^{ns(2 - \frac1{q_s})}
\Gamma(1-q_s)^{\frac{n}{q_s}}. 
$$
One can directly check that 
$$
\lim_{s\downarrow0}
\big(
e^{-\frac{n}{p_s}} 
(2\pi)^{-\frac{n}{q_s}} 
p_s^{\frac{n}2}
e^{ns(2 - \frac1{q_s})}
\Gamma(1-q_s)^{\frac{n}{q_s}}
\big)^{-p_s}
=
\big(
\frac{e}{2\pi}
\big)^n
$$
which corresponds to the constant for the functional form of the inverse Santal\'{o} inequality, \cite{FM} or \eqref{e:FuncInvSantaloGene}, and hence if \eqref{e:FHC_MaIntro} could be proved with the above constant at the endpoint then this would solve Mahler's conjecture for general set  \eqref{e:MahGene} from Proposition \ref{Prop:HCSet}. 
%We do not pursue this direction in this paper. 
%Hence one might expect 
%In view of \eqref{e:ConstMono}, this suggests the endpoint estimate at $ q_s = 1 - e^{2s} < 0 < p_s = 1-e^{-2s}$. 
%With this in mind we ask two question (i) if one can prove the inequality 
%\begin{equation}\label{e:LoughConst}
%\big\|
%P_s \big[ 
%f^\frac1{p_s}
%\big]
%\big\|_{L^{q_s}(\gamma)} 
%\le 
%C_s
%\big(
%\int_{\mathbb{R}^n}
%f\, d\gamma
%\big)^\frac1{p_s}
%\end{equation}
%for some finite constant $C_s>0$ and all $f$.
% ? (ii) if so,  what is the sharp constant $C_s$? 

%\section{Mahler's conjecture as the reverse type of reverse hypercontractivity}
We will prove Theorem \ref{t:RegFHC} with low regularity condition. 
To this end let us introduce further notations. 
For $\alpha \in \mathbb{R}$,  we say a function $\phi:\mathbb{R}^n \to \mathbb{R}$ is $\alpha$-semi-convex if 
\begin{equation}\label{e:DefConv}
\phi((1 - \lambda) x_1 + \lambda x_2) 
\le 
(1 - \lambda) \phi(x_1) + \lambda \phi(x_2) - \frac{\alpha}2 \lambda(1-\lambda) |x_1-x_2|^2
\end{equation}
for all $\lambda \in [0,1]$ and $x_1,x_2 \in \mathbb{R}^n$. 
Similarly a function $\phi:\mathbb{R}^n \to \mathbb{R}$ is said to be $\alpha$-semi-concave if 
\begin{equation}\label{e:DefConc}
\phi((1 - \lambda) x_1 + \lambda x_2) 
\ge 
(1 - \lambda) \phi(x_1) + \lambda \phi(x_2) - \frac{\alpha}2 \lambda(1-\lambda) |x_1-x_2|^2
\end{equation}
for all $\lambda \in [0,1]$ and $x_1,x_2 \in \mathbb{R}^n$. 
In the case of $\alpha=0$,  the conditions coincide with standard convexity and concavity.  
We remark that if $\phi \in C^2(\mathbb{R}^n)$, then $\phi$ is $\alpha$-semi-convex if and only if 
$
\nabla^2 \phi \ge \alpha {\rm id}
$
on $\mathbb{R}^n$ and similar to $\alpha$-semi-concavity. 
We also say that a function $f:\mathbb{R}^n\to (0,\infty)$ is $\alpha$-semi-log-convex (concave) if $\log\, f$ is $\alpha$-semi-convex (concave). 
Employing these notions, we will prove the following which yields Theorem \ref{t:RegFHC}. 

\begin{theorem}\label{t:FHC}
Let $0<p<1$, $q \in (- \infty, 1)\setminus \{0\}$, $s>0$ satisfy $\frac{q-1}{p-1} = e^{2s}$ and $\beta\ge1$. 
Then for any log-convex and $1-\frac1\beta$-semi-log-concave $f \colon \mathbb{R}^n \to (0, \infty)$,  
$$
\| P_s[f^\frac1p] \|_{L^q(\gamma)} \le \| P_s[(\frac{\gamma_\beta}{\gamma})^\frac1p]\|_{L^q(\gamma)} (\int_{\mathbb{R}^n} f \, d\gamma)^\frac1p.
$$
\end{theorem}

%{\color{red}
%In the case of $p<1-e^{-2s}$, one does not need to impose $\beta_{s,p}>0$ since if $\beta$ is such that $\beta_{s,p}\le0$ then $\| P_s[(\frac{\gamma_\beta}{\gamma})^\frac1p]\|_{L^q(\gamma)} = \infty$ and the consequence inequality automatically holds. 
%However, in the case $p>1-e^{-2s}$, if $\beta$ is such that $\beta_{s,p} \downarrow 0$ (probably the case $\beta \downarrow \beta_{**}$) then $\| P_s[(\frac{\gamma_\beta}{\gamma})^\frac1p]\|_{L^q(\gamma)}  \to 0$ and the consequence inequality clearly fails. 
%$\to$ This could be fixed. In fact,  in this case we always have $\beta_{s,p}>0$!
%}

\subsection{Proof of Theorem \ref{t:FHC}}\label{S3.2}
Our strategy of proving Theorem \ref{t:FHC} is based on the flow monotonicity scheme for Fokker--Planck flow.  The critical use of Fokker--Planck flow in the context of hypercontractivity can be found in our previous work with Neal Bez \cite{BNT} for instance. 
Our key idea is the closure type property of the Fokker--Planck equation, see \cite[Theorems 4.5, 4.6]{BNT}.  The origin of the closure type property, which sophisticates the flow monotonicity argument in some sense, can be found in the work of Bennett--Bez \cite{BBJGA} regarding the sharp Young's inequality. This idea was also used to show the forward Brascamp--Lieb inequality in \cite{BBCrell},  the forward hypercontractivity inequality in \cite{ABBMMS}, and the inverse Brascamp--Lieb inequality in \cite{BN}. 

For $\beta>0$ and nonnegative initial data $v_0 \in L^1(dx)$, we say that $v=v(t,x)$ is a $\beta$-Fokker--Planck solution if it solves 
\begin{equation}\label{e:BetaFP}
\begin{cases}
\partial_t v = \mathcal{L}_\beta^* v \coloneqq \beta\Delta v_t + x\cdot \nabla v_t + n v_t,\;\;\; (t,x) \in (0,\infty)\times \mathbb{R}^n, \\
v(0,x) = v_0(x),\;\;\; x\in \mathbb{R}^n. 
\end{cases}
\end{equation}
We will frequently use the notation $v_t \coloneqq v(t,\cdot)$ in below. 
The solution $v_t$ has an explicit formula 
\begin{equation}\label{e:Sol20Oct}
v_t(x) 
= 
\frac1{\big(
2\pi \beta (1-e^{-2t})
\big)^{\frac{n}{2}}}
\int_{\mathbb{R}^n}
e^{ - \frac{ |y-e^{-t}x|^2 }{2\beta (1-e^{-2t})} }
v_0(y)\, dy
\end{equation}
and it is a unique solution if $v_0 \in L^1(dx)$, see \cite{Widder}. 
We then present new closure type property which is the key of the proof of Theorem \ref{t:FHC}.

\begin{theorem}\label{t:RevRevClosureGene}
Let $0 < p <1$, $-\infty < q < 1$, $s>0$ such that 
$\frac{q-1}{p-1} = e^{2s}$ and 
$\beta \ge1$ satisfy
\begin{equation}\label{e:beta_sp}
\beta_{s, p} \coloneqq 1 + (\beta-1)\frac{q}{p} e^{-2s} >0. 
\end{equation}

For any initial data $v_0 \colon \mathbb{R}^n \to (0, \infty)$ such that $\frac{v_0}{\gamma}$ is log-convex and $1-\frac1\beta$-semi-log-concave,  we set 
$$
\widetilde{v}_t(x) \coloneqq P_s \big[ \big( \frac{v_t}{\gamma} \big)^\frac1p \big](x)^q \gamma(x),\;\;\; (t,x) \in (0,\infty) \times \mathbb{R}^n. 
$$
\begin{enumerate}
\item When $1 - e^{-2s} < p < 1$, namely $q> 0$,  we have that 
$$
\partial_t \widetilde{v}_t - \mathcal{L}_{\beta_{s, p}}^* \widetilde{v}_t \ge 0.
$$
\item When $0 < p < 1-e^{-2s}$, namely $q< 0$,  we have that 
$$
\partial_t \widetilde{v}_t - \mathcal{L}_{\beta_{s, p}}^* \widetilde{v}_t \le 0.
$$
\end{enumerate}
\end{theorem} 

It should be remarked that $v_0$ in Theorem \ref{t:RevRevClosureGene} is indeed in $L^1(dx)$ and hence the existence and uniqueness of the $\beta$-Fokker--Planck solution $v_t$ are confirmed.  
We give a short proof of $v_0 \in L^1(dx)$ in Appendix \ref{S6.3}. 

To prove Theorem \ref{t:RevRevClosureGene}, we will make use of an identity that we observed in \cite[Lemma 4.2]{BNT}
\begin{align}\label{e:ClosureFromBNT}
&\frac{ \partial_t \widetilde{v}_t - \mathcal{L}_{\beta_{s, p}}^* \widetilde{v}_t }{ \widetilde{v}_t^{1-\frac2q} \gamma^\frac2q }\nonumber \\
=&
\frac{q}{pp'} p (\beta_{s, p} - \beta) P_s[(\frac{v_t}{\gamma})^\frac1p] P_s[(\frac{v_t}{\gamma})^\frac1p \Delta \log\, \frac{v_t}{\gamma}] 
\\
&+
\frac{q}{pp'} \beta_{s, p} 
\left(
P_s[(\frac{v_t}{\gamma})^\frac1p] P_s[(\frac{v_t}{\gamma})^\frac1p \Gamma ( \log\, \frac{v_t}{\gamma})]
-
| P_s[(\frac{v_t}{\gamma})^\frac1p \nabla \log\, \frac{v_t}{\gamma}] |^2
\right).\notag
\end{align}
To be precise, we give two remarks about this identity. 
Firstly \cite[Lemma 4.2]{BNT} contains extra terms $I,N$ compared to \eqref{e:ClosureFromBNT}.   For this,  we have already mentioned that these terms identically vanish, see subsection 4.5 in \cite{BNT}.  
Secondly we imposed some technical decay assumption on $v_0$ in \cite[Lemma 4.2]{BNT} in order to ensure that terms in the right hand side of \eqref{e:ClosureFromBNT} are well-defined.  
Although these assumptions may not be satisfied for $v_0$ in Theorem \ref{t:RevRevClosureGene},  we can justify as follows. 

\begin{lemma}\label{l:Technical}
Let parameters and $v_0$ satisfy the assumption in Theorem \ref{t:RevRevClosureGene}. 
\begin{enumerate}
\item 
For each $t>0$,  $v_t$ is smooth and satisfies 
\begin{equation}\label{e:SemiLogAssumpHigh}
0 \le \nabla^2 \log\, \frac{v_t}{\gamma}(x) \le (1-\frac1\beta) {\rm id} 
\end{equation}
for all  $x \in \mathbb{R}^n$. 
\item 
For each $t>0$,  
$$
P_s\big[ \big( \frac{v_t}{\gamma} \big)^\frac1p\big],
\;
P_s\big[ \big( \frac{v_t}{\gamma} \big)^\frac1p \Delta \log\, \frac{v_t}{\gamma} \big],
\;
P_s\big[ \big( \frac{v_t}{\gamma} \big)^\frac1p  |\nabla\log\, \frac{v_t}{\gamma}|\big],
\;
P_s\big[ \big( \frac{v_t}{\gamma}  \big)^\frac1p \Gamma(\log\, \frac{v_t}{\gamma}) \big]
$$ 
are pointwisely well-defined. 
Here, $\Gamma(\log\, \frac{v_t}{\gamma})  \coloneqq | \nabla \log\, \frac{v_t}{\gamma}|^2$.
\item
For each $t>0$, we have that 
$$
P_s\big[ \big( \frac{v_t}{\gamma} \big)^\frac1p\big](x) 
\ge 
C_{s,p,t}
e^{\langle b_{s,p,t}, x\rangle},
\;\;\; x\in \mathbb{R}^n
$$
where 
$$
C_{s,p,t}\coloneqq 
(2\pi)^{\frac{n}{2p}}
v_t(0)^\frac1p 
e^{\frac{1-e^{-2s}}2 | \frac1p \nabla \log\, v_t(0) |^2} ,
\;
b_{s,p,t}
\coloneqq 
 \frac{e^{-s}}p \nabla \log\, v_t(0). 
$$
\end{enumerate}

\end{lemma}

\begin{proof}
The smoothness of $v_t$ for each fixed $t>0$ is clear from its explicit form \eqref{e:Sol20Oct}. 
For the statement \textit{(1)}, we recall that $u_t \coloneqq \frac{v_t}{\gamma_\beta}$ solves $\partial_t u_t = \mathcal{L}_{\beta} u_t \coloneqq \beta\Delta u_t - x \cdot \nabla u_t$ and  Ornstein--Uhlenbeck flow $(u_t)_{t>0}$ preserves the log-convexity and concavity, see \cite{EL,BraLi,IPT} for instance.  
These two facts are enough to ensure that $\frac{v_t}{\gamma}$ is log-convex and $1-\frac1\beta$-semi-log-concave in the sense of \eqref{e:DefConv} and \eqref{e:DefConc}, see also \cite[Lemma 3.2]{BNT}. 
Since $v_t$ is smooth on $\mathbb{R}^n$ for $t>0$, this means \eqref{e:SemiLogAssumpHigh}. 

For the statement \textit{(2)}, we first note a general fact that if $h:\mathbb{R}^n\to \mathbb{R}$ satisfies $\lambda_1{\rm id} \le \nabla^2 h \le \lambda_2 {\rm id}$ on $\mathbb{R}^n$ for some $\lambda_1,\lambda_2 \in \mathbb{R}$, then $h$ satisfies 
\begin{equation}\label{e:DecayFromConv}
\frac12\lambda_1 |x|^2 + \langle \nabla h(0), x\rangle + h(0) 
\le 
h(x)
\le 
\frac12\lambda_2 |x|^2 + \langle \nabla h(0), x\rangle + h(0) ,
\;\;\; x\in \mathbb{R}^n. 
\end{equation}
Hence it follows from \eqref{e:SemiLogAssumpHigh} that 
\begin{align}\label{e:vtPointwise}
&\langle \nabla \log\, v_t(0), x \rangle + \log\, v_t(0) + \log\, (2\pi)^{ \frac n2} 
\\
\le& 
\log\, \frac{v_t}{\gamma}(x) 
\notag
\\
\le &
\frac12(1 - \frac1\beta) |x|^2 + \langle \nabla \log\, v_t(0), x \rangle + \log\, v_t(0) + \log\, (2\pi)^{ \frac n2}
\notag 
\end{align}
for all $x \in \mathbb{R}^n$. 
In particular, we have that 
$$
v_t(x) 
\le 
e^{-\frac1{2\beta} |x|^2 + o_t(x)},\;\;\; o_t(x) \coloneqq  \langle \nabla \log\, v_t(0), x \rangle + \log\, v_t(0),
$$
from which 
$$
P_s\big[ \big( \frac{v_t}{\gamma} \big)^\frac1p\big](x)
\le 
C_{p,\beta} 
P_s\big[ 
\big(
\frac{\gamma_\beta}{\gamma}
\big)^\frac1p
e^{\frac1p o_t}
\big](x)
< 
\infty,
$$
see also Lemma \ref{l:GaussInt1} in Appendix \ref{S6.1} for the finiteness. 
Since $0\le \Delta \log\, \frac{v_t}{\gamma} \le n(1-\frac1\beta)$ from \eqref{e:SemiLogAssumpHigh}, we also obtain 
$$
P_s\big[ \big( \frac{v_t}{\gamma} \big)^\frac1p \Delta \log\, \frac{v_t}{\gamma} \big](x)
\le 
n(1-\frac1\beta) 
P_s\big[ \big( \frac{v_t}{\gamma} \big)^\frac1p\big](x)
< 
\infty. 
$$
For the rest of two, we notice from \eqref{e:SemiLogAssumpHigh} that 
$$
|\nabla \log\, \frac{v_t}{\gamma}(x)|
\le 
(1-\frac1\beta)|x|
+ 
|\nabla \log\, \frac{v_t}{\gamma}(0)|. 
$$
Hence 
\begin{align*}
&P_s\big[ \big( \frac{v_t}{\gamma} \big)^\frac1p |\nabla \log\, \frac{v_t}{\gamma}| \big](x)
+
P_s\big[ \big( \frac{v_t}{\gamma} \big)^\frac1p \Gamma( \log\, \frac{v_t}{\gamma}) \big](x)\\
&\le 
C_{t,\beta}
P_s\big[ \big( \frac{v_t}{\gamma} \big)^\frac1p (1+|\cdot|+|\cdot|^2) \big](x) <\infty. 
\end{align*}

For the statement \textit{(3)}, it suffices to combine the lower bound in \eqref{e:vtPointwise} and Lemma \ref{l:GaussInt1}. 
\end{proof}

\begin{proof}[Proof of Theorem \ref{t:RevRevClosureGene}]
Thanks to Lemma \ref{l:Technical} we can justify the identity \eqref{e:ClosureFromBNT} under the assumption of $v_0$ in Theorem \ref{t:RevRevClosureGene}. 
We focus on the second term in the right hand side in \eqref{e:ClosureFromBNT} and notice that 
\begin{align*}
&P_s[(\frac{v_t}{\gamma})^\frac1p](x) P_s[(\frac{v_t}{\gamma})^\frac1p \Gamma ( \log\, \frac{v_t}{\gamma})](x)
-
| P_s[(\frac{v_t}{\gamma})^\frac1p \nabla \log\, \frac{v_t}{\gamma}] |^2(x)
\\
=&
\frac{1}{A_x} \int_{\mathbb{R}^n} |\vec{\varphi}_x(y)|^2 \, d\mu(y)
\end{align*}
for fixed $x \in \mathbb{R}^n$ where 
$$
d\mu(y) 
\coloneqq 
(\frac{v_t}{\gamma})^\frac1p(y)\, d P_{s, x}(y)
=
(\frac{v_t}{\gamma})^\frac1p(y) 
\times 
\frac{1}{(2\pi(1-e^{-2s}))^\frac n2} e^{-\frac12 \frac{|y-e^{-s}x|^2}{1-e^{-2s}}}\, dy
$$
and 
$$
A_x \coloneqq P_s[(\frac{v_t}{\gamma})^\frac1p](x), 
\;
B_x \coloneqq P_s[(\frac{v_t}{\gamma})^\frac1p \nabla \log\, \frac{v_t}{\gamma}](x), 
\;
\vec{\varphi}_x(y) \coloneqq A_x \nabla \log\, \frac{v_t}{\gamma}(y) - B_x.
$$
We will invoke the Poincar\'{e} inequality on $\mathbb{R}^n$ with log-concave measure to handle $\int_{\mathbb{R}^n} |\vec{\varphi}_x(y)|^2 \, d\mu(y)$. 
For that purpose, we first focus on 
$$
\nabla^2 (-\log\, \frac{d\mu}{dy}) = - \frac1p \nabla^2(\log\, v_t - \log\, \gamma) + \frac{1}{1-e^{-2s}} \mathrm{id}, 
$$
and appeal to \eqref{e:SemiLogAssumpHigh} to see that 
%$$
%\nabla^2 \log\, v_t \le - \frac{1}{\beta} \mathrm{id}, 
%$$
%and hence 
\begin{align*}
\nabla^2 (-\log\, \frac{d\mu}{dy})
%=&
%- \frac1p \nabla^2 (\log\, v_t - \log\, \gamma) + \frac{1}{1-e^{-2s}} \mathrm{id}
%\\
\ge& 
( - \frac1p(1-\frac1\beta) + \frac{1}{1-e^{-2s}}) \mathrm{id}
\eqqcolon K_{s,p, \beta} \mathrm{id}
\end{align*}
since $p>0$. 
Notice that $K_{s, p, \beta}>0$. Indeed, it follows from $\frac{q-1}{p-1}=e^{2s}$ and \eqref{e:beta_sp} that 
\begin{align}
K_{s, p,\beta}
=&
%- \frac1p(1-\frac1\beta) + \frac{1}{1-e^{-2s}}
%\notag
%\\
%=&
%\frac{p\beta - (\beta-1)(1-e^{-2s})}{p\beta(1-e^{-2s})}
%\notag
%\\
%=&
\frac{\beta - (\beta-1)\frac{1-e^{-2s}}{p}}{\beta(1-e^{-2s})}
%\notag
%\\
%=&
=
\frac{\beta_{s,p}}{\beta(1-e^{-2s})} 
> 0
\label{e:CalculationK}
\end{align}
where we used the fact that 
\begin{equation}\label{e:CalculationBeta}
\beta_{s,p} 
%=
%1 + (\beta-1)\frac{q}{p} e^{-2s}
=
1 + (\beta-1)\frac{(p-1)e^{2s} + 1}{p} e^{-2s}
=
\beta - (\beta-1)\frac{1-e^{-2s}}{p}.
\end{equation}
Next we focus on the average of $\vec{\varphi}_x$ and notice that 
$$
\int_{\mathbb{R}^n} (\vec{\varphi}_x)_i(y) \, d\mu(y)
= 
A_x (B_x)_i - (B_x)_i A_x = 0 
$$
for each $i =1,2, \dots, n$, where $(\vec{\varphi}_x)_i$ and $(B_x)_i$ are the $i$-th component of $\vec{\varphi}_x$ and $B_x$ respectively.
Therefore we may invoke the Poincar\'{e} inequality,  see \cite[Corollary 4.8.2]{BGL} for instance. 
Namely we apply the Poincar\'{e} inequality for each component of $\vec{\varphi}_x$ with respect to the reference measure $d\mu$ to see that 
\begin{align*}
&P_s[(\frac{v_t}{\gamma})^\frac1p](x) P_s[(\frac{v_t}{\gamma})^\frac1p \Gamma ( \log\, \frac{v_t}{\gamma})](x)
-
| P_s[(\frac{v_t}{\gamma})^\frac1p \nabla \log\, \frac{v_t}{\gamma}] |^2(x)
\\
&=
%\frac{1}{A_x} \int_{\mathbb{R}^n} |\vec{\varphi}_x|^2 \, d\mu
%\\
%&=
\frac{1}{A_x} \sum_{i=1}^n \int_{\R^n} |(\vec{\varphi}_x)_i|^2 \, d\mu
\\
&\le
\frac{1}{A_x K_{s,p,\beta}} \sum_{i=1}^n \int_{\R^n} |\nabla( \vec{\varphi}_x)_i|^2 \, d\mu
\\
%&= 
%\frac{1}{A_x K_{s,p,\beta}} \sum_{i=1}^n \sum_{j=1}^n \int_{\R^n} |\partial_j ( \vec{\varphi}_x)_i|^2 \, d\mu
%\\
&= 
\frac{A_x}{ K_{s,p,\beta}} \int_{\R^n} \sum_{i=1}^n \sum_{j=1}^n |\partial_j \partial_i \log\, \frac{v_t}{\gamma}|^2 (\frac{v_t}{\gamma})^\frac1p \, dP_{s,x}
\\
&= 
\frac{A_x}{ K_{s,p,\beta}} \int_{\mathbb{R}^n} \|\nabla^2 \log\, \frac{v_t}{\gamma} \|_{\mathrm{HS}}^2 (\frac{v_t}{\gamma})^\frac1p \, dP_{s,x}.
\end{align*}
Here $\|\cdot\|_{\rm HS}$ denotes the Hilbert--Schmidt norm. 
Now let $\lambda_1(y), \dots, \lambda_n(y)$ be eigenvalues of $\nabla^2 \log \frac{v_t}{\gamma}(y)$ for given $y \in \mathbb{R}^n$. Then it follows from the bounds of \eqref{e:SemiLogAssumpHigh} that 
$$
0 \le \lambda_1(y), \dots, \lambda_n(y) \le 1- \frac1\beta, 
$$
and hence 
\begin{align*}
\|\nabla^2 \log\, \frac{v_t}{\gamma}(y) \|_{\mathrm{HS}}^2
=&
\sum_{i=1}^n \lambda_i(y)^2
\le
(1-\frac1\beta) \sum_{i=1}^n \lambda_i(y)
%=&
%(1-\frac1\beta) \mathrm{Tr} \nabla^2 \log\, \frac{v_t}{\gamma}(y)
=
(1-\frac1\beta) \Delta \log\, \frac{v_t}{\gamma}(y).
\end{align*}
This reveals that 
\begin{align*}
&P_s[(\frac{v_t}{\gamma})^\frac1p](x) P_s[(\frac{v_t}{\gamma})^\frac1p \Gamma ( \log\, \frac{v_t}{\gamma})](x)
-
| P_s[(\frac{v_t}{\gamma})^\frac1p \nabla \log\, \frac{v_t}{\gamma}] |^2(x)
\\
&\le
\frac{A_x}{ K_{s,p,\beta}} (1-\frac1\beta) \int_{\mathbb{R}^n} (\Delta \log\, \frac{v_t}{\gamma}) (\frac{v_t}{\gamma})^\frac1p \, dP_{s,x}
\\
&=
\frac{1}{ K_{s,p,\beta}} (1-\frac1\beta) 
P_s[ (\frac{v_t}{\gamma})^\frac1p](x) P_s[(\frac{v_t}{\gamma})^\frac1p \Delta \log\, \frac{v_t}{\gamma}](x).
\end{align*}
We then insert this to \eqref{e:ClosureFromBNT} to conclude the proof. 
In the case  of $q/(pp')>0$ which corresponds to $0 < p < 1-e^{-2s}$,  we obtain 
\begin{align*}
&\frac{ \partial_t \widetilde{v}_t - \mathcal{L}_{\beta_{s, p}}^* \widetilde{v}_t }{ \widetilde{v}_t^{1-\frac2q} \gamma^\frac2q }\\
&\le
\frac{q}{pp'} 
\left(
p (\beta_{s, p} - \beta) + \beta_{s,p} \frac{1}{K_{s,p,\beta}}(1-\frac1\beta)
\right)
P_s[(\frac{v_t}{\gamma})^\frac1p] P_s[(\frac{v_t}{\gamma})^\frac1p \Delta \log\, \frac{v_t}{\gamma}]. 
\end{align*}
%and in the case of $q/(pp') <0$,  
%\begin{align*}
%&\frac{ \partial_t \widetilde{v}_t - \mathcal{L}_{\beta_{s, p}}^* \widetilde{v}_t }{ \widetilde{v}_t^{1-\frac2q} \gamma^\frac2q }\\
%&\ge
%\frac{q}{pp'} 
%\left(
%p (\beta_{s, p} - \beta) + \beta_{s,p} \frac{1}{K_{s,p,\beta}}(1-\frac1\beta)
%\right)
%P_s[(\frac{v_t}{\gamma})^\frac1p] P_s[(\frac{v_t}{\gamma})^\frac1p \Delta \log\, \frac{v_t}{\gamma}].
%\end{align*}
By direct calculations with \eqref{e:CalculationK} and \eqref{e:CalculationBeta}, we conclude that 
\begin{align*}
p (\beta_{s, p} - \beta) + \beta_{s,p} \frac{1}{K_{s,p,\beta}}(1-\frac1\beta)
=
- (\beta-1) (1-e^{-2s}) + \beta(1-e^{-2s})(1-\frac1\beta)
=0.
\end{align*}
The argument for the case of $q/(pp') <0$ which corresponds to $1-e^{-2s} < p <1$ is similar and we omit it. 
\end{proof}

Arming Theorem \ref{t:RevRevClosureGene}, we complete the proof of Theorem \ref{t:FHC}. 

\begin{proof}[Proof of Theorem \ref{t:FHC}]
For the sake of simplicity, let us consider the case $p \in (0,1-e^{-2s})$ in which case we have $q<0$. 
If $\beta$ is such that $\big\| P_s \big[ \big(\frac{\gamma_\beta}{\gamma} \big)^\frac1p\big] \big\|_{q} = \infty$, then there is nothing to prove so we assume $\big\| P_s \big[ \big(\frac{\gamma_\beta}{\gamma} \big)^\frac1p\big] \big\|_{q} < \infty$. 
Thanks to \eqref{e:PsGauss2} this means $\beta$ satisfies \eqref{e:beta_sp}. 
With this in mind we take $f$ satisfying the assumption in Theorem \ref{t:FHC} and let $v_0 \coloneqq f \gamma$. 
Note that this $v_0$ satisfies the assumption in Theorem \ref{t:RevRevClosureGene}.  We then let $v_t$ be the $\beta$-Fokker--Planck solution to \eqref{e:BetaFP}. 
Fix $R>0$ which tends to $\infty$ and define 
$$
Q(t)
\coloneqq
\int_{\mathbb{R}^n}
\widetilde{v}_t(x) \, dx,
\;\;\;
Q^R(t)
\coloneqq
\int_{\mathbb{R}^n}
\widetilde{v}_t(x) \chi_R(x)\, dx,
\;\;\;
t>0
$$
where $\chi_R \coloneqq \chi(\frac{\cdot}{R})$ and $\chi$ is a smooth cut off function supported on $[-2,2]^n$ and identically $1$ on $[-1,1]^n$. 
We intend to show the monotonicity of $Q(t)$. 
If we notice that $\partial_t \widetilde{v}_t $ is continuous on $\mathbb{R}^n$, then 
$$
\int_{\mathbb{R}^n} 
|\partial_t \widetilde{v}_t(x)|
\chi_R(x)\, dx \le 
CR^n 
\sup_{x\in [-2R,2R]^n } |\partial_t \widetilde{v}_t(x)|
<\infty. 
$$
Hence we may justify the interchange of the derivative and the integration to see that 
\begin{align*}
\frac{d}{dt} Q^R(t)
&= 
\int_{\mathbb{R}^n}
\partial_t \widetilde{v}_t \cdot \chi_R\, dx \\
&= 
\int_{\mathbb{R}^n}
\big( 
\partial_t \widetilde{v}_t 
-
\mathcal{L}_{\beta_{s,p}}^*  \widetilde{v}_t
\big)\chi_R\, dx
+
\int_{\mathbb{R}^n}
\mathcal{L}_{\beta_{s,p}}^*  \widetilde{v}_t
\cdot 
\chi_R\, dx. 
\end{align*}
By invoking Theorem \ref{t:RevRevClosureGene} and the duality of $\mathcal{L}_{\beta_{s,p}}^*$, 
\begin{align*}
\frac{d}{dt} Q^R(t)
&\le 
\int_{\mathbb{R}^n}
\widetilde{v}_t
\cdot 
\mathcal{L}_{\beta_{s,p}} \chi_R\, dx\\
&= 
\int_{\mathbb{R}^n}
\widetilde{v}_t
\big(
\beta_{s,p} \Delta \chi_R 
-\langle x, \nabla \chi_R\rangle
\big)
\, dx\\
&= 
\frac1R
\int_{\mathbb{R}^n}
\widetilde{v}_t
\big(
\frac{\beta_{s,p} }{R} (\Delta \chi ) (\frac{x}{R}) 
-\langle x,  (\nabla \chi) (\frac{x}{R})\rangle
\big)
\, dx\\
&\le 
\frac{C_\chi }{R}
\int_{\mathbb{R}^n}
(1+|x|)
\widetilde{v}_t\, dx. 
\end{align*}
In view of $q<0$, we apply Lemma \ref{l:Technical}-\textit{(3)} to ensure that $\int (1+|x|) \widetilde{v}_t\, dx \le C_t <\infty$. 
With this in mind,  for arbitrary $T>0$, 
$$
Q^R(T) - Q^R(0) 
=
\int_0^T 
\frac{d}{dt}Q^R(t)\, dt
\le 
\frac{1 }{R} C_\chi\sup_{0\le t\le T} C_t
\to 0\;\;\;(R\to\infty). 
$$
By virtue of the monotone convergence theorem, we know that $Q(t) = \lim_{R\to \infty} Q^R(t)$ and hence we obtain 
$$
Q(T) \le Q(0) 
$$
for all $T>0$. 
Since $v_t$ is the $\beta$-Fokker--Planck solution, we know that $\lim_{T\to \infty} v_T = (\int_{\mathbb{R}^n} v_0\, dx)\gamma_{\beta}$ and that 
$$
\liminf_{T\to\infty} Q(T) \ge 
\int_{\mathbb{R}^n}
P_s \big[ 
\big(
\frac{\gamma_\beta}{\gamma}
\big)^\frac1p
\big]^q\, d\gamma
\big(
\int_{\mathbb{R}^n} 
v_0\, dx 
\big)^{\frac{q}{p}}
$$
by Fatou's lemma. 
This concludes 
$$
\int_{\mathbb{R}^n}
P_s \big[ 
\big(
\frac{\gamma_\beta}{\gamma}
\big)^\frac1p
\big]^q\, d\gamma
\big(
\int_{\mathbb{R}^n} 
v_0\, dx 
\big)^{\frac{q}{p}}
\le 
\int_{\mathbb{R}^n}
P_s \big[ 
\big(
\frac{ v_0 }{\gamma}
\big)^\frac1p
\big]^q\, d\gamma. 
$$

In the case $p > 1-e^{-2s}$, if one notices that $\beta_{s,p}>0$ regardless of $\beta\ge1$ then the similar  proof above can be applied. 
\end{proof}

\section{Applications to convex geometry}\label{S4}
\subsection{Functional form of Blaschke--Santal\'{o} inequality and   inverse Santal\'{o} inequality}\label{S4.1}

Both the Blaschke--Santal\'{o} inequality and the inverse Santal\'{o} inequality are investigated in more general framework as functional inequalities. 
This direction of the study was initiated by Ball \cite{BallPhd} where he formulated a functional form of the Blaschke--Santal\'{o} inequality as follows. 
For any Borel function $\psi$ on $\mathbb{R}^n$ with $\int_{\mathbb{R}^n} e^{-\psi(x)}\, dx <\infty$ and centrally symmetric,  
\begin{equation}\label{e:Lehec2}
\int_{\mathbb{R}^n} e^{-\psi(x)}\, dx \int_{\mathbb{R}^n} e^{- \psi^*(x)}\, dx \le \left(\int_{\mathbb{R}^n} e^{- \frac{1}{2} |x|^2}\, dx \right)^2=(2\pi)^n.
\end{equation}
Here $\psi^*$ is the Legendre transformation of $\psi$ defined by 
$$
\psi^*(x) \coloneqq \sup_{y \in \mathbb{R}^n} [ \langle x, y\rangle - \psi(y)], \;\;\; x \in \mathbb{R}^n. 
$$
Later the symmetric assumption was weakened to $\int_{\mathbb{R}^n} x e^{-\psi(x)}\, dx = 0$ by Artstein-Avidan--Klartag--Milman \cite{AKM}.  
\begin{theorem}[Artstein-Avidan--Klartag--Milman \cite{AKM}, Lehec \cite{LehecDirect}]\label{t:BSBaryCen19Oct}
Let a Borel function $\psi:\mathbb{R}^n\to \mathbb{R}$ be such that $|x| e^{-\psi(x)} \in L^1(dx)$. 
If $e^{-\psi}$ is barycenter zero in the sense that $\int_{\mathbb{R}^n} x e^{-\psi(x)}\, dx =0$, then \eqref{e:Lehec2} holds true.  
\end{theorem}
A passage from this functional form to the original inequality \eqref{e:BS} is to focus on the gauge function. 
In fact, using the property 
\begin{equation}\label{e:PathSetFunc}
\big( \frac12 \|\cdot\|_K^2 \big)^* = \frac12\|\cdot\|_{K^\circ}^2,
\;\;\; 
\int_{\mathbb{R}^n} e^{- \frac12 \| x \|_K^2}\, dx 
= 
\frac{(2\pi)^\frac{n}{2}}{|{\rm B}_2^n|} |K|,
\end{equation}
see \cite{AKM} and \cite[p.55]{Sch} for these,  one can derive \eqref{e:BS} from its functional form \eqref{e:Lehec2}.  
We refer works by Fathi \cite{Fathi} and second author \cite{Tsuji} for the relation to some improved version of Talagrand's transport-cost inequality.  
We will rederive Theorem \ref{t:BSBaryCen19Oct} from our Theorem \ref{t:NelsonGene} in Subsection \ref{S4.4}. 
Similarly one can consider the functional form of the inverse Santal\'{o} inequality and this was formulated by Fradelizi--Meyer \cite{FraMeyPo08,FM} as follows. 
\begin{conjecture}
For any convex function $\psi:\mathbb{R}^n\to\mathbb{R}$, 
\begin{equation}\label{e:FuncInvSantaloGene}
\int_{\mathbb{R}^n} e^{-\psi(x)}\, dx \int_{\mathbb{R}^n} e^{- \psi^*(x)}\, dx \ge e^n.
\end{equation}
For any convex and centrally symmetric function $\psi:\mathbb{R}^n \to \mathbb{R}$, 
\begin{equation}\label{e:FuncInvSantalo}
\int_{\mathbb{R}^n} e^{-\psi(x)}\, dx \int_{\mathbb{R}^n} e^{- \psi^*(x)}\, dx \ge 4^n.
\end{equation}

\end{conjecture}
These functional forms of Mahler's conjecture were proved for $n =1$ in \cite{FM} (both of \eqref{e:FuncInvSantalo} and \eqref{e:FuncInvSantaloGene}), $n=2$ in \cite{FN} (symmetric case \eqref{e:FuncInvSantalo}), and the unconditional case for all dimensions in \cite{FraMeyPo08}.  
%We also note that the functional form of Mahler's conjecture for non symmetric case was also formulated in \cite{FM} and it predicts the inequality \eqref{e:FuncInvSantalo} with lower bound $e^n$ instead of $4^n$. 
We will argue on this functional form rather than inequalities for convex bodies from now on. 

\subsection{An idea of the proof of Proposition \ref{Prop:HCSet}: Hamilton--Jacobi flow}\label{S4.2}
Our link bridging hypercontractivity and Blaschke--Santal\'{o} and inverse Santal\'{o} inequalities is according to a simple observation on  Hamilton--Jacobi flow. 
For a measurable function $\phi :\mathbb{R}^n\to \mathbb{R}$, we define $Q_t \phi$, $t>0$, by 
\begin{equation}\label{e:DefHJ}
Q_t\phi(x) \coloneqq \inf_{y \in \mathbb{R}^n} [ \phi(y) + \frac{|x-y|^2}{2t} ], \;\;\; x \in \mathbb{R}^n.
\end{equation}
This $Q_t \phi$ is known to formally solve the Hamilton--Jacobi equation
$$
\begin{cases}
\partial_t u + \frac12 |\nabla u|^2 = 0,\;\;\; &(t,x) \in (0,\infty)\times \mathbb{R}^n,\\
u(0,x) = \phi(x),\;\;\;&x \in \mathbb{R}^n, 
\end{cases}
$$
in an appropriate sense depending on the regularity of $\phi$. 
There are large number of references concerning the theory of the Hamilton--Jacobi equation and we refer to Evans's book \cite{Evans}. 
For our purpose in this paper, the relevant notion is so-called vanishing viscosity method which reveals the link between $P_sf$ and $Q_t\phi$, see \cite{BGL}. In fact based on this argument, Bobkov--Gentil--Ledoux \cite{BGLJMPA} observed an equivalence between hypercontractivity for Ornstein--Uhlenbeck flow and Hamilton--Jacobi flow. 
This link can be seen by letting  
$$
u^\varepsilon
\coloneqq
-2\varepsilon
\log\, P_\varepsilon \big[ e^{-\frac{\phi}{2\varepsilon}} \big]
$$
for sufficiently regular $\phi\colon\mathbb{R}^n\to\mathbb{R}$ and taking a limit $\varepsilon\downarrow0$. 
For instance Bobkov--Gentil--Ledoux observed that 
\begin{equation}\label{e:VaniVis}
\lim_{\varepsilon \downarrow 0} 
u^\varepsilon
= 
Q_1 \phi 
\end{equation}
in some appropriate sense if $\phi$ is bounded and continuous. 

An another observation is to regard the functional Blaschke--Santal\'{o} inequality as an improved  hypercontractivity for $Q_t\phi$. 
From the definition,  we have a simple identity between $Q_1$ and the Legendre transform 
\begin{equation}\label{e:Qto*}
Q_1\phi(x) = - (\phi + \frac{1}{2} |\cdot|^2)^*(x) + \frac{1}{2}|x|^2, \;\;\; x \in \mathbb{R}^n
\end{equation}
%In fact, for fixed $x \in \mathbb{R}^n$, it follows from 
%\begin{align*}
%Q_1f(x) 
%=&
%\inf_{y \in \mathbb{R}^n} [ f(y) + \frac{|y|^2}{2} - \langle x, y\rangle ] + \frac{1}{2}|x|^2
%\\
%= &
%- (f + \frac{1}{2} |\cdot|^2)^*(x) + \frac{1}{2}|x|^2.
%\end{align*}
%Hence, replacing $\varphi$ by $f + \frac{1}{2}|\cdot|^2$, \eqref{e:Lehec2} can be also read as 
from which we derive 
\begin{equation}\label{e:LinkHJLege}
\int_{\mathbb{R}^n} e^{-\psi}\, dx \int_{\mathbb{R}^n} e^{-\psi^*}\, dx 
= 
(2\pi)^n 
\int_{\mathbb{R}^n} e^{-\phi}\, d\gamma \int_{\mathbb{R}^n} e^{Q_1 \phi}\, d\gamma,
\;\;\;
{\rm for}
\;\;\;   
\psi = \phi + \frac12|\cdot|^2. 
\end{equation}
Therefore \eqref{e:Lehec2} is equivalent to 
\begin{equation}\label{e:HJBS}
\big\| e^{Q_1\phi} \big\|_{L^1(\gamma)} \le \big\| e^\phi \big\|_{L^{-1}(\gamma)}
\end{equation}
for all symmetric $\phi$. 
This \eqref{e:HJBS} can be compared to the hypercontractivity inequality for Hamilton--Jacobi flow due to Bobkov--Gentil--Ledoux \cite{BGLJMPA} which states that 
\begin{equation}\label{e:HCHJClassic}
\big\| e^{Q_1\phi} \big\|_{L^0(\gamma)} \le \big\| e^\phi \big\|_{L^{-1}(\gamma)}
\end{equation}
for all bounded and continuous function $\phi$.  Clearly \eqref{e:HJBS} improves \eqref{e:HCHJClassic} via $L^1(\gamma) \subset L^0(\gamma)$ by virtue of the symmetry of $\phi$. 
%{\color{red}?? This relation between \eqref{e:Lehec2} and \eqref{e:HJBS} might be folklore but we cannot find appropriate reference pointing out this link explicitly. }
Similarly one can regard the functional inverse Santal\'{o} inequality \eqref{e:FuncInvSantaloGene} as the reverse hypercontractivity inequality for Hamilton--Jacobi flow
$$
\big\| e^{Q_1\phi} \big\|_{L^1(\gamma)} 
\ge 
\big(
\frac{e}{2\pi}
\big)^n
\big\| e^{\phi} \big\|_{L^{-1}(\gamma)}
$$
for all $\phi$. 
Now the relation between $P_sf$, $Q_1\phi$, and $\psi^*$ are clarified.  
We will then give more detailed argument to prove Proposition \ref{Prop:HCSet} in the following subsection. 

\if0
Here note that $f$ also satisfies $\int_{\mathbb{R}^n} e^{-f}\, d\gamma <\infty$ and $\int_{\mathbb{R}^n} x e^{-f}\, d\gamma =0$.
\eqref{e:HJBS} is also represented as 
\begin{equation}\label{e:1-HJHCBS}
\| e^{Q_1 f} \|_{L^1(\gamma)} \le \| e^f \|_{L^{-1}(\gamma)}.
\end{equation}
More generally, for any $t>0$, since it holds $t Q_t(f) = Q_1(tf)$, we enjoy 
\begin{equation}\label{e:t-HJHCBS}
\| e^{Q_t f} \|_{L^t(\gamma)} \le \| e^f \|_{L^{-t}(\gamma)}.
\end{equation}
\eqref{e:t-HJHCBS} is related to the hypercontractivity of the Hamilton--Jacobi flow which is investigated very well in differential geometry, in particular under curvature-dimension condition. 
In general, it is known that hypercontractivity of the Hamilton--Jacobi flow is reduced from the reverse hypercontractivity of heat flow. 
Hence, now we have a natural question: Can we also reduce \eqref{e:t-HJHCBS} from a reverse hypercontractivity type inequality of a heat flow? If it can be done, what is that inequality? 
Our first result give an answer for this question. 
\fi 

\subsection{Proof of Proposition \ref{Prop:HCSet}-{\textit{(1)}}}\label{S4.3}
As we explained, it suffices to show the functional form \eqref{e:Lehec2} to obtain the set form \eqref{e:BS}. To show Proposition \ref{Prop:HCSet}-{\textit{(1)}}, we first establish the following. 
\begin{proposition}\label{Prop:HCBS_2}
Let $\psi:\mathbb{R}^n\to\mathbb{R}$ be continuous and satisfy 
\begin{align}\label{e:VanVisCond}
\psi(x) - \frac12|x|^2 \ge - C ( 1 + |x|), \;\;\; x \in \mathbb{R}^n
\end{align}
for some $C>0$. 
If \eqref{e:RHC_BSIntro} holds for $f=f_s$ defined by $\log\, f_s (x)\coloneqq \frac{p_s}{2s} ( \frac12 |x|^2 - \psi(x))$ for all small $s>0$, then 
\begin{equation}\label{e:FuncBS19Oct}
\int_{\mathbb{R}^n} e^{-\psi}\, dx \int_{\mathbb{R}^n} e^{-\psi^*}\, dx 
\le 
(2\pi)^n (\limsup_{s\downarrow 0} {\rm BS}_s)^{-1}. 
\end{equation}
\end{proposition}
\begin{proof}
In view of the relation \eqref{e:LinkHJLege}, it suffices to show 
\begin{equation}\label{e:Goal19Oct}
\big\| e^{Q_1\phi} \big\|_{L^1(\gamma)} \le (\limsup_{s\downarrow 0} {\rm BS}_s)^{-1} \big\| e^\phi \big\|_{L^{-1}(\gamma)}
\end{equation}
for $\phi = \psi - \frac12|\cdot|^2$. 
To this end we appeal to the vanishing viscosity argument due to Bobkov--Gentil--Ledoux \cite{BGLJMPA}. 
For $\varepsilon>0$ we let 
$$
u^\varepsilon
\coloneqq
-2\varepsilon
\log\, P_\varepsilon \big[ e^{-\frac{\phi}{2\varepsilon}} \big]. 
$$
%and take $a_\varepsilon>0$ such that 
%$$
%p_\varepsilon \coloneqq 2\varepsilon a_\varepsilon = 1- e^{-2\varepsilon}.
%$$
Then the assumption \eqref{e:RHC_BSIntro} with $f_\varepsilon = e^{-\frac{p_\varepsilon}{2\varepsilon}\phi}$ ensures that 
$$
\big\| P_\varepsilon \big[ e^{-\frac{\phi}{2\varepsilon}} \big] \big\|_{L^{q_\varepsilon}(\gamma)}
\ge 
{\rm BS}_\varepsilon^{\frac1{1-e^{-2\varepsilon}}}
\big(
\int_{\mathbb{R}^n} e^{ - \frac{1-e^{-2\varepsilon}}{2\varepsilon} \phi}\, d\gamma
\big)^{\frac1{1-e^{-2\varepsilon}}}
$$
which can be read as 
$$
\big\| P_\varepsilon \big[ e^{-\frac{\phi}{2\varepsilon}} \big]^{-2\varepsilon} \big\|_{L^{\frac{-q_\varepsilon}{2\varepsilon}}(\gamma)}
\le 
{\rm BS}_\varepsilon^{-\frac{2\varepsilon}{1-e^{-2\varepsilon}}}
\big(
\int_{\mathbb{R}^n} e^{ - \frac{1-e^{-2\varepsilon}}{2\varepsilon} \phi}\, d\gamma
\big)^{-\frac{2\varepsilon}{1-e^{-2\varepsilon}}}.
$$
Taking the limit $\varepsilon \downarrow 0$, it follows that 
$$
\liminf_{\varepsilon \downarrow 0} 
\big\| P_\varepsilon \big[ e^{-\frac{\phi}{2\varepsilon}} \big]^{-2\varepsilon} \big\|_{L^{\frac{-q_\varepsilon}{2\varepsilon}}(\gamma)}
\le 
(\limsup_{\varepsilon \downarrow 0} 
{\rm BS}_\varepsilon)^{-1}
\big\| e^\phi \big\|_{L^{-1}(\gamma)}. 
$$
For the left hand side, we note that $\phi$ satisfies 
$$
\phi(x) \ge -C(1+|x|)
$$
for some $C>0$ by virtue of the assumption on $\psi$. 
Thanks to this, we can ensure 
\begin{align}\label{e:LiminfQ_1phi}
\liminf_{\varepsilon \downarrow 0} u^\varepsilon \ge Q_1\phi
\end{align}
from \cite[Theorem 5.1]{ABI} and the argument discussed in \cite[p.43]{BNT}.
Since $\frac{-q_\varepsilon}{2\varepsilon} \to 1$ from the assumption $q_\varepsilon = -2\varepsilon + O(\varepsilon^2)$ and $P_\varepsilon \big[ e^{-\frac{\phi}{2\varepsilon}} \big]^{-2\varepsilon} = e^{u^\varepsilon}$, Fatou's lemma yields that 
$$
\liminf_{\varepsilon \downarrow 0} 
\big\| P_\varepsilon \big[ e^{-\frac{\phi}{2\varepsilon}} \big]^{-2\varepsilon} \big\|_{L^{\frac{-q_\varepsilon}{2\varepsilon}}(\gamma)}
\ge 
\int_{\mathbb{R}^n} e^{ \liminf_{\varepsilon \downarrow 0} u^ \varepsilon}\, d\gamma 
\ge 
\int_{\mathbb{R}^n} e^{Q_1\phi}\, d\gamma, 
$$
which concludes \eqref{e:Goal19Oct}. 
\end{proof}

While Proposition \ref{Prop:HCBS_2} formally yields \eqref{e:Lehec2} for \textit{symmetric} $\psi$ satisfying \eqref{e:VanVisCond}, one needs extra argument to complete the proof of Proposition \ref{Prop:HCSet}-\textit{(1)}. 

\begin{proof}[Proof of Proposition \ref{Prop:HCSet}-\textit{(1)}]
In view of Proposition \ref{Prop:HCBS_2} and the fact that $x\mapsto \|x\|_K$ is continuous,  our task is to ensure that we may assume \eqref{e:VanVisCond} without loss of generality.  
However, this can be justified by approximating $\psi = \frac12 \|\cdot\|_K^2$ or equivalently $\phi(x) = \frac12\|x\|_K^2 - \frac12|x|^2$ by 
$
\phi_k(x) \coloneqq \max \{ \phi(x), -k\}
$,  see also Subsection \ref{S4.4} for further details. 
\end{proof}
%While a combination of Theorem \ref{t:NelsonGene} and Proposition \ref{Prop:HCSet} yields \eqref{e:Lehec2} for \textit{symmetric} $\psi$ satisfying \eqref{e:VanVisCond}, one needs extra argument to complete the proof of the implication of Theorem \ref{t:BSBaryCen19Oct}. 
%The first task is to get rid of the assumption \eqref{e:VanVisCond} and the second task is to relax the symmetric assumption to the barycenter assumption.  

%\begin{claim}
%Let $s>0$ and $\phi \colon \mathbb{R}^n \to \mathbb{R}$ be a continuous function such that $|x| e^{-s\phi} \in L^1(\gamma)$. 
%If $\int_{\mathbb{R}^n} x e^{-s\phi}\, d\gamma=0$, then 
%\begin{align}\label{e:HJHC}
%\| e^{Q_s\phi} \|_{L^s(\gamma)} \le \| e^{\phi} \|_{L^{-s}(\gamma)}.
%\end{align}
%Equivalently, if $\psi$ is continuous function on $\R^n$ with $|x| e^{-\psi} \in L^1(dx)$ and $\int_{\mathbb{R}^n} x e^{-\psi}\, dx = 0$, then 
%\begin{align}\label{e:FuncBS}
%\int_{\mathbb{R}^n} e^{-\psi}\, dx \int_{\mathbb{R}^n} e^{-\psi^*}\, dx \le (2\pi)^n.
%\end{align}
%\end{claim}

\subsection{Rederiving Theorem \ref{t:BSBaryCen19Oct} from Theorem \ref{t:NelsonGene}}\label{S4.4}
Let us show that our Theorem \ref{t:NelsonGene} implies Theorem \ref{t:BSBaryCen19Oct}.
To this end we need to involve further argument as follows. 

First we remark that it suffices to show \eqref{e:Lehec2} only for a continuous function $\psi$. 
To see this, let $\psi$ be a Borel function on $\mathbb{R}^n$ with $ e^{-\psi}, |x| e^{-\psi} \in L^1(dx)$ and $\int_{\mathbb{R}^n} x e^{-\psi}\, dx = 0$.
We define a function $\psi_\varepsilon$ on $\mathbb{R}^n$ for $\varepsilon>0$ as 
$$
\psi_\varepsilon(x) \coloneqq - \log \int_{ \mathbb{R}^n } e^{-\psi(y) }\gamma_{2\varepsilon}(x-y)\, dy, \;\;\; x \in \mathbb{R}^n.
$$
Then $\psi_\varepsilon$ is continuous and satisfies $ e^{-\psi_\varepsilon}, |x| e^{-\psi_\varepsilon} \in L^1(dx)$. Moreover, we can observe that 
$$
\int_{\mathbb{R}^n} x e^{-\psi_\varepsilon}\, dx = \int_{\mathbb{R}^n} x e^{-\psi}\, dx=0.
$$
Hence we can apply \eqref{e:Lehec2} for continuous inputs to see that 
\begin{equation}\label{e:BorelApp}
\int_{\mathbb{R}^n} e^{-\psi_\varepsilon(x)}\, dx \int_{\mathbb{R}^n} e^{- (\psi_\varepsilon)^*(x)}\, dx \le (2\pi)^n.
\end{equation}
Now we have 
\begin{equation}\label{e:BorelApp1}
\int_{\mathbb{R}^n} e^{-\psi_\varepsilon}\, dx = \int_{\mathbb{R}^n} e^{-\psi}\, dx.
\end{equation} 
On the other hand, it follows from the definition of the Legendre transform that 
$$
\psi^*(x) \ge \langle x, y \rangle - \psi(y), \;\;\; \forall x, y \in \mathbb{R}^n
$$
which implies that for any $x \in \mathbb{R}^n$, 
\begin{align*}
(\psi_\varepsilon)^*(x)
=&
\sup_{z \in \mathbb{R}^n} \left[
\langle z, x \rangle + \log \int_{ \mathbb{R}^n } e^{-\psi(y) }\gamma_{2\varepsilon}(z-y)\, dy
\right]
\\
\le&
\sup_{z \in \mathbb{R}^n} \left[
 \langle z, x \rangle + \log \int_{ \mathbb{R}^n } e^{\psi^*(x) - \langle x, y \rangle }\gamma_{2\varepsilon}(z-y)\, dy 
 \right]
\\
=&
\psi^*(x)
+
\varepsilon |x|^2.
\end{align*}
Hence we have 
$$
\int_{\mathbb{R}^n} e^{-\psi^*(x) - \varepsilon |x|^2}\, dx 
\le 
\int_{\mathbb{R}^n} e^{- (\psi_\varepsilon)^*(x)}\, dx, 
$$
and thus letting $\varepsilon \downarrow 0$, we enjoy 
\begin{equation}\label{e:BorelApp2}
\int_{\mathbb{R}^n} e^{-\psi^*(x)}\, dx 
\le 
\liminf_{\varepsilon \downarrow 0} \int_{\mathbb{R}^n} e^{- (\psi_\varepsilon)^*(x)}\, dx. 
\end{equation} 
Combining \eqref{e:BorelApp} with \eqref{e:BorelApp1} and \eqref{e:BorelApp2}, we conclude \eqref{e:Lehec2} for $\psi$.

%First note that the equivalency of  \eqref{e:FuncBS} and \eqref{e:HJHC} for any (or all) $s>0$ is confirmed by setting $\psi = s \phi + \frac12 |\cdot|^2$ and from $s Q_s\phi = Q_1(s\phi)$.
%Hence, it suffices to show 
With the relation \eqref{e:LinkHJLege} and the above argument in mind, let us take a continuous function $\phi$ such that $|x| e^{-\phi} \in L^1(\gamma)$ and $\int_{\mathbb{R}^n} xe^{-\phi}\, d\gamma =0$. 
Our goal is to show 
\begin{align}\label{e:HJHCs=1}
\int_{\mathbb{R}^n} e^{-\phi}\, d\gamma \int_{\mathbb{R}^n} e^{Q_1 \phi}\, d\gamma \le 1.
\end{align}
%for any continuous $\phi \colon \mathbb{R}^n \to \R$ with $|x|e^{-\phi} \in L^1(\gamma)$ and $\int_{\mathbb{R}^n} xe^{-\phi}\, d\gamma=0$. 
We may assume $\phi$ satisfies \eqref{e:VanVisCond} without loss of generality. 
%Indeed,  let $\phi$ be a continuous function with $|x|e^{-\phi} \in L^1(\gamma)$ and $\int_{\mathbb{R}^n} x e^{-\phi}\, d\gamma =0$, and
To see this,  set 
$$
\phi_k(x) \coloneqq \max \{ \phi(x), -k\}
$$
for $k \in \mathbb{N}$ and $x \in \mathbb{R}^n$. 
We also set 
$$
\xi_k \coloneqq \frac{\int_{\mathbb{R}^n} x e^{-\phi_k}\, d\gamma}{ \int_{\mathbb{R}^n} e^{-\phi_k}\, d\gamma } \in \mathbb{R}^n
$$
and 
$$
\widetilde{\phi}_k(x) \coloneqq \phi_k(x + \xi_k) + \langle x, \xi_k \rangle + \frac12 |\xi_k|^2, \;\;\; x \in \mathbb{R}^n.
$$
Then it is easy to see that $\widetilde{\phi}_k$ is continuous and $|x| e^{-\widetilde{\phi}_k} \in L^1(\gamma)$, and satisfies $\int_{\mathbb{R}^n} xe^{-\widetilde{\phi}_k}\, d\gamma = 0$ and \eqref{e:VanVisCond}. 
Hence by applying \eqref{e:HJHCs=1} for $\widetilde{\phi}_k$, we obtain 
$$
\int_{\mathbb{R}^n} e^{-\widetilde{\phi}_k}\, d\gamma \int_{\mathbb{R}^n} e^{-Q_1\widetilde{\phi}_k}\, d\gamma \le 1.
$$
Notice that $\xi_k \to0$ as $k\to\infty$ since $\phi_k \ge \phi$ on $\mathbb{R}^n$ and $e^{-\phi}, | \cdot | e^{-\phi} \in L^1(d\gamma)$. 
This fact yields that $\widetilde{\phi}_k \to \phi$ as $k \to \infty$ from which, as well as the dominated convergence theorem,  we obtain 
$$
\lim_{k \to \infty}
\int_{\mathbb{R}^n} e^{-\widetilde{\phi}_k}\, d\gamma = \int_{\mathbb{R}^n} e^{-\phi}\, d\gamma.
$$
We also see that for $x \in \mathbb{R}^n$, 
\begin{align*}
Q_1\widetilde{\phi}_k(x)
=&
\inf_{y \in \mathbb{R}^n} [\phi_k(y + \xi_k) + \langle y, \xi_k \rangle + \frac12 |\xi_k|^2 + \frac12 |x-y|^2]
\\
\ge& 
\inf_{y \in \mathbb{R}^n} [\phi(y + \xi_k) + \langle y, \xi_k \rangle + \frac12 |\xi_k|^2 + \frac12 |x-y|^2]
\\
=&
\inf_{y \in \mathbb{R}^n} [\phi(y) + \langle y - \xi_k, \xi_k \rangle + \frac12 |\xi_k|^2 + \frac12 |x-y + \xi_k|^2]
\\
=&
\inf_{y \in \mathbb{R}^n} [\phi(y) + \langle x, \xi_k \rangle + \frac12 |x-y|^2]
\\
=&
Q_1\phi(x) + \langle x, \xi_k \rangle. 
\end{align*}
Hence, Fatou's lemma yields that 
$$
\liminf_{k \to \infty} \int_{\mathbb{R}^n} e^{Q_1\widetilde{\phi}_k}\, d\gamma
\ge 
\liminf_{k \to \infty} \int_{\mathbb{R}^n} e^{Q_1\phi(x) + \langle x, \xi_k \rangle }\, d\gamma
\ge 
\int_{\mathbb{R}^n} e^{Q_1\phi(x) }\, d\gamma, 
$$
and hence we conclude \eqref{e:HJHCs=1} for $\phi$. 

Therefore it suffices to show \eqref{e:HJHCs=1} for a continuous function $\phi$ satisfying $|x| e^{-\phi} \in L^1(\gamma)$, $\int_{\mathbb{R}^n} x e^{-\phi}\, d\gamma=0$ and \eqref{e:VanVisCond}.
Let $\varepsilon>0$ and take $a_\varepsilon>0$ such that 
$$
p_\varepsilon \coloneqq 2\varepsilon a_\varepsilon = 1- e^{-2\varepsilon}.
$$
Put 
$$
b_\varepsilon \coloneqq \frac{\int_{\mathbb{R}^n} x e^{-a_\varepsilon \phi}\, d\gamma}{\int_{\mathbb{R}^n} e^{-a_\varepsilon \phi}\, d\gamma} \in \mathbb{R}^n
$$
and 
$$
\phi_\varepsilon \coloneqq \phi( \cdot + b_\varepsilon) + \frac{1}{a_\varepsilon}\langle \cdot, b_\varepsilon \rangle + \frac{1}{2a_\varepsilon}|b_\varepsilon|^2
$$
on $\mathbb{R}^n$.
Then by definition of $b_\varepsilon$, 
\begin{align*}
\int_{\mathbb{R}^n} x e^{-a_\varepsilon \phi_\varepsilon}\, d\gamma 
=& 
\int_{\mathbb{R}^n} x e^{-a_\varepsilon \phi( x + b_\varepsilon) - \langle x, b_\varepsilon \rangle - \frac{1}{2}|b_\varepsilon|^2 }\, d\gamma 
\\
=& 
\int_{\mathbb{R}^n} (x - b_\varepsilon) e^{-a_\varepsilon \phi(x) - \langle x- b_\varepsilon, b_\varepsilon \rangle - \frac{1}{2}|b_\varepsilon|^2 } \frac{1}{(2\pi)^{\frac n2}} e^{- \frac{1}{2}|x - b_\varepsilon|^2}\, dx 
\\
=& 
\int_{\mathbb{R}^n} (x - b_\varepsilon) e^{-a_\varepsilon \phi}\, d\gamma 
\\
=&
0.
\end{align*}
Let us consider 
$$
u^\varepsilon \coloneqq -2\varepsilon \log P_{\varepsilon}[e^{-\frac{\phi_\varepsilon}{2\varepsilon}}].
$$
Since $\int_{\mathbb{R}^n} x e^{- a_\varepsilon \phi_\varepsilon}\,d\gamma =0$, 
Theorem \ref{t:NelsonGene} yields that 
\begin{align*}
\|e^{u^\varepsilon}\|_{L^{a_\varepsilon}(\gamma)} 
= 
\| P_{\varepsilon}[e^{-\frac{\phi_\varepsilon}{2\varepsilon}}] \|_{L^{-p_\varepsilon}(\gamma)}^{-2\varepsilon} 
= 
\| P_{\varepsilon}[e^{-\frac{a_\varepsilon \phi_\varepsilon}{p_\varepsilon}}] \|_{L^{-p_\varepsilon}(\gamma)}^{-2\varepsilon} 
	\le 
	%(\int_{\mathbb{R}^n} e^{-a_\varepsilon \phi_\varepsilon}\, d\gamma)^{-\frac{2\varepsilon}{p_\varepsilon}}
%= 
\|e^{\phi_\varepsilon}\|_{L^{-a_\varepsilon}(\gamma)}.
\end{align*}
Equivalently, 
$$
\int_{\mathbb{R}^n} e^{a_\varepsilon u^ \varepsilon}\, d\gamma 
\int_{\mathbb{R}^n} e^{-a_\varepsilon \phi_\varepsilon}\, d\gamma \le 1.
$$
As $\varepsilon \downarrow 0$, we see that $a_\varepsilon \to 1$, and it follows from \eqref{e:VanVisCond} and the dominated convergence theorem that 
$$
\lim_{\varepsilon \downarrow 0} b_\varepsilon = \frac{\int_{\mathbb{R}^n} x e^{-\phi}\, d\gamma}{\int_{\mathbb{R}^n} e^{-\phi}\, d\gamma} =0.
$$
In particular, we see that $\lim_{\varepsilon \downarrow 0} \phi_\varepsilon = \phi$. 
Hence, by \eqref{e:VanVisCond} and the dominated convergence theorem again, we obtain 
$$
\lim_{\varepsilon \downarrow 0} \int_{\mathbb{R}^n} e^{-a_\varepsilon \phi_\varepsilon}\, d\gamma
=
\int_{\mathbb{R}^n} e^{- \phi}\, d\gamma.
$$
On the other hand, we have already observed  
$$
\liminf_{\varepsilon \downarrow 0} \int_{\mathbb{R}^n} e^{a_\varepsilon u^ \varepsilon}\, d\gamma 
\ge 
\int_{\mathbb{R}^n} e^{Q_1\phi} \, d\gamma
$$
in the proof of Proposition \ref{Prop:HCBS_2}.
Hence we obtain the desired assertion.

\subsection{Proof of Proposition \ref{Prop:HCSet}-{\textit{(2)}} }\label{S4.5}
Let $\mathcal{N}$ be a family of concave functions $\phi \colon\mathbb{R}^n\to \mathbb{R}$ satisfying the growth condition
\begin{equation}\label{e:Growth}
\phi(y) \ge -\frac12 \rho |y|^2 + \langle y_0,y\rangle + c_0
\end{equation}
on $|y| \ge R_* \gg1$ for some $R_*\gg1$,  $\rho \in (0,1)$, $y_0 \in \mathbb{R}^n$, and $c_0 \in\mathbb{R}$. 
As we will see in Lemma \ref{l:VaniVis},  if $\phi$ is concave and $-(1-\frac1\beta)$-semi-convex then \eqref{e:Growth} is satisfied with $\rho = 1-\frac1\beta$ and hence $\phi \in \mathcal{N}$. 
Note that $\phi\coloneqq \frac12 \|\cdot\|_K^2 - \frac12|\cdot|^2$ is in the class $\mathcal{N}$ for any convex body $K$. Hence Proposition \ref{Prop:HCSet}-{\textit{(2)}} is a particular case of the following.  
\begin{proposition}\label{Prop:HCMahler}
Let $\psi:\mathbb{R}^n \to \mathbb{R}$ be such that $\psi - \frac12 |\cdot|^2 \in \mathcal{N}$. 
If \eqref{e:FHC_MaIntro} holds for $f = f_s$ defined by $\log\, f_s(x) \coloneqq \frac{p_s}{2s}(\frac12|x|^2 - \psi(x))$ for all small $s>0$, then 
%Suppose $f:\mathbb{R}^n\to(0,\infty)$ is such that $\ -\log\, f \in \mathcal{N}$ and satisfy  
%\begin{equation}\label{e:AssumpFHC}
%\big\|
%P_s \big[f^\frac1{ p_s }\big]
%\big\|_{L^{q_s}(\gamma)}
%\le 
%C_{s,f}^{\frac{1}{p_s}}
%\big(
%\int_{\mathbb{R}^n} 
%f\, d\gamma 
%\big)^{\frac1{p_s}},
%\;\;\;
%p_s=1-e^{-2s},
%\;
%q_s= 1-e^{2s} 
%\end{equation}
%for some $C_{s,f}>0$ and all small $s>0$. 
%Then for such $f$, 
\begin{equation}\label{e:InvS}
\int_{\mathbb{R}^n} e^{-\psi}\, dx \int_{\mathbb{R}^n} e^{-\psi^*}\, dx \ge (2\pi)^n (\liminf_{s\downarrow0} {\rm IS}_s)^{-1}.
%,\;\;\;
%\psi(x) = \frac12|x|^2 - \log\, f(x). 
\end{equation}
\end{proposition}

\if0
\begin{remark}
As it will be clear from the proof one may weaken the assumption \eqref{e:FHC_MaIntro} in terms of $p$. 
Namely, one can derive \eqref{e:InvS} from the \textit{non-endpoint estimate}
\begin{equation}\label{e:AssumpFHCNonEnd}
\big\|
P_s \big[f^\frac1{ r_s }\big]
\big\|_{L^{q_s}(\gamma)}
\le 
{\rm IS}_s^{\frac{1}{r_s}}
\big(
\int_{\mathbb{R}^n} 
f\, d\gamma 
\big)^{\frac1{r_s}},
\;\;\; 
r_s\coloneqq 2s > p_s. 
\end{equation}
\end{remark}
\fi 

We begin with an investigation of the class $\mathcal{N}$. 

\begin{lemma}\label{l:VaniVis}
If $\phi \in \mathcal{N}$ then 
\begin{equation}\label{e:WhatWeNeed}
\limsup_{\varepsilon \downarrow 0 }
-2\varepsilon \log\, P_\varepsilon \big[ e^{-\frac{\phi}{2\varepsilon}} \big](x) 
\le 
Q_1\phi(x)
\end{equation}
for each $x\in \mathbb{R}^n$. 

Moreover if $\phi\colon\mathbb{R}^n\to\mathbb{R}$ is concave and $-(1-\frac1\beta)$-semi-convex for some $\beta\ge1$,  then $\phi \in \mathcal{N}$,  $Q_1\phi$ is continuous on $\mathbb{R}^n$ and \eqref{e:VaniVis} holds pointwisely. 
\end{lemma}

\begin{proof}
Fix arbitrary $x_0\in\mathbb{R}^n$ and let us show \eqref{e:WhatWeNeed} at the point. 
%{\color{red}We additionally impose $\phi \in C^1(\mathbb{R}^n)$. }
Recall the definition of $R_*$ for \eqref{e:Growth}.  Since $\rho<1$ we know that 
\begin{align*}
&\inf_{|y|>R_{x_0}} \big[ \phi(y) + \frac12|x_0 - y|^2 \big] \\
&\ge 
\inf_{|y|>R_{x_0}} \big[  \frac12 (1- \rho) |y|^2 +  \langle y_0 - x_0 , y \rangle +c_0 +\frac12|x_0|^2\big] 
>
Q_1\phi(x_0)
\end{align*}
for some $R_{x_0} \ge R_*$. 
This means 
$$
Q_1\phi(x_0)
=
\inf_{|y|\le R_{x_0}} \big[ \phi(y) + \frac12|x_0 - y|^2 \big] 
$$
and hence there exists $y_0 \in \mathbb{R}^n$ which attains the infimum since $\phi$ is concave on $\mathbb{R}^n$ and in particular continuous. 
We notice that $\phi$ is indeed differentiable at $y_0$ since $\phi$ is concave, see Lemma \ref{l:Diffable} in Appendix \ref{S6.4}.  
With this in mind we know that 
\begin{equation}\label{e:ArgminQ_t}
\nabla \phi(y_0)  + y_0 - x_0 = 0
\end{equation}
since $y_0$ is a global minimum of $ y\mapsto  \phi(y) + \frac12|x_0 - y|^2$. 
Moreover, by concavity of $\phi$ and \eqref{e:ArgminQ_t}, it holds that 
\begin{align*}
\phi(e^{-\varepsilon} x_0 + \sqrt{1-e^{-2\varepsilon}} z)
\le &
\phi(y_0) + \langle \nabla \phi(y_0), e^{-\varepsilon} x_0 + \sqrt{1-e^{-2\varepsilon}} z -y_0 \rangle
\\
=&
\phi(y_0) + \langle x_0-y_0, e^{-\varepsilon} x_0 + \sqrt{1-e^{-2\varepsilon}} z -y_0 \rangle
\end{align*}
for any $z \in \mathbb{R}^n$, 
from which we obtain that 
\begin{align*}
u^\varepsilon (x_0)
=&
-2 \varepsilon \log \int_{\mathbb{R}^n} e^{-\frac{1}{2\varepsilon} \phi(e^{-\varepsilon} x_0 + \sqrt{1-e^{-2\varepsilon}} z)}\, d\gamma(z)
\\
\le&
\phi(y_0)
+
\langle x_0-y_0, e^{-\varepsilon} x_0 -y_0 \rangle 
-2\varepsilon 
\log\, 
\bigg(
\int_{\mathbb{R}^n}  e^{-\frac{\sqrt{1-e^{-2\varepsilon}}}{2\varepsilon} \langle x_0-y_0, z \rangle} \, d\gamma(z)
\bigg)
\\
=&
\phi(y_0) + \langle x_0-y_0, e^{-\varepsilon} x_0 -y_0 \rangle - \frac{1-e^{-2\varepsilon}}{4\varepsilon} |x_0-y_0|^2,
\end{align*}
where we also used Lemma \ref{l:GaussInt1} in the last equality. 
This concludes 
$$
\limsup_{\varepsilon \downarrow 0} u^\varepsilon(x_0)
\le 
\phi(y_0) + \frac12 |x_0-y_0|^2 
= 
Q_1\phi(x_0).
$$

Next we assume $\phi$ is concave and $-(1-\frac1\beta)$-semi-convex. 
First it follows from $-(1-\frac1\beta)$-semi-convexity of $\phi$ that 
\begin{equation}\label{e:Decay21Sep}
\phi(y) \ge -\frac12(1-\frac1\beta) |y|^2 + \langle y_0 , y\rangle + c_0,\;\;\; y \in \mathbb{R}^n, 
\end{equation}
for some $y_0\in\mathbb{R}^n$ and $c_0\in \mathbb{R}$ in view of \eqref{e:DecayFromConv}\footnote{To be precise one needs some regularity of $\phi$ to use \eqref{e:DecayFromConv}. This problem can be justified by invoking the subdifferential, see also Appendix \ref{S6.4}. }.
This in particular ensures $\phi \in \mathcal{N}$.
Let us then show $Q_1\phi$ is lower semi-continuous; for all $x_0 \in \mathbb{R}^n$ 
\begin{equation}\label{e:Q1LSC}
\lim_{\delta\downarrow 0} \inf_{x\in \mathbb{R}^n: |x- x_0|<\delta} Q_1\phi(x) \ge Q_1\phi(x_0). 
\end{equation}
Remark that $Q_1\phi(x_0) < \infty$ for any $x_0$ from the definition. 
From \eqref{e:Decay21Sep} we have for any $x,y\in\mathbb{R}^n$ that 
\begin{align*}
\phi(y)+\frac12|x-y|^2 
&\ge 
\frac1{2\beta} |y|^2 
+ 
\langle y_0 - x, y\rangle 
+c_0 + \frac12|x|^2 \\
&\ge
\frac1{2\beta} |y|^2 
+ 
\langle y_0 - x_0, y\rangle 
+ 
\langle x_0-x,y\rangle
+c_0.
\end{align*}
Hence we can find sufficiently large $R_0 = R_0(\phi,x_0)$ satisfying 
\begin{equation}\label{e:Lower20Sep}
\phi(y)+\frac12|x-y|^2 
\ge 
\frac1{2\beta}|y|^2 
-
\big(
|y_0 - x_0| + \frac1{10}
\big)|y|+c_0
>
Q_1\phi(x_0)
\end{equation}
for any $x,y$ such that $|x-x_0|<10^{-1}$ and $|y|\ge R_0$. 
This shows 
$$
\inf_{x\in \mathbb{R}^n:|x-x_0| < 10^{-1}} 
\inf_{y\in\mathbb{R}^n:|y|\ge R_0} 
\big[
\phi(y)
+
\frac12|x-y|^2
\big]
\ge 
Q_1\phi(x_0)
$$
and hence \eqref{e:Q1LSC} would follow once one could see that 
\begin{equation}\label{e:Goal20Sep1}
\lim_{\delta\downarrow0} \inf_{|x-x_0|<\delta} 
\inf_{|y|\le R_0} 
\big[
\phi(y)
+
\frac12|x-y|^2
\big]
\ge 
Q_1\phi(x_0). 
\end{equation}
With this in mind we estimate for any $x,y\in\mathbb{R}^n$ by 
\begin{align*}
\phi(y)
+
\frac12|x-y|^2
&\ge 
\phi(y)
+
\frac12\big( |x-x_0| - |x_0-y|\big)^2 \\
&= 
\phi(y)
+
\frac12 |x-x_0|^2  
+\frac12 |x_0-y|^2
- 
|x-x_0||x_0-y|\\
&\ge
\phi(y)
+
\frac12 |y-x_0|^2  
- 
|x-x_0||x_0-y|.
\end{align*}
This yields that 
\begin{align*}
\inf_{|y|\le R_0}
\big[
\phi(y)
+
\frac12|x-y|^2
\big]
&\ge
\inf_{|y|\le R_0}
\big[
\phi(y)
+
\frac12 |y-x_0|^2  
\big]
- 
\sup_{|y|\le R_0} 
|x-x_0||x_0-y|\\
&\ge 
Q_1\phi(x_0)
- 
( |x_0| + R_0) |x-x_0|
\end{align*}
from which we establish \eqref{e:Goal20Sep1}. 

Second, we show that $Q_1\phi$ is upper semi-continuous. 
Fix $x_0 \in \mathbb{R}^n$ and take $\lambda > Q_1 \phi(x_0)$.  It suffices to show 
$\lambda > Q_1 \phi$ on some neighborhood of $x_0$. 
If it is not true, then we can take some sequence $(x_k)_{k \in \mathbb{N}} \subset \mathbb{R}^n$ such that $\lim_{k \to \infty} |x_k -x_0|=0$ and $\lambda \le Q_1 \phi(x_k)$. 
The latter property implies that for given $\lambda' \in (Q_1\phi(x_0), \lambda)$, it holds that 
$$
\phi(y) + \frac12 |x_k-y|^2 \ge \lambda', \;\;\; \forall y \in \mathbb{R}^n, k \in \mathbb{N}. 
$$
Thus as $k \to \infty$, we see that $\phi(y) + \frac12 |x_0-y|^2 \ge \lambda'$ for any $y \in \mathbb{R}^n$. This means $Q_1\phi(x_0) \ge \lambda'$ which is a contradiction. 

To show \eqref{e:VaniVis},  it suffices to show 
$$
\liminf_{\varepsilon \downarrow 0 }
-2\varepsilon \log\, P_\varepsilon \big[ e^{-\frac{\phi}{2\varepsilon}} \big](x_0) 
\ge 
Q_1\phi(x_0)
$$
for each $x_0\in \mathbb{R}^n$ since we have already proved the reverse inequality \eqref{e:WhatWeNeed}. 
However we may invoke the argument for \cite[(6.6)]{BNT} to ensure the inequality since we observed $Q_1\phi$ is continuous and in particular lower semi-continuous. 
\end{proof}

\begin{proof}[Proof of Proposition \ref{Prop:HCMahler}]
Let $\phi(x) \coloneqq \psi(x) - \frac12|x|^2$. %and denote 
%$
%r_\varepsilon \coloneqq 2\varepsilon > p_\varepsilon
%$
%for small $\varepsilon>0$. 
From the assumption \eqref{e:FHC_MaIntro} with $f = e^{-\frac{p_\varepsilon}{2\varepsilon} \phi}$,  we  have that 
%$$
%\big\| 
%P_\varepsilon\big[
%f^{\frac1{r_\varepsilon}}
%\big]
%\le 
%C_{\varepsilon,f}^\frac1{p_\varepsilon} 
%\big(
%\int_{\mathbb{R}^n} 
%f\, d\gamma 
%\big)^\frac1{r_\varepsilon}. 
%$$
%This can be read as 
$$
\big\| 
P_\varepsilon\big[
e^{ - \frac{\phi}{2\varepsilon }}
\big]
\big\|_{L^{q_\varepsilon}(\gamma)}
%=
%\big\| 
%P_\varepsilon\big[
%f^{\frac1{r_\varepsilon}}
%\big]
%\big\|_{L^{q_\varepsilon}(\gamma)}
\le 
{\rm IS}_{\varepsilon}^\frac1{p_\varepsilon} 
\big(
\int_{\mathbb{R}^n} 
e^{-\frac{p_\varepsilon}{2\varepsilon} \phi}\, d\gamma 
\big)^\frac1{p_\varepsilon}. 
$$
For the left hand side, we know that 
$$
\big\| 
P_\varepsilon\big[
e^{ - \frac{\phi}{2\varepsilon }}
\big]
\big\|_{L^{q_\varepsilon}(\gamma)}
= 
\big\| 
P_\varepsilon\big[
e^{ - \frac{\phi}{2\varepsilon }}
\big]^{-2\varepsilon}
\big\|_{L^{-\frac{q_\varepsilon}{2\varepsilon}}(\gamma)}^{-\frac1{2\varepsilon}}
$$
and this yields that 
$$
\big\| 
P_\varepsilon\big[
e^{ - \frac{\phi}{2\varepsilon }}
\big]^{-2\varepsilon}
\big\|_{L^{-\frac{q_\varepsilon}{2\varepsilon}}(\gamma)}
\ge 
{\rm IS}_{\varepsilon}^{- \frac{2\varepsilon}{p_\varepsilon} }
\big\| 
e^{\phi}
\big\|_{L^{-\frac{p_\varepsilon}{2\varepsilon}}(\gamma)}.
$$
We then take a limit $\varepsilon \downarrow 0$ in which case 
%Note that 
%$ \frac{p_\varepsilon}{2\varepsilon} \to 1$ and $\frac{q_\varepsilon}{2\varepsilon} \to -1$. 
the right hand side becomes 
$$
(\liminf_{\varepsilon \downarrow 0}
{\rm IS}_{\varepsilon})^{- 1 }
\big\| 
e^{\phi}
\big\|_{L^{-1}(\gamma)}
$$
since $ \frac{p_\varepsilon}{2\varepsilon} \to 1$ follows from the assumption $p_s = 2s +O(s^2)$. 
For the left hand side we employ Lemma \ref{l:VaniVis} as follows. 
Since $\frac{q_\varepsilon}{2\varepsilon} \to -1$ we know that 
$$
\lim_{\varepsilon \downarrow 0}
\big\| 
P_\varepsilon\big[
e^{ - \frac{\phi}{2\varepsilon }}
\big]^{-2\varepsilon}
\big\|_{L^{-\frac{q_\varepsilon}{2\varepsilon}}(\gamma)}
=
\lim_{\varepsilon \downarrow 0}
\int_{\mathbb{R}^n} 
(e^{u^\varepsilon})^{ - \frac{q_\varepsilon}{2\varepsilon}}\, d\gamma .
$$
We appeal to the dominated convergence theorem to interchange the limit and integral. 
To this end we notice from concavity of $\phi$ that 
$$
\phi(x) 
\le \langle x_0, x\rangle + c_0,\;\;\; x\in \mathbb{R}^n
$$
for some $x_0 \in \mathbb{R}^n$ and $c_0 \in\mathbb{R}$ from which we see that 
$$
e^{u^\varepsilon}(x)^{-\frac{q_\varepsilon}{2\varepsilon}} 
=
P_\varepsilon\big[
e^{ - \frac{\phi}{2\varepsilon }}
\big](x)^{q_\varepsilon}
\le 
P_\varepsilon\big[
e^{ - \frac{1}{2\varepsilon }(\langle x_0,\cdot\rangle + c_0)}
\big](x)^{q_\varepsilon}
$$
since $q_\varepsilon<0$. We then use Lemma \ref{l:GaussInt1} to see that 
$$
e^{u^\varepsilon}(x)^{-\frac{q_\varepsilon}{2\varepsilon}} 
\le 
e^{-\frac{q_\varepsilon}{2\varepsilon}c_0}
e^{q_\varepsilon\frac{p_\varepsilon}{2}|\frac{1}{2\varepsilon}x_0|^2}
e^{- q_\varepsilon \frac1{2\varepsilon} \langle x_0,e^{-\varepsilon}x\rangle }
\;
\to
\;
e^{c_0} e^{-\frac12|x_0|^2+\langle x_0 , x\rangle}
$$
as $\varepsilon\downarrow0$. This ensures that 
$$
e^{u^\varepsilon}(x)^{-\frac{q_\varepsilon}{2\varepsilon}} 
\le 
C e^{100|x_0||x|} \in L^1(\gamma)
$$
for sufficiently small $\varepsilon>0$. 
Hence the dominated convergence theorem and \eqref{e:WhatWeNeed} yield that 
\begin{align*}
\lim_{\varepsilon \downarrow 0}
\big\| 
P_\varepsilon\big[
e^{ - \frac{\phi}{2\varepsilon }}
\big]^{-2\varepsilon}
\big\|_{L^{-\frac{q_\varepsilon}{2\varepsilon}}(\gamma)}
&=
\int_{\mathbb{R}^n} 
e^{\lim_{\varepsilon \downarrow 0}
 u^\varepsilon} 
\, d\gamma
\le 
\int_{\mathbb{R}^n} 
e^{Q_1\phi(x)} 
\, d\gamma.
\end{align*}
Putting altogether we obtain 
\begin{equation}\label{e:HC-HJ}
\big\| e^{Q_1 \phi} \big\|_{L^1(\gamma)}
\ge 
(\liminf_{\varepsilon \downarrow 0}
{\rm IS}_{\varepsilon})^{- 1 }
\big\| 
e^{\phi}
\big\|_{L^{-1}(\gamma)}. 
\end{equation}
Since we have 
$$
Q_1\phi(x) = -\psi^*(x) + \frac12|x|^2, 
$$
we conclude \eqref{e:InvS}.
\end{proof}

One can obtain more general statement of Proposition \ref{Prop:HCMahler}.  For the sake of simplicity, let us state under Nelson's time relation. 
\begin{lemma}\label{l:Useful}
Suppose we are given $q_s<0<p_s$ for each small $s>0$ and assume that they satisfy $\frac{q_s-1}{p_s-1} = e^{2s}$ and that $a \coloneqq \lim_{s \downarrow 0} \frac{-q_s}{2s}$ exists and is included in $(0, 1)$.\footnote{We remark that our assumptions of $p_s$ and $q_s$ implies $a \in [0, 1]$ if the $a$ exists. In fact, $q_s <0$ immediately yields $a \ge 0$. On the other hand, since $1-a = \lim_{\varepsilon} \frac{p_s}{2s}$ by Nelson's time and $p_s >0$, we also have $1-a \ge 0$.}
\begin{itemize}
\item (Assumption)
Suppose one could show 
\begin{equation}\label{e:Assump13Oct}
%0 \le \nabla^2 \log\, f \le 1 - \frac1\beta 
%\;\;\;
%\Rightarrow
%\;\;\;
\big\|
P_s \big[f^\frac1{p_s}\big]
\big\|_{L^{q_s}(\gamma)}
\le 
\big\|
P_s \big[\big( \frac{\gamma_\beta}{\gamma} \big)^\frac1{p_s}\big]
\big\|_{L^{q_s}(\gamma)}
\big(
\int_{\mathbb{R}^n}
f\, d\gamma 
\big)^\frac1{p_s}
\end{equation}
for all positive, log-convex, and $1-\frac1\beta$-semi-log-concave $f$,  all small $s>0$, and for all $\beta\ge1$ such that $\beta_{s,p}>0$. 
\item(Consequence)
For any $\beta\ge1$ and $-(1-\frac1\beta)$-semi-convex and concave $\phi$, %satisfying $-(1-\frac 1\beta) \le \nabla^2 \phi \le 0$
one has 
\begin{equation}\label{e:Conseq13Oct}
\big\| e^{Q_1 \phi} \big\|_{L^a(\gamma)}
\big\| e^{-\phi} \big\|_{L^{1-a}(\gamma)}
\ge 
\big\| e^{Q_1 \big[ -\frac12(1-\frac1\beta)|\cdot|^2 \big]} \big\|_{L^a(\gamma)}
\big\| e^{\frac12(1-\frac1\beta)|\cdot|^2} \big\|_{L^{1-a}(\gamma)} 
\end{equation}
and 
\begin{equation}\label{e:Identity19Oct}
\big\| e^{Q_1 \big[ -\frac12(1-\frac1\beta)|\cdot|^2 \big]} \big\|_{L^a(\gamma)}
\big\| e^{\frac12(1-\frac1\beta)|\cdot|^2} \big\|_{L^{1-a}(\gamma)}
=
\big( 
\beta^{\frac1{1-a}}
(\beta a - a +1)^{ - \frac{1}{a(1-a)}}
\big)^\frac{n}{2}. 
\end{equation}
\end{itemize}
\end{lemma}

\begin{proof}
The argument is almost parallel to the proof of Proposition \ref{Prop:HCMahler} so we give a sketch of the proof. 
Take $\beta\ge1$ and $\phi$ satisfying the condition.  
We then define 
$$
f_\varepsilon := e^{-\frac{p_\varepsilon}{2\varepsilon}\phi}
$$ 
which is log-convex and $\frac{p_\varepsilon}{2\varepsilon}(1-\frac1\beta) =: 1 - \frac1{\beta(\varepsilon)}$-semi-log-concave. We remark that $\beta(\varepsilon) \ge 1$ always holds by definition. 
Note that $ -\log\, f_\varepsilon \in \mathcal{N}$ follows from Lemma \ref{l:VaniVis}, $\beta(\varepsilon) \ge 1$ and the assumption on $\phi$.  Hence the vanishing viscosity argument can be justified as in the proof of Proposition \ref{Prop:HCMahler}. 
Hence,  in view of $a = \lim_{\varepsilon \downarrow 0} \frac{- q_\varepsilon}{2\varepsilon}$, one can see that 
$$
\big\|
e^{Q_1\phi} 
\big\|_{L^a(\gamma)}
=
\lim_{\varepsilon \downarrow 0}
\big\|
P_\varepsilon \big[ f_\varepsilon^\frac1{p_\varepsilon} \big] 
\big\|_{L^{q_\varepsilon}(\gamma)}^{-2\varepsilon}.
$$
We now intend to apply the assumption \eqref{e:Assump13Oct}.
To this end we need to check $\beta(\varepsilon)\ge1$ and $(\beta(\varepsilon))_{\varepsilon, p_\varepsilon} >0$. 
The first one is immediate from the definition and moreover one can check $1 \le  \beta(\varepsilon) \le \beta $. 
The second one can be ensured from the assumption that $q_\varepsilon$ and $p_\varepsilon$ satisfy Nelson's time relation. 
Hence we may apply \eqref{e:Assump13Oct} to see that 
\begin{align*}
\big\|
e^{Q_1\phi} 
\big\|_{L^a(\gamma)}
\ge& 
\lim_{\varepsilon \downarrow 0}
\big\|
P_\varepsilon \big[ \big( \frac{ \gamma_{\beta(\varepsilon)} }{\gamma} \big)^\frac1{p_\varepsilon} \big] 
\big\|_{L^{q_\varepsilon}(\gamma)}^{-2\varepsilon}
\big( 
\int f_\varepsilon\, d\gamma 
\big)^{-\frac{2\varepsilon}{p_\varepsilon}}\\
=& 
\lim_{\varepsilon \downarrow 0}
\big\|
P_\varepsilon \big[ \big( \frac{ \gamma_{\beta(\varepsilon)} }{\gamma} \big)^\frac1{p_\varepsilon} \big] 
\big\|_{L^{q_\varepsilon}(\gamma)}^{-2\varepsilon}
\| e^{\phi} \|_{L^{a-1}(\gamma)}
\end{align*}
since $\lim_{\varepsilon\downarrow 0} \frac{p_\varepsilon}{2\varepsilon} = 1 - \lim_{\varepsilon\downarrow 0} \frac{-q_\varepsilon}{2\varepsilon}  = 1-a$.  
To compute the constant we note that if $\phi(x) = \frac12(\frac1\beta - 1)|x|^2$, then $f_\varepsilon = c \frac{\gamma_{\beta(\varepsilon)}}{\gamma}$ and hence the inequality \eqref{e:Assump13Oct} becomes equality for all $\varepsilon$. 
This means 
$$
\lim_{\varepsilon \downarrow 0}
\big\|
P_\varepsilon \big[ \big( \frac{ \gamma_{\beta(\varepsilon)} }{\gamma} \big)^\frac1{p_\varepsilon} \big] 
\big\|_{L^{q_\varepsilon}(\gamma)}^{-2\varepsilon}
=
\big\| e^{Q_1 \big[ -\frac12(1-\frac1\beta)|\cdot|^2 \big]} \big\|_{L^a(\gamma)}
\big\| e^{\frac12(1-\frac1\beta)|\cdot|^2} \big\|_{L^{1-a}(\gamma)}.
$$
The identity \eqref{e:Identity19Oct} is a direct consequence from the Gaussian integration. 

\end{proof}

\if0
Combining Theorem \ref{t:FHC} and Proposition \ref{Prop:HCMahler} we obtain the following which contains Theorem \ref{t:RegInvS}.  We provide slightly stronger statement with respect to the regularity of $\psi$. 

\begin{corollary}\label{Cor:RegInvS}
Let $\beta \ge1$ and $\psi \colon \mathbb{R}^n\to \mathbb{R}$ be $\frac1\beta \Lambda$-semi-convex and $\Lambda$-semi-concave for some positive definite $\Lambda$.
Then we have \eqref{e:RegInvS}.
\end{corollary}

\begin{proof}
Since the inequality \eqref{e:RegInvS} is scale free we may assume $\Lambda = {\rm id}$. 
If we let $f(x) \coloneqq {\rm exp}\, \big( \frac12|x|^2 - \psi(x) \big)$ then the assumption on $\psi$ ensures that $f$ satisfies the assumption in Corollary \ref{Cor:p->1-e(-2s)} and hence \eqref{e:LoughEndpoint}. 
Moreover $ -\log\, f \in \mathcal{N}$ follows from Lemma \ref{l:VaniVis}.
Therefore we may apply Proposition \ref{Prop:HCMahler} to conclude. 
\end{proof}

\fi 

\subsection{Proof of Theorems \ref{t:ImpInv}}\label{S4.6}
The following is a direct consequence from Theorem \ref{t:FHC} and Lemma \ref{l:Useful}. 
\begin{theorem}\label{t:Func19Oct}
Let $n\in\mathbb{N}$, $a\in (0,1)$, and $\beta\ge1$. 
For any $\frac1\beta$-semi-convex and $1$-semi-concave $\psi:\mathbb{R}^n\to\mathbb{R}$, 
\begin{align}\label{e:Func19Oct}
&\big(
\int_{\mathbb{R}^n} 
e^{ - a \psi^*(x)} e^{-\frac12 (1-a)|x|^2} \frac{dx}{(2\pi)^{\frac{n}{2}}}
\big)^{\frac1a}
\big(
\int_{\mathbb{R}^n} 
e^{ - (1-a) \psi(x)} e^{-\frac12 a|x|^2} \frac{dx}{(2\pi)^{\frac{n}{2}}}
\big)^{\frac1{1-a}}\\
&\ge 
\big(
\beta^\frac1{1-a} 
(\beta a - a + 1)^{-\frac{1}{a(1-a)}}
\big)^\frac{n}{2} .\nonumber 
\end{align}
Equality holds if $\psi(x) = \frac1{2\beta}|x|^2$. 
\end{theorem}

\begin{proof}
Take any $a \in (0,1)$ and let $p_\varepsilon \coloneqq 2\varepsilon (1-a)$ and $q_\varepsilon \coloneqq 1 + (p_\varepsilon -1)e^{2\varepsilon}$ so that $\frac{q_\varepsilon-1}{p_\varepsilon -1} = e^{2\varepsilon}$.  Note that $\lim_{\varepsilon \downarrow 0 } - \frac{q_\varepsilon}{2\varepsilon} = a$. 
Also \eqref{e:Assump13Oct} follows from Theorem \ref{t:FHC} and hence we obtain  \eqref{e:Conseq13Oct}. 
By replacing $\psi = \phi + \frac12|\cdot|^2$, we conclude \eqref{e:Func19Oct}. 
\end{proof}

By choosing $\psi(x) = \frac12\|x\|_K^2$ for appropriate convex body $K$, we derive the set form of Theorem \ref{t:Func19Oct}. 
For that purpose, we introduce the weighted volume product. For a measure $\mu$ on $\mathbb{R}^n$ we define the weighted volume product by $v_\mu(K) \coloneqq \mu(K)\mu(K^\circ)$ for a convex body $K$ with $0 \in {\rm int}\, K$.  

\begin{theorem}\label{t:RegInvSanSet}
Let $n \ge 2$, $\kappa \in (0,1]$ and $\beta = \beta_\kappa\coloneqq \kappa^{-2}$. 
Then for any convex body $K$ with  $0 \in \mathrm{int}\, K$ satisfying \eqref{e:CondCurvNov} with $\Lambda = \kappa^{-1}{\rm id}$, 
we have 
\begin{equation}\label{e:SetInvSant}
\mu_{1-a,n}(K)^\frac1{1-a} \mu_{a,n}(K^\circ)^\frac1a 
\ge 
\big(
\beta_\kappa^\frac1{1-a} 
(\beta_\kappa a - a + 1)^{-\frac{1}{a(1-a)}}
\big)^\frac{n}{2} 
\mu_{1-a,n}({\rm B}_2^n)^\frac1{1-a} \mu_{a,n}({\rm B}_2^n)^\frac1a 
\end{equation}
for all $a \in (0,1)$, where 
$$
d\mu_{a,n}(x) \coloneqq \frac{a}{(a+(1-a)|x|^2)^{\frac{n+2}{2}}}\, dx.
$$ 
Moreover equality in \eqref{e:SetInvSant} is attained when $K = \frac1{\kappa} {\rm B}_2^n$.  
\end{theorem}

\begin{proof}
To derive \eqref{e:SetInvSant}, we apply \eqref{e:Func19Oct} to $\psi= \frac12\| \cdot\|_K^2$. 
To this end, we remark that the assumption \eqref{e:CondCurvNov} with $\Lambda = \kappa^{-1}{\rm id}$ can be read as 
\begin{equation}\label{e:CondCurvature2}
\frac1\beta {\rm id} \le \nabla^2 \big( \frac12 \|\cdot\|_K^2 \big)(x) \le {\rm id},\;\;\; x \in \mathbb{R}^n\setminus\{0\}. 
\end{equation}
We give a brief proof of this equivalence in Appendix \ref{S6.5}. 
With this in mind, we note that 
\begin{equation}\label{e:GaussId19Oct}
\int_{\mathbb{R}^n} 
e^{-\frac12 b \| x \|_K^2} 
e^{-\frac12a|x|^2}\, \frac{dx}{(2\pi)^\frac{n}{2}}
=
\frac1{|{\rm B}_2^n|} \int_K \frac{b}{( b+ a|x|^2 )^{\frac{n+2}{2}}}\, dx 
\end{equation}
for any $a,b>0$. 
To see this we have that 
\begin{align*}
\int_{\mathbb{R}^n} 
e^{-\frac12 b \| x \|_K^2} 
e^{-\frac12a|x|^2}\, \frac{dx}{(2\pi)^\frac{n}{2}}
&= 
\int_{\mathbb{R}^n} 
\int_{\|x\|_K}^\infty 
\big( - e^{-\frac12 b t^2} \big)'
e^{-\frac12a|x|^2}\,  dt\frac{dx}{(2\pi)^\frac{n}{2}}\\
&= 
b
\int_0^\infty \int_{\mathbb{R}^n} 
{\bf 1}_{t\ge \|x\|_K} e^{-\frac12a|x|^2}\, \frac{dx}{(2\pi)^\frac{n}{2}}
t e^{-\frac12 bt^2}dt\\
&= 
b
\int_0^\infty \int_{\mathbb{R}^n} 
{\bf 1}_{1\ge \|y\|_K} e^{-\frac12at^2|y|^2}\, \frac{dy}{(2\pi)^\frac{n}{2}}
t^{n+1} e^{-\frac12 bt^2}dt\\
&= 
b
\int_{K} 
\int_0^\infty e^{-\frac12(b+a|y|^2)t^2}\, 
t^{n+1} dt\frac{dy}{(2\pi)^\frac{n}{2}}\\
&= 
b
\int_{K} 
\int_0^\infty e^{-\frac12s^2}\, 
s^{n+1} ds
(b+a|y|^2)^{-\frac{n+2}{2}}\, 
\frac{dy}{(2\pi)^\frac{n}{2}}.
\end{align*}
Then \eqref{e:GaussId19Oct} follows from the following identity, see for instance \cite{Kla07}:  
$$
\frac{1}{(2\pi)^\frac{n}{2}} \int_0^\infty s^{n+1} e^{-\frac12s^2}\, ds
= 
\frac1{|{\rm B}_2^n|}.
$$ 
Hence we obtain from \eqref{e:Func19Oct} with \eqref{e:CondCurvature2} that 
$$
\big(\frac1{| {\rm B}_2^n |}
\mu_{1-a,n}(K)\big)^\frac1{1-a}
\big(\frac1{| {\rm B}_2^n |}
 \mu_{a,n}(K^\circ) \big)^\frac1a 
\ge 
\big(
\beta^\frac1{1-a} 
(\beta a - a + 1)^{-\frac{1}{a(1-a)}}
\big)^\frac{n}{2}.  
$$
Moreover this inequality must be equality when $\beta =1$ and $K = {\rm B}_2^n$ and hence 
$$
\big(\frac1{| {\rm B}_2^n |}
\mu_{1-a,n}({\rm B}_2^n)\big)^\frac1{1-a}
\big(\frac1{| {\rm B}_2^n |}
 \mu_{a,n}({\rm B}_2^n) \big)^\frac1a 
=1. 
$$ 
This concludes \eqref{e:SetInvSant}. 
\end{proof}

\begin{remark}
We note that \eqref{e:GaussId19Oct} particularly implies 
\begin{equation}\label{e:Id5Nov}
\mu_{1-a, n}( {\rm B}_2^n ) = | {\rm B}_2^n |,\;\;\; a \in (0, 1).
\end{equation}
\end{remark}
We are now at the place of proving Theorem \ref{t:ImpInv}. 
\begin{proof}[Proof of Theorem \ref{t:ImpInv}]
Let $K$ be a convex body satisfying \eqref{e:CondCurvNov} for some $\kappa \in (0,1]$ and positive definite $\Lambda$. 
First, we may assume $\Lambda = \kappa^{-1}{\rm id}$ without loss of generality since the inequality is linear invariant. 
This enables us to apply Theorem \ref{t:RegInvSanSet}. 
We then notice from \eqref{e:Id5Nov}, \eqref{e:GaussId19Oct} and H\"{o}lder's inequality that 
\begin{align*}
\frac{\mu_{1-a,n}(K)}{\mu_{1-a,n}( {\rm B}_2^n )}
=& 
\int_{\mathbb{R}^n}
e^{- \frac12 (1-a)\|x\|_K^2}
e^{- \frac12 a|x|^2}\,
\frac{dx}{(2\pi)^\frac{n}{2}}\\
\le& 
\big(
\int_{\mathbb{R}^n}
e^{- \frac12 \|x\|_K^2}\,
\frac{dx}{(2\pi)^\frac{n}{2}}
\big)^{1-a}
\big(
\int_{\mathbb{R}} 
e^{- \frac12 |x|^2}\,
\frac{dx}{(2\pi)^\frac{n}{2}}
\big)^a\\
=&
\big( 
\frac{|K|}{|{\rm B}_2^n|}
\big)^{1-a}
\end{align*}
where we also used \eqref{e:PathSetFunc}. 
Thus we have that 
$$
\big( \frac{\mu_{1-a,n}(K)}{\mu_{1-a,n}( {\rm B}_2^n )} \big)^{\frac1{1-a}} 
\le 
\frac{|K|}{|{\rm B}_2^n|}
$$
for all $a\in (0,1)$. 
By using this and then taking $a\to 1$ in \eqref{e:SetInvSant}, with $\beta = \kappa^{-2}$ in mind, we conclude the proof. 
\end{proof}

Another interesting consequence from Theorem \ref{t:RegInvSanSet} appears when $a = \frac12$. 
To motivated,  it is worth to mention a generalization of the Blaschke--Santal\'{o} inequality \eqref{e:BS} to the weighted volume product. 
Then one may ask for which measure $\mu$ and convex body $K$, 
\begin{equation}\label{e:BSWeight}
	v_\mu (K) \le v_\mu({\rm B}_2^n)
\end{equation}
holds true. This problem was raised by Fradelizi--Meyer \cite{FraMeyMathZ}  and they obtained partial progress. 
It was also investigated by Klartag \cite{Kla07}. For instance, he proved  that \eqref{e:BSWeight} holds for all symmetric $K$ and certain rich family of symmetric measures $\mu$.  This class of measures particularly contains, as a typical example,  the Cauchy type distribution 
\begin{equation}\label{e:DefCauchy}
d\mu_{n}(x) \coloneqq \frac1{(1+|x|^2)^\frac{n+2}2}\, dx.
\end{equation}
We here ask an analogous question on the lower bound of $v_{\mu_{n}}(K)$. 
As a simple observation,  one realizes that $\lim_{r\to 0,\infty} v_{\mu_{n}}(r{\rm B}_2^n) = 0$ since $\mu_{n}$ is a finite measure on $\mathbb{R}^n$. 
Thus there is no hope to expect nontrivial global lower bound of $v_{\mu_{n}}(K)$. 
In other words, one needs to fix a ``scale" of the size of convex bodies to make the problem well-defined. 
We fix the scale by assuming \eqref{e:CondCurvNov} with $\Lambda = \kappa^{-1}{\rm id}$ and obtain the following by taking $a=\frac12$ in Theorem \ref{t:RegInvSanSet}. %\footnote{Another way of formulating the problem is to use the maximal and minimal (L\"{o}wner) ellipsoid.  For example, one can ask the lower bound of $v_{\mu_n}(K)$ for convex bodies $K$ whose maximal ellipsoid is $r {\rm B}_2^n$ and minimal ellipsoid is $R{\rm B}_2^n$ for some fixed $r<R$.  }. 

\begin{corollary}\label{t:WeightInv}
Let $n \ge 2$ and $\kappa \in (0,1]$. 
Then for any convex body $K$ with $0 \in \mathrm{int}\, K$ satisfying \eqref{e:CondCurvNov} with $\Lambda = \kappa^{-1}{\rm id}$, 
we have 
$$
v_{\mu_n}(K)
\ge 
\big( \frac{4}{(\kappa+\kappa^{-1})^2} \big)^\frac{n}{2} 
v_{\mu_n}( {\rm B}_2^n).
$$
Moreover this is sharp in the sense that equality is attained when $K = \frac1\kappa {\rm B}_2^n$.  
\end{corollary}

We give a functional form of Theorem \ref{t:ImpInv} which is a direct consequence from Theorem \ref{t:Func19Oct}.  
\begin{corollary}\label{t:RegInvS}
Let $n\ge 2$ and $\kappa \in (0,1]$. Then we have 
\begin{equation}\label{e:RegInvS}
\int_{\mathbb{R}^n} e^{-\psi}\, dx 
\int_{\mathbb{R}^n} e^{-\psi^*}\, dx 
\ge 
(2\pi)^n 
\big(
\kappa^2
e^{1-\kappa^2}
\big)^{\frac{n}{2}}
\end{equation}
for all $\psi \in C^2(\mathbb{R}^n\setminus\{0\}) \cap C(\mathbb{R}^n)$ satisfying 
$$
\nabla^2 \psi(x),\; \nabla^2 \psi^*(x) \ge \kappa\, {\rm id},\;\;\; \forall x \in \mathbb{R}^n \setminus\{0\}.
$$
\end{corollary}

\begin{remark}
Since Theorem \ref{t:FHC} and Lemma \ref{l:Useful} holds under semi-log-convexity/concavity without smoothness, we can also show Theorem \ref{t:ImpInv} and Corollary \ref{t:RegInvS} without smoothness. 
In fact, the similar argument in the proof of Theorem \ref{t:RegInvSanSet} yields \eqref{e:RegInvS} for all $\kappa$-semi-convex $\psi \colon \mathbb{R}^n \to \mathbb{R}$ whose polar $\psi^*$ is also $\kappa$-semi-convex. 
\end{remark}

\subsection{Further consequences}\label{S4.7}
Here we provide alternative view point of Theorem \ref{t:ImpInv} by using nations of uniform convexity and uniform smoothness, see \cite{BCL,Oh,Sch} for its historical background. 
We say that $K$ or its associate gauge function $\| \cdot \|_K$ is 2-uniformly convex (or uniformly convex in short) if there exists some finite constant $C>0$ such that 
\begin{equation}\label{e:2UC}
\|\frac{u+w}{2}\|_K^2 \le \frac{1}{2}\|u\|_K^2 + \frac{1}{2}\|w\|_K^2 - \frac{1}{4C} \|u-w\|_K^2, \quad \forall u,w \in \R^n.
\end{equation}
In general, \eqref{e:2UC} always holds with $C = \infty$, and thus \eqref{e:2UC} with a finite $C>0$ means stronger convexity. 
We denote the best constant $C>0$ satisfying \eqref{e:2UC} by $C_K>0$ (exactly attained), namely 
$$
C_K \coloneqq \min \{ C>0 : \text{ \eqref{e:2UC} holds with $C>0$}\}. 
$$
One can also see that $C_K \ge 1$ by inserting $w=0$ to \eqref{e:2UC}.
Similarly we say that $\|\cdot\|_K$ is 2-uniformly smooth (or uniformly smooth in short) if there exists some finite constant $S>0$ such that 
\begin{equation}\label{e:2US}
\|\frac{u+w}{2}\|_K^2 \ge \frac{1}{2}\|u\|_K^2 + \frac{1}{2}\|w\|_K^2 - \frac{S}{4} \|u-w\|_K^2, \quad \forall u,w \in \R^n.
\end{equation}
We denote 
$$
S_K \coloneqq \min \{ S>0 : \text{\eqref{e:2US} holds with $S>0$}\}
$$
and see that $S_K \ge 1$ by taking $w=0$. 
It is worth to mention that the 2-uniformly smoothness is a dual notion on the 2-uniformly convexity in the sense that 
$
S_K = C_{K^\circ}. 
$
We mention the work by Klartag--Milman \cite{KlaMil08} where they investigated how 2-uniformly convex bodies are ``well-behaved" and in particular ensured the hyperplane conjecture for convex bodies with appropriate bound on $C_K$, see also \cite{Sch}.  Recently Klartag \cite[Corollary 1.2]{KlaAdv18} observed that the stronger quantitative version of the hyperplane conjecture implies Mahler's conjecture.  Although Klartag--Milman's result on the hyperplane conjecture under the 2-uniformly convexity cannot be directly applied to derive some inverse Santal\'{o} inequality, this raise an investigation of Mahler's conjecture under the 2-uniformly convexity. 
For this purpose, it is worth to mention that if $\| \cdot \|_K$ is the Minkowski norm, then there is another characterization of uniformly convexity and uniformly smoothness. 
In such a case, when $n \ge 2$, it is known that $\| \cdot \|_K$ is 2-uniformly convex with $C \ge 1$ if and only if 
\begin{equation}\label{e:2UCHess}
\langle \nabla^2 (\frac{1}{2}\|\cdot\|_K^2)(x)w, w\rangle \ge \frac{1}{C} \|w\|_K^2, \;\;\; \forall x, w \in \mathbb{R}^n \setminus \{0\}.
\end{equation}
Similarly, $\| \cdot \|_K$ is 2-uniformly smooth with $S \ge 1$ if and only if 
\begin{equation}\label{e:2USHess}
\langle \nabla^2 (\frac{1}{2}\|\cdot\|_K^2)(x)w, w\rangle \le S \|w\|_K^2, \;\;\; \forall x, w \in \mathbb{R}^n \setminus \{0\}.
\end{equation}
Conditions \eqref{e:2UCHess} and \eqref{e:2USHess} are similar to \eqref{e:CondCurvature}. To clarify this relation, we introduce the Banach--Mazur distance between convex bodies $K, L \subset \mathbb{R}^n$ including the origin in their interior by\footnote{Remark that the definition above is slightly different from the standard one since we need to consider affine maps instead of linear ones,  but our definition coincide with the standard one when $K, L$ are symmetric.  We also note that $d_{\mathrm{BM}}(K, L)$ is finite since $K, L$ has the origin in their interior.  
} 
$$
d_{\mathrm{BM}}(K, L) \coloneqq \inf \{ r \ge 1 :\; TL \subset K \subset r TL \text{ for some linear isomorphism}\; T \text{ on $\mathbb{R}^n$} \}.  
$$
If one concerns the volume product $v(K)$, then one may assume 
$$
{\rm B}_2^n
\subset 
K 
\subset 
d_{\rm BM}(K, {\rm B}_2^n) {\rm B}_2^n
$$
without loss of generality as $v(K)$ is linear invariant. 
In such case, it follows that 
$$
\frac1{d_{\rm BM}(K, {\rm B}_2^n)}|x|
\le 
\|x\|_K
\le 
|x|
$$
and hence 
$$
\frac{1}{C_K  d_{\rm BM}(K, {\rm B}_2^n)^2}{\rm id}
\le 
\nabla^2 \big( \frac12 \|\cdot\|_K^2 \big)(x)
\le S_K 
$$
for all $x \in \mathbb{R}^n\setminus \{0\}$ and $K$ with $0 \in {\rm int}\, K$ whose norm generates a Minkowski norm. In view of $S_K = C_{K^\circ}$, Theorem \ref{t:ImpInv} yields the following.

\begin{corollary}
	Let $n\ge2$. Then for any convex body $K$ with $0 \in {\rm int}\, K$ whose norm generates a Minkowski norm, we have 
	$$
	v(K)
	\ge 
	\big( \kappa_K^2 e^{1- \kappa_K^2}
	\big)^\frac{n}{2}
	v({\rm B}_2^n)
	$$
	where $\kappa_K$ is given by $\kappa_K^{-2} \coloneqq C_K C_{K^\circ} d_{\rm BM}(K, {\rm B}_2^n)^2$. 
\end{corollary}

%As a basic property it is known that $\| \cdot \|_K$ comes from an inner product if and only if $C_K=1$ and if and only if $S_K=1$ (see \cite[Proposition 1.6]{Oh}). 
%Thus $C_K$ and $S_K$ characterize Hilbert spaces among Banach spaces. 
%For example, $\ell^p$-norm on $\mathbb{R}^n$ is 2-uniformly convex with $C = (p-1)^{-1}$ if $ p \in (1, 2]$, and 2-uniformly smooth with $S= p-1$ if $p \in [2, \infty)$, see \cite[Proposition 3]{BCL}. 
%Intuitively speaking,  as in \cite{BCL},  that $K$ is 2-uniformly convex means $K$ is uniformly free of ``flat spots". Similarly that $K$ is 2-uniformly smooth, as a dual property,  means $K$ is uniformly free of ``corners",  or in other words $K^\circ$ is uniformly free of flat spots.  Such flat points and corners are appeared on ${\rm B}_{p}^n$, which is a unit ball of $\ell^p$-norm on $\mathbb{R}^n$,  with $p=1,\infty$. 
%We will exhibit more concrete statement how $C_K$ and $S_K$ control the curvature of $\partial K$ in forthcoming Proposition \ref{Prop:Curv} although it might be folklore. 
%Furthermore it is important to us that $C_K, S_K$ are linear invariant since $\| \cdot \|_{TK} = \| T^{-1} \cdot \|_K$ for any invertible linear map $T$ on $\mathbb{R}^n$. 
%Therefore these notions are well-fitted to our purpose.  We now present our first result. 

\section{Appendix}\label{S6}
\subsection{On the formulation of Mahler's conjecture \eqref{e:MahGene}}\label{S6.0}
For the sake of convenience, let us state again the general case of Mahler's conjecture together with the case of equality. 
\begin{conjecture}\label{Conj:Mah2}
	For any convex body $K \subset \mathbb{R}^n$ with $b_K = 0$, 
	\begin{equation}\label{e:MahGene2}
	v(K)\ge v(\Delta^n_0) = \frac{(n+1)^{n+1}}{(n!)^2}.
	\end{equation}
	Moreover, equality in \eqref{e:MahGene2} is achieved if and only if $K = \Delta^n_0$. 
\end{conjecture}
This formulation might not be standard in the context.  
For instance, Fradelizi--Meyer's formulation \cite{FM} is as follows. 
For any convex body $K$, let $P(K)$ be the affine invariant volume product defined by $P(K) \coloneqq \min_{z\in\mathbb{R}^n} |K||K^z|$ where $K^z \coloneqq \{ y \in \mathbb{R}^n : \langle y-z, x - z \rangle\le 1,\; \forall x \in K\}$. 
\begin{conjecture}[\cite{FM}]\label{Conj:FMMah}
For any convex body $K \subset \mathbb{R}^n$, 
\begin{equation}\label{e:FMMahler}
P(K) \ge P(\Delta^n) = \frac{(n+1)^{n+1}}{(n!)^2},   
\end{equation}
where $\Delta^n$ is an arbitrary non-degenerate simplex in $\mathbb{R}^n$ which is not necessarily centered. 
Moreover, equality \eqref{e:FMMahler} is achieved if and only if $K = \Delta^n$. 
\end{conjecture}
These two conjectures are indeed equivalent. 
For the sake of completeness, we give a proof of this fact in below. 
We will make use of the Santal\'{o} point of $K$ denoted by $z_K$ which is a unique point attaining the minimum of $P(K)$. It is worth to mention that 
\begin{equation}\label{e:SantaloBary}
	x = z_K\; \Leftrightarrow\; b_{(K - \{x\})^\circ} = 0, 
\end{equation}
for $x \in {\rm int}\, K$, see \cite{Sch}.
From this property, we can easily see that 
\begin{equation}\label{e:Simplex}
	v(\Delta^n_0) = P(\Delta^n). 
\end{equation}
In fact, if we let $L = (\Delta^n_0)^\circ$ which is again non-degenerate simplex but not necessarily centered, then $b_{L^\circ} = 0$ from the definition of $L^\circ = \Delta^n_0$. Hence \eqref{e:SantaloBary} reveals that $z_L = 0$ from which we obtain \eqref{e:Simplex} as 
$$
v(\Delta^n_0)
=
| \Delta^n_0 | | (\Delta^n_0 )^\circ|
=
P(\Delta^n_0)
=
P(\Delta^n). 
$$
With these in mind, suppose first one could prove Conjecture \ref{Conj:Mah2}. 
Let us show the inequality \eqref{e:FMMahler}. If one notices that $K^z = (K - \{z\})^\circ + \{z\}$, then the translation invariance of the Lebesgue measure gives that 
$$
P(K) = |K||K^{z_K}| = v(K - \{z_K\}). 
$$ 
If we let $L= (K - \{z_K\})^\circ$, then $b_L =0$ by virtue of \eqref{e:SantaloBary}. We also have $v(K-\{z_K\}) = v(L^\circ) = v(L)$. 
Hence we may apply the assumption \eqref{e:MahGene2} to $L$ to obtain \eqref{e:FMMahler} as 
\begin{equation}\label{e:IneqNov}
	P(K) = v(L) \ge v(\Delta^n_0) = P(\Delta^n)
\end{equation}
where we also used \eqref{e:Simplex}. 
Next suppose $K$ attains equality in \eqref{e:FMMahler} and show that $K$ must be $\Delta^n$. This means the inequality \eqref{e:IneqNov} must be equality. Hence if  equality part of Conjecture \ref{Conj:Mah2} could be proved then $L = (K-\{z_K\})^\circ$ must be $\Delta^n_0$. This means $K = (\Delta^n_0)^\circ + \{z_K\}$ which is $\Delta^n$. 
This completes the implication Conjecture \ref{Conj:Mah2} $\Rightarrow$ Conjecture \ref{Conj:FMMah}. 

Next let us prove the reverse implication. 
If the inequality \eqref{e:FMMahler} could be proved then one can obtain the inequality part of Conjecture \ref{Conj:Mah2} as follows. 
Since $b_K = 0$, we know from \eqref{e:SantaloBary} that $z_L=0$ for $L = K^\circ$. This shows that 
$$
v(K) = v(L) = |L||L^{z_L}| = P(L) \ge P(\Delta^n) = v(\Delta^n_0)
$$
by \eqref{e:Simplex} and the assumption \eqref{e:FMMahler}. 
Suppose next that $K$ is $b_K=0$ and attains equality in \eqref{e:MahGene2}. Then the inequality for $L$ we used just above must be equality. Hence equality part of Conjecture \ref{Conj:FMMah} ensures that $K^\circ = L = \Delta^n$. Namely, $K = (\Delta^n)^\circ$ which is again the non-degenerate simplex. Since we assumed $b_K=0$, this concludes $K = \Delta^n_0$. This completes the proof of the equivalence.

\subsection{The assumption \eqref{e:CondCurvNov} and the Gaussian curvature}\label{S6.6}
Here we exhibit an explicit link between our curvature condition \eqref{e:CondCurvNov} and principal curvatures as well as Gaussian curvature. 
For the sake of simplicity, we only consider the case $\Lambda = \kappa^{-1} {\rm id}$ here.

\begin{proposition}\label{Prop:Curv}
Suppose $K$ is a convex body with $0 \in {\rm int}\, K$ satisfying \eqref{e:CondCurvNov} with $\Lambda = \kappa^{-1}{\rm id}$. 
Also we assume ${\rm B}_2^n \subset K$. 
Then 
$$
\kappa^2 \le \inf_{x \in \partial K} \lambda_{\rm min}(x) 
\le  \sup_{x \in \partial K} \lambda_{\rm max}(x)
\le 
d_{\rm BM}(K, {\rm B}_2^n)
$$ 
where $ \lambda_{\rm max}(x),  \lambda_{\rm min}(x)$ are maximal, minimal eigenvalues of the principle curvatures of $\partial K$ at $x$, respectively. 
\end{proposition}

\begin{proof}
Take arbitrary $x_0 \in \partial K$. 
By the rotation we may suppose $(x_0)_n <0$ and the tangent space $T_{x_0} \partial K = \langle e_1,\ldots, e_{n-1}\rangle$. 
In such case, the principle curvatures of $\partial K$ at $x_0$ is given by 
$$
\frac1{\| \mathbf{n}_{x_0}\|_{K^\circ}} 
\nabla_{x_1,\ldots,x_{n-1}}^2 \big( \frac12 \|\cdot\|_K^2 \big) (x_0), 
$$
where $\mathbf{n}_{x_0}$ denotes the outer unit normal vector of $\partial K$ at $x_0$.
%thanks to Lemma \ref{l:Curv1}. 
%Now we recall that $C_K,S_K$ satisfies \eqref{e:2UCHess} and \eqref{e:2USHess} respectively.
%In addition we assume the maximal ellipsoid of $K$ is an unit ball and hence ${\rm B}_2^n \subset K \subset d_{\rm BM}(K,{\rm B}_2^n){\rm B}_2^n$. This in particular shows that 
%$$
%\frac{1}{d_{\rm BM}(K,{\rm B}_2^n)}| x | \le \| x \|_K \le |x|
%$$ 
%for all $x \in \mathbb{R}^n$. 
%By combining this with \eqref{e:2UCHess} and \eqref{e:2USHess}, 
%$$
%\frac1{C_K d_{\rm BM}(K,{\rm B}_2^n)^2 } {\rm id}_{\mathbb{R}^n}
%\le 
%\nabla^2_{x_1,\ldots,x_n} \big( \frac12 \|\cdot \|_K^2 \big) (x_0) 
%\le S_K {\rm id}_{\mathbb{R}^n}. 
%$$
This and \eqref{e:CondCurvNov}, or its equivalent form \eqref{e:CondCurvature2}, show that 
$$
\kappa^2  
\| \mathbf{n}_{x_0}\|_{K^\circ}  
\le 
\lambda_{\rm min}(x_0)
\le 
\lambda_{\rm max}(x_0)
\le 
\| \mathbf{n}_{x_0}\|_{K^\circ}.   
$$ 
To obtain an uniform control of $\| \mathbf{n}_{x_0}\|_{K^\circ}$, we appeal to ${\rm B}_2^n \subset K \subset d_{\rm BM}(K,{\rm B}_2^n){\rm B}_2^n$ or equivalently $ \frac1{d_{\rm BM}(K,{\rm B}_2^n)}{\rm B}_2^n \subset K^\circ \subset {\rm B}_2^n$ which particularly yields 
$$
|x| \le \| x \|_{K^\circ} \le d_{\rm BM}(K,{\rm B}_2^n) |x|
$$
for all $x \in \mathbb{R}^n$. In particular 
$$
1 \le \|  \mathbf{n}_{x_0} \|_{K^\circ} \le d_{\rm BM}(K,{\rm B}_2^n) 
$$
since $| \mathbf{n}_{x_0}| =1$ and this concludes the proof.  
\end{proof}

\subsection{Relation between hypercontractivity and Brascamp--Lieb inequality}\label{S.6HC}
We give an explicit way of understanding the hypercontractivity inequality  as the Brascamp--Lieb inequality, see \cite{BW,Lieb}. 
\begin{enumerate}
	\item 
	Let $s>0$ and $p,q\in\mathbb{R}\setminus\{0\}$ be arbitrary. For $c_1,c_2$ and $\mathcal{Q}$ defined by \eqref{e:HCBLLink} and nonnegative $f \in L^1(\gamma)$, we have 
\begin{equation}\label{e:Id10Nov}
	\big\| P_s \big[f^\frac1p \big] \big\|_{L^q(\gamma)}
	=
	\big(
	\frac{(2\pi)^{\frac12(c_1+c_2)-1}} {\sqrt{1-e^{-2s}}}
	\big)^n
	\int_{\mathbb{R}^{2n}}
	e^{-\pi \langle x,\mathcal{Q}x\rangle}
	\prod_{j=1,2} f_j(x_j)^{c_j}\, dx
\end{equation}
where 
$
f_1\coloneqq f\cdot \gamma$ and $f_2\coloneqq 
\big\| P_s \big[f^\frac1p \big] \big\|_{L^q(\gamma)}^{-q}
P_s \big[f^\frac1p \big]^q \cdot \gamma.
$

\item 
Let $s>0$ and $c_1,c_2\in \mathbb{R}\setminus\{0\}$ be arbitrary. 
For $p,q$ and $\mathcal{Q}$ defined by \eqref{e:HCBLLink} and nonnegative $f_1,f_2 \in L^1(dx)$, we have 
\begin{equation}\label{e:Id10Nov2}
	\int_{\mathbb{R}^{2n}}
	e^{-\pi \langle x,\mathcal{Q}x\rangle}
	\prod_{j=1,2} f_j(x_j)^{c_j}\, dx
	=
	\big(
	\frac{(2\pi)^{\frac12(c_1+c_2)-1}} {\sqrt{1-e^{-2s}}}
	\big)^{-n}
	\int_{\mathbb{R}^n} 
	P_s \big[\big( \frac{f_1}{\gamma} \big)^\frac1p \big] 
	\big( \frac{f_2}{\gamma} \big)^{\frac1{q'}}\, d\gamma. 
\end{equation}
\end{enumerate}

\subsection{Gaussian integral}\label{S6.1}
We give the identity for the explicit form of $P_s f$ for Gaussian inputs as follows. 
\begin{lemma}\label{l:GaussInt1}
Let $s>0$, $a\in \mathbb{R}$ satisfy $a_s \coloneqq 1+a(1-e^{-2s})>0$ and $b \in \mathbb{R}^n$. 
Then 
\begin{equation}\label{e:GaussInt1}
P_s
\big[
e^{ -\frac12a|\cdot|^2 + \langle b, \cdot\rangle }
\big](x)
=
a_s^{-\frac{n}{2}}
{\rm exp}\,  
\bigg[
-\frac1{2a_s} 
\big(
ae^{-2s} |x|^2 
- 
2e^{-s} \langle b,x\rangle
-
(1-e^{-2s}) |b|^2
\big)
\bigg].
\end{equation}
\end{lemma}
This is an elementary Gaussian calculus so we omit the proof. 

\if0
\subsection{Linear invariance}\label{S6.2}
\begin{lemma}\label{l:LinearInv}
Let $s >0$, $p=1-e^{-2s}$ and $q = 1-e^{2s}$. 
For nonnegative $f \in L^1(\gamma)$ and $A \in GL(n)$ with ${\rm det}\, A>0$, let 
$$
f_A(x)\coloneqq f(Ax) e^{ - \frac12 ( |Ax|^2  - |x|^2) },\;\;\; x\in\mathbb{R}^n. 
$$
Then 
$$
\big\|
P_s \big[f_A^\frac1p\big]
\big\|_{L^q(\gamma)}
= 
({\rm det}\, A)^{-1+\frac1q} 
\big\|
P_s \big[f^\frac1p\big]
\big\|_{L^q(\gamma)},
\; 
\big(
\int_{\mathbb{R}^n} 
f_A \, d\gamma 
\big)^\frac1p
= 
({\rm det}\, A)^{-\frac1p}
\big(
\int_{\mathbb{R}^n} 
f \, d\gamma 
\big)^\frac1p. 
$$
\end{lemma}

\begin{proof}
It suffices to notice that 
$$
P_s\big[f_A^\frac1p](x) 
= 
({\rm det}\, A)^{-1} 
P_s\big[f^\frac1p](A^*x)
e^{ - \frac12 ( |A^*x|^2  - |x|^2) }
$$
where $A^* \coloneqq (A^{-1})^{\rm t}$. See Hiroshi's note Lemma 16.1.
\end{proof}

\fi

\subsection{Proof of \eqref{e:Admissible} and \eqref{e:Admissible2}}\label{App}
In this appendix, we write 
$$
\big\|
P_s\big[ \big(\frac{\gamma_\beta}{\gamma}\big)^\frac1p \big] 
\big\|_{L^q(\gamma)}
=
\beta^{-\frac{n}{2p}} 
\phi_1(\beta)^{\frac{n}2(\frac1q - 1)}
\phi_2(\beta)^{-\frac{n}{2q}},
$$
where 
$$
\phi_1(\beta):= 1+ (\frac1\beta - 1) \frac{1-e^{-2s}}{p},
\;\;\;
\phi_2(\beta):=1+ \frac{1-e^{-2s}+qe^{-2s}}p (\frac1\beta -1). 
$$
Note that this identity makes sense only when $\beta$ is such that $\phi_1(\beta), \phi_2(\beta)>0$, otherwise $\big\|
P_s\big[ \big(\frac{\gamma_\beta}{\gamma}\big)^\frac1p \big] 
\big\|_{L^q(\gamma)} $ is not well-defined. 
The sufficiency of conditions of $p,q$ in \eqref{e:Admissible} and \eqref{e:Admissible2} is clear from \eqref{e:ScaleGauss} and H\"{o}lder's inequality so we have only to show that for given $q<0<p<1$ satisfying $\frac{q-1}{p-1}>e^{2s}$, 
\begin{equation}\label{e:NecRev}
	\inf_{\beta>0}
	\big\|
P_s\big[ \big(\frac{\gamma_\beta}{\gamma}\big)^\frac1p \big] 
\big\|_{L^q(\gamma)}
>0
\;\;\;
\Rightarrow
\;\;\;
1-e^{2s} \le q < 0 < p \le 1-e^{-2s}
\end{equation}
and 
\begin{equation}\label{e:NecFor}
	\sup_{\beta>0}\big\|
P_s\big[ \big(\frac{\gamma_\beta}{\gamma}\big)^\frac1p \big] 
\big\|_{L^q(\gamma)}
<\infty 
\;\;\;
\Rightarrow
\;\;\;
q\le 1-e^{2s},\;  1-e^{-2s} \le p.
\end{equation}

To this end, we first investigate the behavior of $\phi_1$ and $\phi_2$. 
The $\phi_1$ has two different behaviors with the threshold at $ p = 1-e^{-2s}$. In fact, one can see from the definition and $p>0$ that 
\begin{equation}\label{e:phi_1Small}
0<p<1-e^{-2s} 
\;\;\;
\Rightarrow
\;\;\;
\exists \beta_1>1:\; 
\begin{cases}
\phi_1(\beta)>0,\;\;\; &{\rm if}\;\;\; 0<\beta<\beta_1,\\
\phi_1(\beta)=0,\;\;\; &{\rm if}\;\;\; \beta=\beta_1, \\
\phi_1(\beta)<0,\;\;\; &{\rm if}\;\;\;\beta>\beta_1,
\end{cases}
\end{equation}
and 
\begin{equation}\label{e:phi_1Big}
1-e^{-2s} < p < 1
\;\;\;
\Rightarrow
\;\;\;
\phi_1(\beta)\ge 1 - \frac{1-e^{-2s}}p >0,\; \forall \beta>0. 
\end{equation}
For $\phi_2$, note that we always have  $p >1-e^{-2s} + qe^{-2s}$ thanks to the assumption $\frac{q-1}{p-1} > e^{2s}$ in \eqref{e:phi_2Big}. 
With this and $q<0<p<1$ in mind, the $\phi_2$ has two different behaviors:
\begin{equation}\label{e:phi_2small}
q<1-e^{2s}
\;\;\;
\Rightarrow
\;\;\;
\exists \beta_2<1:\; 
\begin{cases}
\phi_2(\beta)<0,\;\;\; &{\rm if}\;\;\; 0<\beta<\beta_2,\\
\phi_2(\beta)=0,\;\;\; &{\rm if}\;\;\; \beta=\beta_2, \\
\phi_2(\beta)>0,\;\;\; &{\rm if}\;\;\;\beta>\beta_2,
\end{cases}
\end{equation}
and 
\begin{equation}\label{e:phi_2Big}
q>1-e^{2s}
\;\;\;
\Rightarrow
\;\;\;
\phi_2(\beta) \ge 1 - \frac{1-e^{-2s} + qe^{-2s}}p>0,
\; \forall \beta>0. 
\end{equation} 
Remark that we used $p >1-e^{-2s} + qe^{-2s}$ for \eqref{e:phi_2Big}. 
Let us see \eqref{e:NecRev}. Suppose $q<1-e^{2s}$. In this case, because of  $\beta_2<1<\beta_1$ and $q<0$, we know that $\phi_1(\beta_2) \in (0,\infty)$ and hence 
$$
\lim_{\beta \downarrow \beta_2} 
\big\|
P_s\big[ \big(\frac{\gamma_\beta}{\gamma}\big)^\frac1p \big] 
\big\|_{L^q(\gamma)}
=
\beta_2^{-\frac{n}{2p}}
\phi_1(\beta_2)^{\frac{n}{2}(\frac1q-1)} 
\lim_{\beta \downarrow \beta_2}
\phi_2(\beta)^{-\frac{n}{2q}}
=0.
$$
Suppose next $p>1-e^{-2s}$. In this case, we know that 
$$
\inf_{\beta>0} \phi_1(\beta)= \lim_{\beta\to\infty} \phi_1(\beta) =  1 - \frac{1-e^{-2s}}p >0,
$$
and
$$
\inf_{\beta>0} \phi_2(\beta)= \lim_{\beta\to\infty} \phi_2(\beta) =  1 - \frac{1-e^{-2s}+qe^{-2s}}p >0.
$$
Hence, in view of $p>0$,  
$$
\lim_{\beta \to \infty} 
\big\|
P_s\big[ \big(\frac{\gamma_\beta}{\gamma}\big)^\frac1p \big] 
\big\|_{L^q(\gamma)}
=
C_{p,q,s}\lim_{\beta\to\infty}\beta^{-\frac{n}{2p}} 
=0.
$$
This concludes \eqref{e:NecRev}. 
We next show \eqref{e:NecFor}. Suppose $p<1-e^{-2s}$. Recalling \eqref{e:phi_2small}, \eqref{e:phi_2Big} and \eqref{e:phi_1Small} with $\beta_1>1$, we see that $\phi_2(\beta_1) \in (0,\infty)$ and hence 
$$
\lim_{\beta \uparrow \beta_1} \big\|
P_s\big[ \big(\frac{\gamma_\beta}{\gamma}\big)^\frac1p \big] 
\big\|_{L^q(\gamma)}
=
\beta_1^{-\frac{n}{2p}} 
\phi_2(\beta_1)^{- \frac{n}{2q}}
\lim_{\beta \uparrow \beta_1}
\phi_1(\beta)^{\frac{n}2(\frac1q-1)}
=\infty
$$
by virtue of $q<0$. Suppose next $q>1-e^{2s}$. In this case, in view of \eqref{e:phi_2Big} and $p<1$, we conclude that 
$$
\lim_{\beta \downarrow 0} 
\big\|
P_s\big[ \big(\frac{\gamma_\beta}{\gamma}\big)^\frac1p \big] 
\big\|_{L^q(\gamma)}
=
\lim_{\beta \downarrow 0}
\beta^{\frac{n}2(1-\frac1p)}
(\frac{1-e^{-2s}}p)^{\frac{n}2(\frac1q-1)}
(\frac{1-e^{-2s} + qe^{-2s}}p)^{-\frac{n}{2q}}
=\infty.
$$
This concludes \eqref{e:NecFor}. 

\subsection{Proof of $v_0 \in L^1(dx)$ in Theorem \ref{t:RevRevClosureGene}}\label{S6.3}
Suppose $v_0$ satisfies assumptions in Theorem \ref{t:RevRevClosureGene} and then show $v_0 \in L^1(dx)$. 
From the assumption we know that $\log\, \frac{\gamma_\beta}{v_0}$ is convex and hence subdifferentiable in the sense that 
$$
\partial \big( \log\, \frac{\gamma_\beta}{v_0} \big)(x) 
:= 
\{
y\in \mathbb{R}^n: 
\log\, \frac{\gamma_\beta}{v_0} (z) 
\ge 
\log\, \frac{\gamma_\beta}{v_0} (x)
+ 
\langle y, z-x\rangle ,
\;
\forall z\in \mathbb{R}^n 
\}
\neq 
\emptyset
$$
for all $x\in \mathbb{R}^n$. 
In particular we find $y_0 \in \partial  \big( \log\, \frac{\gamma_\beta}{v_0} \big)(0) $. 
From the definition we have that 
$$
\frac{\gamma_\beta}{v_0}(z) 
\ge 
\frac{\gamma_\beta}{v_0}(0)
e^{\langle y_0,z\rangle} 
$$
and hence 
$$
v_0(z) \le \frac{v_0}{\gamma_\beta}(0)
e^{- \langle y_0,z\rangle} 
\gamma_\beta(z) \in L^1(dx). 
$$

\subsection{Differentiability of concave function}\label{S6.4}
\begin{lemma}\label{l:Diffable}
Let $\phi\colon \mathbb{R}^n \to \mathbb{R}$ be concave and suppose there exists $y_0 \in \mathbb{R}^n$ such that $Q_1\phi(x_0) = \phi(y_0) + \frac12|x_0-y_0|^2$ for given $x_0 \in \mathbb{R}^n$. 
Then $\phi$ is differentiable at $y_0$ and 
$$
\nabla \phi(y_0) = x_0 -y_0. 
$$
\end{lemma}

\begin{proof}
From the assumption 
$$
\phi(y) +\frac12|x_0-y|^2 \ge 
\phi(y_0) +\frac12|x_0-y_0|^2
$$
for all $y$. This means 
\begin{equation}\label{e:LBD}
\phi(y) - \phi(y_0)
\ge 
\langle x_0, y-y_0\rangle 
- \frac12|y|^2 + \frac12|y_0|^2.
\end{equation}
On the other hand $-\phi$ is convex and hence its subdifferential $\partial (-\phi)(y_0)$ is not empty. 
We pick arbitrary $w \in \partial (-\phi)(y_0)$ and show $w = y_0 -x_0$. 
In fact $w \in \partial (-\phi)(y_0)$ yields 
$$
\phi(y) - \phi(y_0) 
\le 
- \langle w, y - y_0\rangle
$$
for all $y$. Combining this and \eqref{e:LBD} we see that 
$$
\frac12 | y - (x_0+w)|^2 - \frac12|x_0+w|^2 + \langle x_0 +w, y_0 \rangle -\frac12|y_0|^2\ge0
$$
for all $y$. 
In particular for $y = x_0 +w$ it follows that $|x_0 + w -y_0|^2\le 0$ which in turn shows $w = y_0 -x_0$. 
Since we proved $\partial(-\phi)(y_0) = \{ y_0 -x_0 \}$ we conclude that $-\phi$ is differentiable at $y_0$ and $\nabla(-\phi)(y_0) = y_0 -x_0$.
\end{proof}

%\begin{proposition}\label{Prop:Vague}
%For appropriately well-behaved function $f: \mathbb{R}^n\to (0,\infty)$, define 
%$$
%\psi(x) = \psi_f(x) \coloneqq \frac12|x|^2 - \log\, f(x),\;\;\; x\in \mathbb{R}^n. 
%	\begin{enumerate}
%		\item 
%		If $f$ satisfies 
%		\begin{equation}\label{e:RHC_BS}
%		\big\| 
%		P_s \big[ f^{\frac1{p_s}} \big]
%		\big\|_{L^{q_s}(\gamma)}
%		\ge 
%		{\rm BS}_{s,f}^\frac{1}{p_s} 
%		\big(
%		\int_{\mathbb{R}^n}
%		f\, d\gamma
%		\big)^\frac1{p_s},
%		\;\;\;
%		p_s\coloneqq 1-e^{-2s},\; q_s\coloneqq - 2s
%		\end{equation}
%		for some constant ${\rm BS}_{s,f}>0$ and all small $s>0$, then 
%		\begin{equation}\label{e:FuncBS}
%			\int_{\mathbb{R}^n} e^{-\psi}\, dx
%			\int_{\mathbb{R}^n} e^{-\psi^*}\, dx
%			\le 
%			(2\pi)^n
%			\lim_{s\downarrow0} {\rm BS}_{s,f}^{-1}.
%		\end{equation}
%		{\color{red}NEED TO CHECK IF IT IS CORRECT}
%		\item 
%		If $f$ satisfies 
%		\begin{equation}\label{e:FHC_Ma}
%		\big\| 
%		P_s \big[ f^{\frac1{p_s}} \big]
%		\big\|_{L^{q_s}(\gamma)}
%		\le 
%		{\rm IS}_{s,f}^\frac{1}{p_s} 
%		\big(
%		\int_{\mathbb{R}^n}
%		f\, d\gamma
%		\big)^\frac1{p_s},
%		\;\;\;
%		p_s\coloneqq 1-e^{-2s},\; q_s\coloneqq 1-e^{2s}
%		\end{equation}
%		for some constant ${\rm IS}_{s,f}>0$ and all small $s>0$, then 
%		\begin{equation}\label{e:FuncMa}
%			\int_{\mathbb{R}^n} e^{-\psi}\, dx
%			\int_{\mathbb{R}^n} e^{-\psi^*}\, dx
%			\ge 
%			(2\pi)^n 
%			\lim_{s\downarrow0} {\rm IS}_{s,f}^{-1}.
%		\end{equation}
%	\end{enumerate}
%\end{proposition}

\subsection{Proof of the equivalence between \eqref{e:CondCurvNov} with $\Lambda = \kappa^{-1}{\rm id}$ and \eqref{e:CondCurvature2}}\label{S6.5}
Note that \eqref{e:CondCurvNov} with $\Lambda = \kappa^{-1}{\rm id}$ means 
$$
\frac1{\kappa^2} {\rm id}\le \nabla^2\big( \frac12 \|\cdot\|_K^2 \big),
\;\;\; 
{\rm id} \le \nabla^2\big( \frac12 \|\cdot\|_{K^\circ}^2 \big)
$$
holds on $\mathbb{R}^n\setminus\{0\}$ since $\nabla^2 \big( \frac12 \|\cdot\|_K^2 \big)$ is 0-homogeneous. 
Since $\|\cdot\|_K$ is the Minkowski norm, $K$ is strictly convex and superlinear in the sense that $\|x\|_K^2/|x| \to \infty$ as $|x| \to \infty$. Hence from the general theory of convex function, we have 
$$
\nabla^2 \big( \frac12 \|\cdot\|_K^2 \big)(y_x)
\nabla^2 \big( \frac12 \|\cdot\|_{K^\circ}^2\big)(x) 
=
{\rm id},
\;\;\;
y_x \coloneqq \nabla\big(\frac12 \|\cdot\|_{K^\circ}^2\big)(x) \in \mathbb{R}^n\setminus\{0\} 
$$
for $x \in\mathbb{R}^n\setminus\{0\}$ and $x\mapsto y_x$ is bijective. 
Thus the assumption 
$
{\rm id}
\le 
\nabla^2 
\big(
\frac12 
\|\cdot\|_{K^\circ}^2
\big)
$
can be read as 
$
\nabla^2 
\big(
\frac12 
\|\cdot\|_{K}^2
\big)
\le 
{\rm id}. 
$
This shows the equivalence to \eqref{e:CondCurvature2}.

\section*{Acknowledgements}
This work was supported by JSPS Kakenhi grant numbers 19K03546, 19H01796 and 21K13806 (Nakamura), and JST,  ACT-X Grant Number JPMJAX200J, Japan, and JSPS Kakenhi grant number 22J10002 (Tsuji). 
Authors would like to thank to Neal Bez for sharing his insight which leads us to this work. The second author also expresses his gratitude to his supervisor, Shin-ichi Ohta for some helpful comments. 
This work is a part of the second author's PhD thesis. 

%\section*{Data Availability Statements} 
%Data sharing not applicable to this article as no datasets were generated or analysed during the current study. 

\end{document}